\documentclass[11pt]{amsart}
\usepackage[T1]{fontenc}

\usepackage{lmodern, amsfonts,amsmath,amstext,amsbsy,amssymb,
amsopn,amsthm,upref,eucal,mathptmx}

\usepackage{mathrsfs}

\RequirePackage{xcolor} % [dvipsnames]
\definecolor{halfgray}
{gray}{0.55}%chapter numbers will be semi
\definecolor{webgreen}
{rgb}{0,0.4,0}
\definecolor{webbrown}
{rgb}{.8,0.1,0.1}
\definecolor{red}
{rgb}{1,0,0}
\usepackage{microtype}

\newcommand \R {{ \mathbb R}}
\def\C{{\mathbb C}}
\newcommand \Z {{ \mathbb Z}}
\newcommand \N {{ \mathbb N}}
\newcommand \T {{ \mathbb T}}

\newcommand \re {{%
\operatorname{Re}
}}
\newcommand \im {{%
\operatorname{Im}
}}

\newcommand{\<}{{\langle}} 
\renewcommand{\>}{{\rangle}}

\newtheorem{theorem}{Theorem}[section]
\newtheorem {lemma} [theorem]{Lemma}

\newtheorem{corollary}[theorem]{Corollary}
\newtheorem{remark}[theorem]{Remark}

\newtheorem{definition}[theorem]{Definition}

\title[Twisted Translation Flows and Effective Weak Mixing ]%
{Twisted Translation Flows \\ and Effective Weak Mixing}

  \author{Giovanni Forni}

\address{Department  of Mathematics\\
  University of Maryland \\
  College Park, MD USA}
  
\email
    {gforni@math.umd.edu}

\date{\today}

\begin{document}

\def\echo#1{\relax}
    
\begin{abstract}
  \begin{sloppypar}
  We introduce a twisted cohomology cocycle over the Teichm\"uller flow and prove a ``spectral
  gap'' for its Lyapunov spectrum with respect to the Masur--Veech measures. We then derive H\"older
  estimates on spectral measures and bounds on the speed of weak mixing for almost all translation flows
  in every stratum of Abelian differentials on Riemann surfaces, as well as bounds on the deviation of ergodic averages for product translation flows on the product of a translation surface with a circle.
   \end{sloppypar}
\end{abstract}

\date{\today}

 \maketitle
 \section{Introduction}

 Every holomorphic Abelian differential $h$ on a Riemann surface $M$ induces a {\it translation structure} on the surface $M$, that is, an equivalence class of atlases
  on $M\setminus\{h=0\}$ whose changes of coordinates are given by translations on $\R^2$. 
  
  It follows that the horizontal and vertical directions, as well as all directions, are well
  defined on the translation surface  $(M, h)$ and we can consider horizontal, vertical, or in general directional flows with normalized unit speed.  
  
  Such {\it translation flows} are only defined on the complement finitely many one-dimensional immersed sub-manifolds given by all trajectories  which end up in the finite set $\Sigma_h=\{h=0\}$ either forward or
  backward. However, since their domain of definitions has full area the ergodic theory of translation flows makes perfect sense and has been well-studied. 
 
 \smallskip
 We state below the main results in the ergodic theory of translation flows. 
 
 Let $\mathcal H(\kappa)$  denote the stratum of the moduli space of Abelian differentials of unit total area  with zeros of multiplicities~$\kappa:=(k_1, \dots, k_\sigma)$ with $\sum_{i=1}^\sigma k_i= 2g-2$.  There exists a natural action of the group $SL(2, \R)$ on $\mathcal H(\kappa)$ given by post-composition of translation charts with elements
 of the group (as linear maps on $\R^2$). 
 
 Each stratum $\mathcal H(\kappa)$ is endowed with a unique absolutely continuous $SL(2, \R)$-invariant probability measure $\mu_\kappa$, called a {\it Masur--Veech} measure.
 Since Masur--Veech measure are $SL(2, \R)$-invariant, hence rotation invariant, results for translation  flows, for the ``typical'' Abelian differential (with respect to a Masur--Veech measure), in the ``typical'' direction (with respect to the Lebesgue measure on the circle)  can equivalently be stated for the horizontal (or the vertical) translation flow alone.
 
 The following foundational theorem was proved independently by  H.~Masur \cite{Ma82} and  W.~Veech \cite{Ve82}.

 \begin{theorem}  \cite{Ma82}, \cite{Ve82} For almost all Abelian differential $h \in \mathcal H(\kappa)$ with respect to the Masur--Veech measure, the horizontal translation flow is uniquely ergodic, hence the directional translation flow in almost all directions $\theta \in \T$ with respect to the Lebesgue measure is also uniquely ergodic.
 \end{theorem} 

 The above unique-ergodicity theorem was later refined by S.~Kerckhoff, H.~Masur and J.~Smillie in \cite{KMS86}:
 
 \begin{theorem} \cite{KMS86} For all $h \in \mathcal H(\kappa)$,  the  directional translation flow  is uniquely ergodic in almost all directions $\theta \in \T$
 with respect to the Lebesgue measure.
 In particular, for almost all $h\in \mathcal H(\kappa)$ with respect to any $SL(2, \R)$-invariant probability measure, the horizontal translation flow is uniquely ergodic.
  \end{theorem}
  
  It is known since the work of A.~Katok \cite{Ka80} that interval exchange transformations (IET's) and translation flows are never mixing. However, it was conjectured that the typical  IET and translation flow are weakly mixing. After partial results of several authors (see \cite{KS67}, \cite{Ve84}, \cite{GK88}, \cite{NR97}, \cite{Lu98}), the conjecture was proved by A.~Avila and the author: 
  
  \begin{theorem} \cite{AvF07}
  For almost all Abelian differential $h \in \mathcal H(\kappa)$, a stratum  of higher genus surfaces, with respect to the Masur--Veech measure $\mu_\kappa$ on 
  $\mathcal H(\kappa)$, the horizontal translation flow is weakly mixing.
  \end{theorem}
  
 As a consequence of this theorem and of the above-mentioned result of A.~Katok,  typical IET's and  translation flows are perhaps the simplest natural example of weakly mixing dynamical systems which are not mixing
 (the first examples, starting with the Chacon map~\cite{C69},  were constructed by cutting-and-stacking). 
 These examples are perhaps not surprising in view of the Halmos-Rokhlin Theorem which asserts that (with respect to the weak topology) weak mixing is a generic property, while mixing is meager. 
  
  \medskip
  An effective version of Kerckhoff--Masur--Smillie  unique ergodicity theorem, establishing a polynomial (power-law) speed of convergence of ergodic averages, was later proved by J.~Athreya and the author.  
  
  For every holomorphic Abelian differential $h$ on $M$,  let $(\phi^\theta_t)$ denote the directional translation flow on~$M$ in the direction $\theta\in \T$, that is, a the horizontal translation flow of the Abelian differential $e^{2\pi \imath \theta} h$. Let $\omega_h$ denote the area-form of the translation surface $(M,h)$, which is invariant under all translation flows $(\phi^\theta_t)$ on $(M,h)$. 
  
  Let $H^1(M)$ denote the Sobolev space of square-integrable functions with square-integrable weak first derivative of the compact surface $M$, and for any function $f \in H^1(M)$ let $\vert f \vert_1$ denote its Sobolev norm in the space $H^1(M)$.
  
 \begin{theorem} \cite{AtF08}
 \label{thm:AtF08}
 There exists a real number $\alpha_\kappa >0$ and, for all $h \in \mathcal H(\kappa)$, there is a measurable function $C_h: \T \to \R^+$ such that for Lebesgue almost all $\theta \in \T$,  for all functions $f\in H^1(M)$ and for all $(x,\mathcal T) \in M\times \R^+$, we have
 $$
\vert  \int_0^{\mathcal T}  f \circ \phi^S_t (x) dt - {\mathcal T}\int_M f d\omega_h \vert \leq C_h(\theta) \vert f\vert_1 {\mathcal T}^{1-\alpha_\kappa}\,.
 $$
  \end{theorem}  
A more complete picture of the finer behavior of ergodic integrals for almost all translation flows, 
which include lower bounds of the ergodic integrals along subsequences of times for almost all $x\in M$ was proposed conjecturally in the work of A. Zorich~and M.~Kontsevich \cite{Zo97}, \cite{Ko97}. A proof of a substantial part of the Kontsevich--Zorich conjectures was given by the author in~\cite{F02}, and later completed 
by the result of A.~Avila and M.~Viana~\cite{AV07} who proved the simplicity of the 
Kontsevich--Zorich spectrum.  

\bigskip 
In this paper we prove effective unique ergodicity results for typical product translation flows on the product translation $3$-manifold $M\times \T$, analogous to the above mentioned result by Athreya and the author. 
It is a standard result of ergodic theory that the ergodicity of a product of ergodic flows follows from
 the weak mixing property of one of the factors.

Let $\Phi^{S,\lambda}_t$ denote the flow  $\phi^S_t \times R^\lambda_t$ on $M\times \T$, product of the translation flow $(\phi^S_t)$ and of the linear flow with speed $\lambda \in \R\setminus \{0\}$ on $\T$, which is generated by the vector field $S + \lambda \frac{\partial}{\partial \theta}$ on $M\times \T$. 
  
 We recall that, by basic ergodic theory, the product flow  $\Phi^{S,\lambda}_t$ is ergodic whenever the flow~$\phi^S_t$ is weakly mixing. The latter property holds for almost all holomorphic differential $h$ in any stratum of the moduli space by the result of A.~Avila and the author~\cite{AvF07}. In fact, it is proved in~\cite{AvF07}, by a ``linear elimination'' argument and by a weak mixing criterion of Veech~\cite{Ve84}, that the set of holomorphic differentials with non weakly mixing horizontal translation flows has Hausdorff codimension  (at least) $g-1$ in every stratum of translation surfaces of genus $g\geq 2$.  A well-known argument by Furstenberg implies that every ergodic product flow $\Phi^{S,\lambda}_t$ such that $\phi^S_t$ is uniquely ergodic is also uniquely ergodic. 
 
 \smallskip
Our goal is to prove the following results. For any $s>0$, let $H^s(\T, H^1(M))$ denote the 
Sobolev space of square-integrable functions with square integrable first derivatives in the directions
tangent to $M$ and square-integrable derivatives up to order $s>0$ in the circle direction.

\begin{theorem} \label{thm:Effect_Erg} There exists  a real number  $\alpha''_\kappa >0$ and, for almost all Abelian differentials $h\in \mathcal H(\kappa)$ with respect to the 
Masur--Veech measure and  for all $\lambda \in \R\setminus \{0\}$,  there exists a constant $C_{\lambda}(h) >0$ such that, for all functions $F\in H^s(\T, H^1(M))$, with $s>s_\kappa$ (for some $s_\kappa>1$), and for all $(x,\theta, {\mathcal T}) \in M\times \T \times \R^+$, we have
 $$
\vert  \int_0^{\mathcal T}  F\circ \Phi^{S, \lambda}_t (x,\theta) dt - {\mathcal T}\int_{M\times \T} F d\omega_h d\theta 
\vert \leq C_{ \lambda}(h) \Vert F \Vert_{ H^s(\T, H^1(M)) }  {\mathcal T}^{1-\alpha''_\kappa}\,.
 $$
 \end{theorem}
 \begin{remark} It follows from the argument that the exponent power saving $\alpha''_\kappa$ in Theorem  \ref{thm:Effect_Erg} can be taken to be the minimum of the 
 exponent power savings $\alpha_\kappa$ of Theorem~\ref{thm:AtF08} and $\alpha'_\kappa$ of Theorem~\ref{thm:Twist_Erg}  below.
  \end{remark}

\begin{remark} The higher differentiability assumption in the above Theorem~\ref{thm:Effect_Erg}, as well as in Corollary~\ref{cor:Eff_WM} below, follows from the fact that 
the proofs of these results require better estimates, with respect to the phase parameter $\lambda \in \R$, than those of Theorem~\ref{thm:twist_integral_bound} below. 
These sharper bounds are derived from those of Theorem~\ref{thm:twist_integral_bound} by sufficiently many integration by parts (with respect to time) to counter the
possible polynomial blow up of the constant with respect to the phase parameter. As a consequence, the differentiability threshold $s_\kappa>1$ is at least as large as the 
difference $N_\kappa-\beta_\kappa$ of the constants of Theorem~\ref{thm:twist_integral_bound}  (see \S~\ref{sec:Proofs} for details).
\end{remark}  
 
 A.~Bufetov and B.~Solomyak  \cite{BS18b} have derived  from uniform estimates on twisted ergodic integrals for suspension flows over substitution systems (or a self-similar translation flow) an interesting result on the speed of ergodicity for  ergodic flows which are product of such a flow with a general ergodic flow. Their result is a generalization of the above theorem (since the twisted flow is defined as a product with a rotation flow on a circle). We do not know whether it is possible to generalize their result to almost all translation flows, or equivalently, our result above to general ergodic transformations. 

\smallskip
The above theorem is derived from the following effective result on twisted ergodic integrals for translation flows:

 \begin{theorem} \label{thm:Twist_Erg} There exist real numbers $\alpha_\kappa' >0$, $\beta_\kappa>0$ and 
 $N_\kappa>0$  and, for almost all Abelian differentials $h\in \mathcal H(\kappa)$ with respect to the Masur--Veech measure, there exists a constant $C_\kappa(h) >0$ such that, for all $\lambda \in \R\setminus \{0\}$, for all zero average functions $f\in  H^1(M)$ and for all $(x,\mathcal T) \in M \times \R^+$, we have
 $$
\vert  \int_0^{\mathcal T} e^{2\pi \imath \lambda t}  f\circ \phi^{S}_t (x) dt  
\vert \leq C_\kappa(h)  \frac{(1+  \lambda^2)^{\frac{N_\kappa}{2}}}{\vert \lambda\vert^{\beta_\kappa}}  \vert f \vert_{ H^1(M) } {\mathcal T}^{1-\alpha'_\kappa}\,.
 $$
 \end{theorem}  
 We remark that Theorem~\ref{thm:Effect_Erg} and Theorem~\ref{thm:Twist_Erg} are in fact almost equivalent.
 In Theorem~\ref{thm:Twist_Erg} we have additional control on the twisted integral for small frequencies, which
 is important in the proof of the effective weak mixing result stated below. 
 In the paper we prove Theorem~\ref{thm:Twist_Erg} and derive Theorem~\ref{thm:Effect_Erg} from it.
 In case of self-similar translation flows (related to substitutions)  and for the Masur--Veech measures 
 on the strata ${\mathcal H}(2)$ and ${\mathcal H}(1,1)$ in genus $2$ this result has been proved by A.~Bufetov and B.~Solomyak~\cite{BS14}, \cite{BS18a}, \cite{BS18c}.   After our paper was completed\footnote{A complete draft of the present paper was sent  by the author to B.~Solomyak on May 22, 2019.}, Bufetov and Solomyak~\cite{BS19} were able to extend their symbolic approach, based on a twisted version of the Rauzy--Veech cocycle, to all genera (and to all $SL(2, \mathbb R)$-invariant orbifolds of rank higher than one),  drawing  in part on our refinement of the key ``linear elimination'' argument of~\cite{AvF07}, Appendix A.

 A similar result on twisted integrals of horocycle flows was proved by L.~Flaminio, the author and J.~Tanis~\cite{FFT16}, improving on earlier result by A.~Venkatesh \cite{V10} and J.~Tanis and P.~Vishe~\cite{TV15}. Twisted ergodic integrals of nilflows are ergodic integrals of product nilflows, hence they are covered by results on deviation of ergodic averages of nilflows. The Heisenberg (and the general step $2$) nilflow case are better understood, by renormalization methods (see for instance \cite{FlaFo06}), while the higher step case is not renormalizable, hence harder (see for instance \cite{GT12}, \cite{FlaFo14}).  We remark that  the nilpotency class is unchanged by taking the product of a nilmanifold with a circle.
 
 \smallskip
 Theorem~\ref{thm:Twist_Erg} is related to H\"older estimates on spectral measures.
 In particular we derive the following result.
 
  \begin{corollary}  \label{cor:spectral}
  There exist real numbers $\alpha_\kappa' >0$, $\beta_\kappa>0$ and  $N_\kappa>0$ and, for almost all Abelian differentials $h\in \mathcal H(\kappa)$ with respect to the 
  Masur--Veech measure, there exists a constant $C_h>0$ such that the spectral measure $\sigma_f$ of any function $f \in H^1(M)$  satisfies the bound
 $$
 \sigma_f ([\lambda-r, \lambda +r]) \leq  C_h \frac{(1+ \vert \lambda\vert)^{N_\kappa}}{ \vert \lambda\vert^{\beta_\kappa} }  \vert f \vert_{H^1(M)}  r^{2\alpha'_\kappa} \,, \quad  \text { for all }
 \lambda \in \R \text{ and }  r >0\,.
 $$
 In particular, the lower local dimension $\underline{d}_f(\lambda)$ of the spectral measure $\sigma_f$ satisfies  the inequality 
  $$
  \underline{d}_f( \lambda):= \underline{\lim}_{r\to 0^+}  \frac{\log \sigma_f ([\lambda-r, \lambda +r]) }{\log r}   \,\, \geq \,\, 2 \alpha'_\kappa \,, \quad \text{ for all } \lambda \in \R\,.
  $$
\end{corollary} 
 \begin{remark} The  positive exponents  $\alpha_\kappa'0$, $\beta_\kappa$ and  $N_\kappa$  of Corollary~\ref{cor:spectral} are indeed the same as the exponents 
 of Theorem~\ref{thm:Twist_Erg} above.
  \end{remark}

Finally, {\it uniform}  H\"older estimates on spectral measures are known to imply power-law quantitative weak mixing estimates (see for instance ~\cite{Kn98}, Corollary~3.8). However we do not know {\it a priori}  whether uniform 
H\"older estimates on spectral measures hold for almost all translation flows.  

We are nevertheless able to
derive the following effective weak mixing result directly from the bounds on twisted integrals of~Theorem~\ref{thm:Twist_Erg}.   

\begin{corollary}  \label{cor:Eff_WM}
 There exist a real number $\alpha'''_\kappa>0$ and, for almost all Abelian differentials $h\in \mathcal H(\kappa)$ with respect to the Masur--Veech measure, there exists a constant $C_h>0$ such that, for any zero-average functions  $f  \in H^{s}(M)$ with $s>s'_\kappa$ (for some $s'_\kappa>1$) and $g\in L^2_h(M)$,  and for all $\mathcal T>0$ we have
 $$
 \frac{1}{{\mathcal T}} \int_0^{\mathcal T}  \left\vert \< f\circ \phi^S_t, g\>_{L^2_h(M)} \right\vert^2  dt  \leq C_h \vert f \vert^2_{H^s(M)} \vert g \vert^2_{L^2_h(M)}  
 {\mathcal T}^{-\alpha'''_\kappa}\,.
 $$
 \end{corollary}
 
 From the effective weak mixing result, we can then derive {\it a posteriori} uniform estimates on twisted ergodic integrals of sufficiently differentiable functions (see Corollary~\ref{cor:Twist_Erg}), which in turn imply uniform H\"older estimates on spectral measures, following a suggestion 
of O.~Khalil.   In fact, by following a general argument of A.~Venkatesh (see Lemma~3.1 in~\cite{V10}),  for sufficiently smooth functions  it is possible to upgrade the upper bound on twisted ergodic integrals of Theorem~\ref{thm:Twist_Erg} to an upper bound uniform with respect to  the phase constant $\lambda\in \R$,  as a consequence of the effective ergodicity result of Theorem~\ref{thm:AtF08} and the effective weak mixing result of Corollary~\ref {cor:Eff_WM}.  We can therefore state (without proof) the following result:
 
\begin{corollary}  \label{cor:Twist_Erg}
 There exist a real number $\alpha^{(iv)}_\kappa>0$ and, for almost all Abelian differentials $h\in \mathcal H(\kappa)$ with respect to the Masur--Veech measure, there exists a constant $C_h>0$ such that,  for all $\lambda \in \R\setminus \{0\}$, for all zero-average functions  $f  \in H^{s}(M)$ with $s>s_\kappa$ (for some $s_\kappa>1$)  and for all $(x,\mathcal T) \in M \times \R^+$, we have
 $$
\vert  \int_0^{\mathcal T} e^{2\pi \imath \lambda t}  f\circ \phi^{S}_t (x) dt  
\vert \leq C_\kappa(h)  \vert f \vert_{ H^s(M) } {\mathcal T}^{1-\alpha^{(iv)}_\kappa}\,.
 $$

 \end{corollary}

 \begin{remark}  \label{remark} It is not difficult to extend all of the above results to almost everywhere statements with respect to absolutely continuous $SL(2, \R)$-invariant measures on any $SL(2, \R)$-invariant orbifold $\mathcal M$ of rank at least $2$. In fact, the ``linear elimination'' argument of section~\ref{sec:Tor_KZ}, which is a strengthened version of the argument given in the Appendix of \cite{AvF07}, is based on the condition that  the restriction of the Kontsevich--Zorich cocycle to the projection $p(T\mathcal M)$ of the tangent space $T\mathcal M$ has at least $2$ strictly positive exponents. It is known from the work of S.~Filip (see \cite{Fi17} , Corollary 1.3) that in fact all the Kontsevich--Zorich exponents on $p(T\mathcal M)$ are non-zero (this conclusion can also be derived from the cylinder deformation theorem of A.~Wright (see Theorem 1.10  of \cite{Wri15}) and the criterion of \cite{F11}). Since the (cylinder) rank $r$ of  $\mathcal M$ is by definition (see \cite{Wri15}, Definition 1.11) equal to half of the complex dimension of  $p(T\mathcal M)$ the conclusion follows.  In particular, the Hausdorff dimension bound of Lemma \ref{lemma:growth_2} holds for any sub-orbifold $\mathcal M$ with the genus $g\geq 2$ replaced by the rank $r\geq 1$, and  for rank at least $2$  it  follows that the results hold almost everywhere on $\mathcal M$ since  we have $r+1 < 2r =\text{\rm dim}_\C (T\mathcal M)$. 
 \end{remark}

\smallskip
The paper is organized as follows. We recall definitions and basic facts about translation surfaces and flows
in section~\ref{sec:TFlows}. In section~\ref{sec:Twist_Int} we establish relations between twisted integrals of translation flows and ergodic integrals of the twisted flow on the product $3$-dimensional translation manifold, and we describe them in terms of $1$-dimensional (closed) currents. In section~\ref{sec:Twist_Co} we introduce the twisted cohomology space and the twisted cocycle over the Teichm\"uller flow, which is in fact a cocycle over the toral quotient of the Kontsevich--Zorich cocycle. The core of our approach comes in section~\ref{sec:Var_For} 
where we prove a first variation formula for the Hodge norm of the twisted cocycle.  In section~\ref{sec:Tor_KZ} we prove a result about a generalized weak stable space of the toral Kontsevich--Zorich cocycle, inspired by the
``linear elimination'' argument of~\cite{AvF07}.  Finally, in section \ref{sec:Proofs} we prove all the main results and corollaries stated above in this Introduction (with the only exception of Corollary~\ref{cor:Twist_Erg} whose proof we leave to the reader). Rather standard facts on the relations between bounds on twisted ergodic integrals, local dimension of spectral measures and effective weak mixing are postponed to section~\ref{sec:spectral_dim} at the end of the paper.

\section*{Acknowledgements} The motivation for this work came from conversations with Corinna Ulcigrai and
Selim Ghazouani about Bufetov's and Solomyak's work~\cite{BS14},~\cite{BS18a} on twisted cohomology and twisted ergodic integrals of translation flows.  We wish to thank Simion Filip and Carlos Matheus for  enlightening conversations which inspired the final form of the twisted cocycle. We are grateful to W.~Goldman for explaining to us his argument from~\cite{G84} given in our proof of Lemma~\ref{lemma:cohom_dim}. Finally, Jenny Rustad read a draft of the paper and with her criticism helped improve the exposition and correct mistakes.  This research was supported by the NSF grant DMS 1600687 and by a Research Chair of  the Fondation Sciences Math\'ematiques de Paris (FSMP). The author is grateful to the Institut Math\'ematiques de Jussieu (IMJ)
for its hospitality.

 \section{Translation Surfaces and Flows}
 \label{sec:TFlows}
 
Let $\Sigma_{h}:=\{p_{1},\dots,p_{\sigma}\}\subset M_h$ be the set of zeros of the holomorphic Abelian differential~$h$ on a Riemann surface $M$, of orders $(k_{1},\dots,k_{\sigma}) \in (\N \setminus \{0\})^\sigma$ respectively with~$k_{1} + \dots + k_{\sigma}=2g-2$.  Let  $R_{h}:=\vert h\vert$ be the flat metric with cone singularities at $\Sigma_{h}$ induced by the Abelian  differential $h$ on $M$ and let $\omega_h$ denote its area form.  With respect to a holomorphic local coordinate $z=x+\imath y$ at a regular point, the Abelian differential 
$h$ has the form $h=\phi(z)dz$, where $\phi$ is a locally defined holomorphic function, and, consequently, 
\begin{equation}
\label{eq:metric}
 R_h= |\phi(z)| (dx^2 +dy^2)^{1/2}\,,\quad  \omega_h=|\phi(z)|^2\,dx\wedge dy\,.
\end{equation}
\noindent The metric $R_{h}$ is flat, degenerate at the finite set $\Sigma_{h}$ of zeroes of $h$ 
and has trivial holonomy, hence $h$ induces on $M$ the structure of a {\it translation 
surface}. 

 \smallskip
  \noindent The weighted $L^{2}$ space is the standard space $L^{2}_{h}(M):= L^{2}(M,\omega_{h})$  with respect to the area element $\omega_{h}$  of the metric $R_{h}$. Hence the weighted $L^{2}$  norm $\vert \cdot\vert_{0}$ (the dependence on the Abelian differential is suppressed in the notation) is induced by the hermitian product $\<\cdot, \cdot\>_{h}$ defined as follows: for all functions $u$,$v\in L^{2}_{h}(M)$,
 \begin{equation}
 \label{eq:0norm}
 \< u ,v\>_{h} :=  \int _{M} u\,\bar v \, \omega_{h}\,\,.
 \end{equation}
 Let $\mathcal  F_{ \im (h)}$ be the {\it horizontal foliation},  $\mathcal  F_{\re (h)}$ be the {\it vertical foliation} for
the holomorphic Abelian differential $h$ on $M$. The foliations $\mathcal  F_{\im (h)}$ and $\mathcal  F_{\re(h)}$
are measured foliations (in the sense of Thurston):  $\mathcal  F_{\im(h)}$ is the foliation given by the equation $\im (h)=0$ endowed with the invariant transverse measure 
$\vert \im (h) \vert$,  $\mathcal  F_{\re(h)}$ is the foliation given  by the equation $\re (h)=0$ endowed with the 
invariant transverse measure $\vert \re (h) \vert$.  Since the metric $R_h$ is flat with trivial holonomy,  there exist  commuting vector fields $S_h$
and $T_h$ on $M\setminus \Sigma_{h}$ such that 
\begin{enumerate}
\item The frame $\{S_h,T_h\}$ is a parallel  orthonormal frame with respect to the metric $R_{h}$ for the restriction of the tangent bundle $TM$ to the complement $M\setminus \Sigma_{h}$  of the set of cone points;
\item the vector field $S_{h}$ is tangent to the horizontal foliation $\mathcal  F_{\im(h)}$, the vector field $T_{h}$
 is tangent to the vertical foliation $\mathcal  F_{\re(h)}$ on $M\setminus \Sigma_{h}$ \cite{F97}, \cite{F07}. 
 \end{enumerate}
 In the following  we will often drop the dependence of the vector fields $S_{h}$, $T_{h}$ on the Abelian differential in order to simplify the notation. The symbols
$\mathcal  L _{S}$, $\mathcal  L _{T}$ denote the Lie derivatives, and $\imath_S$, $\imath_T$ the contraction operators with respect to the vector field  $S$, $T$ on $M\setminus \Sigma_h$. We have:
\begin{enumerate}
\item $\mathcal  L _{S} \omega_{h} = \mathcal  L_{T}\omega_{h} =0$ on $M\setminus \Sigma_{h}$ , that is, the 
area form $\omega_{h}$ is invariant with respect to the flows generated by $S$ and $T$;
\item $\imath_{S} \omega_{h}= \im (h)$ and $\imath_{T} \omega_{h}= -\re (h)$, hence
the $1$-forms $\eta_{S} :=\imath_{S} \omega_{h}$,  $\eta_{T} :=-\imath_{T} \omega_{h}$ are smooth 
and closed on $M$ and $\omega_{h}= \eta_{T}\wedge \eta_{S}$.
\end{enumerate}
It follows from the area-preserving property $(1)$ that the vector field $S$, $T$ are anti-symmetric
as densely defined operators on $L^{2}_{h}(M)$, that is, for all functions $u$, $v \in C_0^{\infty} (M\setminus\Sigma_h)$,  (see \cite{F97}, $(2.5)$),
\begin{equation}
\label{eq:antisymm}
\< Su ,v\>_{h} = -\< u ,Sv\>_{h}\,\,, \quad \text{ respectively } \,\, \< Tu ,v\>_{h} =-\< u ,Tv\>_{h} \,\,.
\end{equation}
In fact, by Nelson's criterion~\cite{Ne59}, Lemma 3.10, the anti-symmetric operators $S$, $T$ are {\it essentially skew-adjoint} on the Hilbert space $L^{2}_{h}(M)$.

\smallskip
\noindent The {\it weighted Sobolev norms} $\vert \cdot\vert_{k}$, with integer exponent $k>0$, are the euclidean norms, introduced in \cite{F97}, induced by the hermitian product defined as follows: for all 
functions $u$, $v\in L^{2}_{h}(M)$,
\begin{equation}
 \label{eq:knorm}
 \< u,v \>_{k} :=   \frac{1}{2}\sum_{i+j\leq k}\<S^{i}T^{j}u, S^{i}T^{j}v\>_{h} + 
 \<T^{i}S^{j}u, T^{i}S^{j}v\>_{h}\,.
 \end{equation}
  The  {\it weighted Sobolev norms } $\vert \cdot\vert_{-k}$, with integer exponent $-k<0$ are defined to be the dual norms of  the norms $\vert \cdot\vert_{k}$ on the  maximal {\it common invariant domain}
 \begin{equation}
 \label{eq:cid}
 H^{\infty}_h(M):= \bigcap_{i,j\in \N}  \text{\rm Dom}( \bar S^i \bar T^j) \cap \text{\rm Dom}( \bar T^i \bar S^j)
 \end{equation}
 of the closures $\bar S$, $\bar T$ of the essentially skew-adjoint operators $S$,  $T$ on $L^2_h(M)$.
 
The {\it weighted Sobolev space }$H^{k}_h(M)$, with integer exponent $k\in\Z$, is the 
Hilbert space obtained as the completion with respect to the norm $\vert \cdot  \vert _{k}$ of the 
space $  H^{\infty}_h(M)$ endowed with the norm $\vert \cdot\vert_{k}$.
The weighted Sobolev space $H^{-k}_h(M)$ is isomorphic to  the dual space of the Hilbert space $H^{k}_h(M)$, for all $k\in \Z$. 

The weighted Sobolev norms can be extended to differential forms as follows. 
Let  $\Omega^1 H^\infty_h (M)$ denote the space of $1$-forms
$$
\Omega^1 H^\infty_h (M):= \{ \alpha_T \eta_T + \alpha_S \eta_S \vert (\alpha_T, \alpha_S) \in 
H^{\infty}_h(M)^2 \}\,.
$$
Since the space $\Omega^1 H^\infty_h (M)$ is by definition identified with the square $H^{\infty}_h(M)^2$, it
is possible to define, for all~$k\in \N$,  the Sobolev norms $\vert \cdot \vert_k$  on $\Omega^1 H^\infty_h (M)$ 
as follows: for all $\alpha = \alpha_T \eta_T + \alpha_S \eta_S\in \Omega^1 H^\infty_h (M)$ we let
$$
\vert \alpha \vert_{k}  = \left (  \vert \alpha_T \vert^2_k + \vert \alpha_S \vert_k^2 \right)^{1/2} \,.
$$
The {\it weighted Sobolev space }$\Omega^1 H^{k}_h(M)$, with integer exponent $k\in \N$, is the 
Hilbert space obtained as the completion with respect to the norm $\vert \cdot  \vert _{k}$ of the 
space $\Omega^1 H^{\infty}_h(M)$ endowed with the norm $\vert \cdot\vert_{k}$.

The {\it weighted Sobolev space }$\Omega^1 H^{-k}_h(M)$, with negative integer exponent $-k$, is the 
Hilbert space obtained as the completion with respect to the norm $\vert \cdot  \vert _{-k}$ of the 
space $  \Omega^1 H^{\infty}_h(M)$ endowed with the norm $\vert \cdot\vert_{-k}$.  Since $M$ has 
dimension $2$, it is isomorphic to the space of currents of degree $1$ (and dimension $1$) dual to the 
Hilbert space $\Omega^1 H^{k}_h(M)$ of $1$-forms.

The {\it weighted Sobolev space }$\Omega^2 H^{k}_h(M)$, with integer exponent $k\in \N$,
of differential $2$-forms is defined by identification of the Sobolev space of functions  $H^k_h(M)$ with the space of 
$2$-forms given by multiplication times the area form $\omega_h$.  The {\it weighted Sobolev space }$\Omega^2 H^{-k}_h(M)$, 
with negative exponent $-k$, is defined as the dual of the Hilbert space $\Omega^2 H^{k}_h(M)$. It is a space of currents of
degree $0$ and dimension $2$.

Finally, the weighted Sobolev spaces $\Omega^* H^{s}_h(M)$ of differential forms, with arbitrary real exponent $s >0$, can be defined by interpolation,
and the weighted Sobolev space $\Omega^* H^{-s}_h(M)$, with negative real exponent $-s$, is defined as the dual Hilbert space of the
space $\Omega^* H^{s}_h(M)$. It is by definition a space of currents. 
 
 \medskip
 \noindent We recall below some basic results of analysis on translation (flat) surfaces which will be relevant in the following.

 Let $H^{s}(M)$, $s\in \R$, denote a family of standard Sobolev spaces on the compact 
manifold $M$ (defined with respect to a Riemannian metric).  

\begin{lemma}  (\cite{F07}, Lemma 2.11)
\label{lemma:comparison}
For any Abelian differential $h\in {\mathcal H}(\kappa)$, the following continuous embedding and isomorphisms of Banach spaces hold:
\begin{enumerate}
\item $ \,\, H^{s}(M) \,\, \subset   \,\, H_{h}^{s}(M)  \,,  
\quad\text{for }0\leq s<1$;
\item $\,\,H^{s}(M) \,\, \equiv   \,\, H_{h}^{s}(M) \,,
\quad\text{for }s=1$;
\item $\,\,H_{q}^{s}(M) \,\, \subset  \,\,  H^{s}(M)\,, 
\quad\text{for }s >1$.
\end{enumerate}
For $s \in [0,1]$, the space $H^{s}(M)$ is dense in $H_{h}^{s}(M)$ and, for $s >1$, the closure 
of $H_{h}^{s}(M)$ in $H^s(M)$ has finite codimension.   
\end{lemma}

\noindent {\bf Notation} : {\it In the following the symbols $C_\star$, $C'_\star$ or $C(\star)$ will denote positive constants generally independent of the 
Abelian differential and depending on the quantities $\star$.}

\medskip
For any Abelian differential $h\in {\mathcal H}(\kappa)$, let $\delta(h)$ denote the length of the shortest saddle connection on the translation
surface $(M,h)$. 

 \begin{lemma}[Sobolev embedding theorem] 
 \label{lemma:Sob_embed}  For any $s>1$ there exists a constant $C_s>0$ such that the following holds. 
 For any stratum $\mathcal H(\kappa)$, for any Abelian differential $h\in \mathcal H(\kappa)$ and for any function $u\in H^s_h(M)$, we have
 $$
 \max_{x\in M}  \vert u(x) \vert  \leq  \frac{C_s}{\delta(h)}  \vert  u \vert_{H^s_h(M)} \,.
 $$
 \end{lemma} 
 \begin{proof} It is proved in \cite{F07}, Lemma 2.1, that for $s>1$ the weighted Sobolev space $H^s_h(M)$ embeds into the standard Sobolev space $H^s(M)$ of the
 compact surface $M$ (defined by local charts or by any given Riemannian metric). It follows by the Sobolev embedding theorem (see for instance
Th. 5. 4 in \cite{Ad}) that  $H^s_h(M)$ embeds into the space $C^0(M)$ of continuous functions on $M$ endowed with the uniform norm. 
It remains to analyze the dependence of the norm of the embedding on the Abelian differential $h\in {\mathcal H}(\kappa)$.  

For any $x \in M$, there exists an embedded open flat rectangle $R (x) \subset M_h$ with edges of length at least $\delta(h)/4$ such that $x\in R(x)$. 
The Sobolev embedding theorem for an arbitrary flat rectangle can be derived by scaling of the variables  from  the result in the case of a unit square  or 
by Fourier series expansion: let $R_{a,b}= [0, a] \times [0,b]$ denote a closed flat rectangle with edges of length $a, b \in (0, 1)$, and  for all $s\in \R$ let 
$H^s(R_{a,b})$ denote the Sobolev space of the domain $R_{a,b}$.  For any $s>1$, there exists a constant $C'_s>0$ such that, for any function 
$u \in H^s(R_{a,b})$  we have
$$
\max_{x\in R_{a,b}} \vert u(x) \vert   \leq    \frac{C'_s}{(ab)^{1/2} } \vert u \vert_{H^s(R_{a,b})} \,.
$$
The result follows by applying the above statement to the embedded flat rectangle $R(x)$ whose edges have by construction length $a, b \geq \delta(h)/4$.

\end{proof} 
 
 \begin{lemma}[Sobolev trace theorem] 
 \label{lemma:Sob_trace}  For any stratum ${\mathcal H}(\kappa)$ and any $s>1/2$ there exists a constant $C_{\kappa,s}>0$ such that the following holds. 
 For any Abelian differential $h\in \mathcal H(\kappa)$, any regular geodetic segment $\gamma\subset M$
 of finite $R_h$-length $L_h(\gamma)$ defines by integration a current of degree $1$ (and dimension $1$) $\gamma \in \Omega^{1} H^{-s}_h(M)$  such that
 $$
  \vert \gamma \vert_{\Omega^{1} H^{-s}_h(M)}  \leq C_{\kappa,s}(1+  \frac{L_h( \gamma)}{\delta(h)})  \,.
 $$
 \end{lemma} 
 \begin{proof}  The result was proved in \cite{F02}, Lemma~9.2, for the case of a horizontal or vertical regular arc (and for $s=1$). It can be generalized to any regular arc by rotation
 of the Abelian differential (and to any $s >1/2$ by following the argument and invoking the general Sobolev trace theorem for rectangles).  
 
 We outline the argument below.
 
 The regular arc $\gamma$ can be decomposed as union $\gamma = \cup_{i=0}^{N} \gamma_i$ of consecutive sub-arcs such that the following properties hold:
 \begin{enumerate}
 \item the length $L_h( \gamma_i)$ of the arcs $\gamma_i$, with respect to the flat metric $R_h$ induced by the Abelian differential $h$, satisfy the bounds
$$
 \begin{aligned}
  \delta(h)/3 \leq L_h( \gamma_i)  &\leq 2 \delta(h) /3 \,, \text{ for all } i\in \{1, \dots, N-1\} \,,  \\
   L_h( \gamma_0) \,, L_h( \gamma_N)  &\leq  2 \delta(h) /3 ;
 \end{aligned} 
$$
 \item the rectangle $R_i = [0, L_h(\gamma_i)] \times (-\delta(h)/3, \delta(h)/3) \subset \R^2$ embeds isometrically in the flat surface $(M \setminus \Sigma_h, R_h)$, so that the arc $ \bar \gamma_i :=[0, L_h(\gamma_i)] \times \{0\}$ has image equal to $\gamma_i \subset M$, for all $i\in \{0, \dots, N\}$.
 \end{enumerate}
The statement then follows from the Sobolev trace theorem (in $\R^2$) applied to each arc $\bar \gamma_i \subset R_i \subset \R^2$, for $i \in \{0, \dots, N\}$.  
Let 
$$
R_{a,b}:=\{(x,y)\in {\R}^2\,|\,0<x<a\,,\,\,-b<y<b\}\,.
$$ 
By a rescaling argument, that  is, by reducing to the case of $a=b=1$ by an affine change of coordinates, 
and by the Sobolev trace theorem~(see for instance \cite{Ad}, Th. 5.4 (5)), for every $s>1/2$ there exists a constant $K_s >0$,
$$
\Bigm\vert\int_0^a f(x,0)dx\,\Bigm\vert\leq K_s\,\Bigl(\frac{a} {b}\Bigr)^{1/2} 
\max\{a,b,1\}^s\,|f|_{ H^s (R_{a,b})}\,\,,
$$
hence (by taking into account that the systole function is uniformly bounded above on each stratum, see~\cite{MS91}, Corollary 5.6), there exists a constant 
$C_{\kappa,s}>0$ such that 
$$
\vert \gamma_i \vert_{\Omega^{1} H^{-s}_h(M)}  \leq   C_{\kappa,s} \,, \quad \text{ for all } i\in \{0, \dots, N\}\,.
$$
 The estimate in the statement then follows by taking into account the inequality $N-1\leq 3 L_h(\gamma) /\delta(h)$, which is an immediate consequence of the above lower bounds
on the lengths of the sub-arcs $\gamma_i$ for $i\in \{1, \dots, N-1\}$.
 
 \end{proof} 
 
For every $h \in {\mathcal H}(\kappa)$, let ${\mathcal Q}_h: H^1_h(M) \times H^1_h(M) \to \C$ denote the hermitian form
 \begin{equation}
 {\mathcal Q}_h (u, v) := \<S_h u, S_h v \>_h + \<T_h u, T_h v \>_h \,, \quad \text{ for every} u,v \in H^1_h(M).
 \end{equation}

  \section{Twisted Integrals}
  \label{sec:Twist_Int}
  For every holomorphic Abelian differential $h$ on $M$,  let $(\phi^S_t)$ denote the horizontal directional translation flows on $M$, that is, a flow with generator a vector field $S$ on $M\setminus \Sigma_h$.  We are interested in bounds on twisted ergodic integrals for the flow $(\phi^S_t)$, that is, for all $\lambda \in \R$ and for all  $f\in H^1(M)$, the integrals
  $$
  \int_0^{\mathcal T}    e^{2\pi \imath \lambda t}   f\circ \phi^S_t (x) dt  \,, \quad \text{ for all } {\mathcal T}>0\,.
  $$ 
  These integrals can be viewed as ergodic integrals for a product flow as follows. Let $\Phi^{S, \lambda}_t$ denote the (translation) flow with generator the vector field 
  $S_\lambda:= S + \lambda \frac{\partial}{\partial \theta}$ on $M \times \T$, that is, the product flow $(\phi^S_t) \times  (R^\lambda_t)$ of the horizontal translation flow $\phi^S_t$
  times the linear flow $(R^\lambda_t)$ on $\T$.   There is an immediate Fourier decomposition of $L^2(M\times \T)$ into eigenspaces of the circle action on $M\times \T$
  with generator $\Theta:= \frac{\partial}{\partial \theta}$ on $\T$: for all $f \in L^2(M\times \T)$, 
  $$
  f (x, \theta) = \sum_{n\in \Z}   {\bar f }_n(x) e^{2\pi \imath n \theta}\,, \quad \text{ with }   {\bar f}_n (x) := \int_\T  f(x,\theta) e^{-2\pi \imath n \theta} d\theta  \in L^2_h(M)\,.
  $$
  Let $f_n (x, \theta) =  {\bar f }_n(x) e^{2\pi \imath n \theta}$. We have
  $$
  \int_0^{\mathcal T}  f_n \circ \Phi^{S, \lambda}_t (x,\theta) dt =   e^{2\pi \imath n \theta}  \int_0^{\mathcal T} e^{2\pi \imath n \lambda t}  {\bar  f}_n \circ \phi^S_t (x) dt \,.
  $$  
  Ergodic integrals on $M\times \T$ can be extended as  linear functionals on $1$-forms, that is, as currents of dimension $1$ and degree $2$. Since any orbit can
  be decomposed as a union of arcs which can then be closed by the addition of uniformly bounded (transverse) arcs, we are especially interested in {\it closed currents }of
  degree $2$.  
  
  \smallskip
 For any vector bundle $V$ over $M\times\T$, let $\mathcal E (M\times \T, V):= C^\infty(M\times \T, V)$ denote the space of infinitely
 differentiable sections of $V$ over $M\times \T$, and let $\mathcal E' (M\times \T, V)$ denote the dual 
 space of currents. Let $\Omega^2(M\times \T):= C^\infty(M\times \T, \wedge^2 T^\ast (M\times \T))$ be the space of smooth $2$-forms on $M\times \T$.   Since $T^\ast (M\times \T)$ has a splitting 
 $$
 T^\ast (M\times \T) = T^\ast M \oplus  \R \, d\theta
 $$
 (with the natural identification of $T^\ast M$ and $T^\ast \T=\R d\theta$ to subspaces of $T^\ast (M\times \T)$
 via the canonical projections $M\times \T \to M$ and $M\times \T \to \T$), there exists
  a direct splitting of the space $\Omega^2(M\times \T)$ and a dual splitting of the space $\Omega^2(M\times \T)'$ 
  of currents of degree $2$ (and dimension $1$):
  $$
  \begin{aligned}
  \Omega^2(M\times \T) &=  C^\infty (M\times \T,  T^*M) \wedge d\theta \oplus C^\infty (M\times \T, \wedge^2 T^\ast M) \,, \\
    \Omega^2(M\times \T)' &\equiv \mathcal E' (M\times \T,  T^*M)  \oplus   \mathcal E' (M\times \T, \wedge^2 T^\ast M) \,.
  \end{aligned}
  $$

  As a consequence, {\it any} current $C$ of degree $2$ (and dimension $1$) on $M\times \T$ is of the form
  \begin{equation}
  \label{eq:CAB}
  C =   A    +   \imath_\Theta B \,,
  \end{equation}
  with $A\in  \mathcal E' (M\times \T,  T^*M) $ a current of degree $2$ (and dimension $1$),  and 
  $B$ a current of degree $3$ (and dimension $0$), a distribution, on $M\times \T$ (the symbol $\imath_\Theta$
  denotes the contraction operator on currents, with respect to the vector field $\Theta$ on $M\times \T$).
  
  It is also  possible to decompose any current 
  on $M\times \T$ into a sum of Fourier components with respect to the circle action:
  \begin{equation}
  \label{eq:CAN_n}
  C = \sum_{n\in \Z}   C_n =   \sum_{n\in \Z}   A_n   +  \imath_\Theta B_n   \,.
  \end{equation}
  Let $d_M$ denote the exterior derivative on currents on $M$. 
  
  \begin{lemma} \label{lemma:closed} A current $C$ of degree $2$ (and dimension $1$) on $M\times \T$ is closed if and only if $d_MA_0= 0$ and, for all $n\in \Z\setminus\{ 0\}$, 
  $$
  d_M A_n  +2\pi \imath n   B_n =0 \,.
  $$
  \end{lemma}
  \begin{proof} By a straightforward calculation, for any closed current $C$ on $M\times \T$ we have
  $$
  dC =   \sum_{n\in \Z}   d_M A_n     +  \mathcal L_\Theta  B_n =  \sum_{n\in \Z}   d_M A_n     +  2\pi \imath n  B_n  =0\,,
  $$
  hence the statement follows by the orthogonality of the Fourier decomposition.  \end{proof}
  
The $1$-form $\lambda \eta_T - d\theta$ has kernel the vector field $S +   \lambda \Theta= S+ \lambda \frac{\partial}{\partial \theta}$, hence
the current of integration along an orbit of the flow $\Phi^{S, \lambda}_t$ (which is generated by $S +   \lambda \Theta$)
has zero wedge product with~$\lambda \eta_T - d\theta$.

\smallskip  
 Let $K_{h,\lambda}(M\times \T)$ denote the space of all currents of degree $2$ (and dimension $1$)  which have zero wedge product with the $1$-form $\lambda \eta_T - d\theta$.
 \begin{lemma} \label{lemma:kernel} A current $C= A+ \imath_\Theta B$ of degree $2$  (and dimension $1$) on $M\times \T$, as in formula~\eqref{eq:CAB},  belongs to  the subspace $K_{h,\lambda}(M\times \T)$ of currents in the perpendicular of the $1$-form $\lambda \eta_T - d\theta$ if and only if  
 $$
 C = A - \lambda \imath_\Theta (A \wedge \eta_T)\,.
 $$
 \end{lemma} 
 \begin{proof} Since  $C= A  + \imath_\Theta B$ we have
 $$
 C\wedge (\lambda \eta_T - d\theta)=  \lambda A \wedge \eta_T    -\imath_\Theta B \wedge d\theta  =0
 \quad\Leftrightarrow \quad  \imath_\Theta B=-\lambda \imath_\Theta(A\wedge \eta_T)\,.
 $$
 \end{proof}
 Finally we have a characterization of the subspace  of closed currents
 $$ZK_{h,\lambda} (M\times \T) := \mathcal Z (M\times \T) \cap K_{h,\lambda}(M\times\T) \subset K_{h,\lambda}(M\times \T)\,.$$
 \begin {lemma} \label{lemma:kernel_closed} A current $C= \sum_{n\in \Z} A_n + \imath_\Theta B_n$, as in formula~\eqref{eq:CAN_n}, belongs to the subspace $ZK_{h,\lambda} (M\times \T) $ of closed currents in $K_{h,\lambda} (M\times \T)$  if and only if
 $$
 d_M A_n + 2\pi \imath \lambda n \, \eta_T \wedge A_n =0 \,, \quad  \text{ for all } n\in \Z\,.
 $$
  \end{lemma}
 \begin {proof} By Lemma~\ref{lemma:closed} we have 
 $$
 d_M A_n + 2\pi \imath n  B_n =0,  \quad \text{ for all } n\in \Z\,,
 $$
 and by Lemma~\ref{lemma:kernel} 
 $$
 \imath_\Theta  B= -\lambda \imath_\Theta(A \wedge \eta_T)\,, \quad \text{ or, equivalently,} \quad
 B= -\lambda (A \wedge \eta_T) \,,
 $$
 hence for all $n\in \Z$ we have $B_n= -\lambda (A_n \wedge \eta_T)$, so that
 $$
 d_M A_n -  2\pi \imath \lambda n (A_n \wedge \eta_T) = d_M A_n + 2\pi \imath n B_n =0\,.
 $$
 
 \end{proof} 
 
 Bounds on currents in the subspace $ZK_{h,\lambda} (M\times \T)$ are therefore reduced to bounds on currents  of degree~$1$ (and dimension~$1$) 
 on the surface~$M$, which are closed with respect to the  twisted exterior derivatives $d_{h,\lambda}$,  defined as follows:
 $$
d_{h,\lambda} \alpha :=  d_M \alpha +  2\pi \imath  \lambda \eta_T\wedge \alpha\,, \quad \text{ for all } \alpha \in \Omega^1(M)  \,.
 $$
 
 For any  $\lambda \in \R$  and $(x, \theta, {\mathcal T}) \in M\times \T\times  \R$, we can define the current 
 $C_{h,\lambda}(x, \theta, {\mathcal T})$ of degree $2$ (and dimension $1$) on $M\times \T$  as follows: for every $1$-form $\hat \alpha$ on $M\times \T$,
 \begin{equation}
 \label{eq:Ccurrent}
 C_{h,\lambda}(x, \theta, {\mathcal T}) (\hat \alpha) = \int_0^{\mathcal T}   \imath_{S_\lambda} \hat \alpha \circ  \Phi^{S, \lambda} _t (x,\theta) dt 
 \end{equation} 
 (the symbol $\imath_{S_\lambda}$ in the above formula denotes the contraction operator on forms, with respect to the vector field $S_\lambda= S + \lambda \Theta$
 on $M\times \T$).

Since the current $C_{h,\lambda}(x, \theta, {\mathcal T})$ belongs to the subspace $K_{h,\lambda}(M\times \T)$ (that is,  the subspace of all currents of degree $2$ and dimension $1$  which have zero wedge product with the $1$-form $\lambda \eta_T - d\theta$),  by Lemma~\ref{lemma:kernel}  there exists a current  $A_{h,\lambda}(x,\theta, {\mathcal T})$ of degree $2$ (and dimension $1$) such that
 $$
 C_{h,\lambda}(x,\theta, {\mathcal T}) =  A_{h,\lambda}(x,\theta, {\mathcal T})  -  \lambda  \imath_\Theta \left(A_{h,\lambda}(x,\theta, {\mathcal T}) \wedge \eta_T\right)\,.
 $$
 There exists a Fourier decomposition 
 $$
  A_{h,\lambda}(x,\theta, {\mathcal T})  = \sum_{n\in \Z}  e^{-2\pi \imath n \theta}  A^{(n)}_{h,\lambda}(x,{\mathcal T}) \,.
 $$
 
 \begin{lemma}
 For every $n\in \N$, the current $A^{(n)}_{h,\lambda}(x,{\mathcal T}) $ is given, for all $1$-forms $\alpha$ on~$M$, by the formula
  \begin{equation}
 \label{eq:Acurrent}
 A^{(n)}_{h,\lambda}(x,{\mathcal T}) (\alpha) = \int_0^{\mathcal T}  e^{2\pi \imath \lambda n t} \imath_S \alpha \circ \phi^S_t (x) dt 
 \end{equation}
 \end{lemma}
 \begin{proof} For every $1$-form $\alpha$ on $M$, let $\alpha^{(n)} = e^{2\pi \imath n\theta} \alpha$.
 We have
 \begin{equation}
 \begin{aligned}
 A^{(n)}_{h,\lambda}(x,{\mathcal T}) (\alpha) &= e^{-2\pi \imath n\theta}  A^{(n)}_{h,\lambda}(x,{\mathcal T}) (\alpha^{(n)}) \\ &= 
 e^{-2\pi \imath n\theta}  C_{h,\lambda}(x,\theta, {\mathcal T}) (\alpha)
 = \int_0^{\mathcal T} e^{2\pi \imath \lambda n t}  \imath_S  \alpha  \circ \phi^S_t(x) dt \,.
\end{aligned}
 \end{equation}
 \end{proof} 

The analysis is therefore reduced to bounds on currents  of degree $1$ (and dimension $1$) on $M$ of the form
$$
A_{h,\lambda}(x,{\mathcal T}) (\alpha) = \int_0^{\mathcal T}  e^{2\pi \imath \lambda t} \imath_S \alpha \circ \phi^S_t (x) dt 
$$
In fact, for any $1$-form $\alpha$ such that $\imath_S \alpha \in \C$ is constant, we can readily compute
$$
A_{h,\lambda}(x,{\mathcal T}) (\alpha) =  \frac{ e^{2\pi\imath \lambda {\mathcal T}} -1}{2\pi \imath \lambda}  (\imath_S \alpha)
=  A_{h,\lambda}(x,{\mathcal T})(\eta_T) \int_M \alpha \wedge \eta_S  \,.
$$
It follows that it is enough to prove bounds for currents of the form
\begin{equation}
\label{eq:A_diesis}
A^{\#}_{h,\lambda}(x,{\mathcal T})   = A_{h,\lambda}(x,{\mathcal T})   + A_{h,\lambda}(x,{\mathcal T}) (\eta_T) \eta_S \,.
\end{equation}
given for any $1$-form $\alpha \in \Omega^1 H^1(M)$ by the formula
$$
A^{\#}_{h,\lambda}(x,{\mathcal T})(\alpha) =  \int_0^{\mathcal T}  e^{2\pi \imath \lambda t} \imath_S \alpha \circ \phi^S_t (x) dt 
-  \frac{ e^{2\pi\imath \lambda {\mathcal T}} -1}{2\pi \imath \lambda} \int_M \imath_S \alpha \,\omega_h\,.
$$
We estimate below the distance of such currents $A^{\#}_{h,\lambda}(x,{\mathcal T})$ from the subspace of $d_{h,\lambda}$-closed
currents and prove that it is uniformly bounded (even as the phase parameter $\lambda\in \R$ degenerates to zero). The argument will be based  
on a ``twisted''  version of the Poincar\'e inequality, which we now state.  Let $\delta(h)$ denote, as above, the length of the shortest 
saddle connection on the translation surface $(M,h)$. 

\begin{lemma}[Twisted Poincar\'e inequality] \label{lemma:Poincare_twisted}  For every $s>1/2$ and for every stratum ${\mathcal H}(\kappa)$ of Abelian differentials,
 there exists a constant $C_{\kappa,s} >0$ such that, for every $h \in {\mathcal H}(\kappa)$ and for any $\lambda
 \in \R$, we have the following a priori bounds. For  every function $u  \in H^s_h(M)$, we have
$$
\max_{x\in M}  \vert u (x) \vert  \leq  \frac{1}{2\pi \vert \lambda \vert} \vert  d_{h, \lambda} u  \vert_{ \Omega^1 L^2_h(M)}+      \frac{C_{\kappa,s}}{\delta(h)^2}  \vert  d_{h, \lambda} u  \vert_{ \Omega^1 H^s_h(M)}\,,
$$ 
and  we also have
$$
\max_{x\in M}  \vert u (x) -\int_M u \omega_h \vert  \leq     \frac{C_{\kappa,s}}{\delta(h)^2}  \vert  d_{h, \lambda} u  \vert_{ \Omega^1 H^s_h(M)}\,,
$$
\end{lemma} 
\begin{proof} Let $\gamma_{x,y}$ denote any regular geodesic segment on $M$ for the flat metric $R_h$ of the Abelian differential, of length
$L_h(\gamma_{x,y})>0$, 
possibly with endpoints at the set $\Sigma_h$ of conical points, of endpoints $x, y\in M$, oriented from $x$ to $y$. By integrating along the 
path $\gamma_{x,y}$, for any smooth function $u$ on $M$ we have
$$
e^{2\pi \imath \lambda L_h (\gamma_{x,y})} u(y) - u(x) =   \int_{\gamma_{x,y}} d_{h, \lambda} u\,.
$$
For any $x, y\in M$ there exists a piece-wise regular  minimizing geodesic with endpoints  at $x, y$, which is a union of at
most $\text{\rm card}(\Sigma_h) +1$ regular segments. Let $d_h(x,y)$ denote the distance of $x, y\in M$ with respect to the flat metric $R_h$.
By the Sobolev trace theorem (see Lemma~\ref{lemma:Sob_trace}), we then derive that, for any $s>1/2$, there exists a constant $C_{\kappa,s} >0$ 
such that 
\begin{equation}
\label{eq:dtwist_bound}
\vert e^{2\pi \imath \lambda d_h (x,y)} u(y) - u(x)  \vert \leq  C_{\kappa,s} \left (1+ \frac{d_h(x,y)}{\delta(h)} \right)\vert d_{h, \lambda} u \vert_{\Omega^1 H^s(M)} \,.
\end{equation}
By definition of the twisted differential $d_{h, \lambda}$ we also have
\begin{equation}
\label{eq:average}
2\pi \imath \lambda \int_M  u \omega_h  =  \int_M   d_{h, \lambda}u \wedge \im (h) \,, 
\end{equation}
hence there exists $x_{min} \in M$ such that 
$$
\vert u(x_{min}) \vert \leq     \frac{1}{2\pi \vert \lambda \vert }  \vert  d_{h, \lambda}u \vert_{\Omega^1L^2_h(M)} \,.
$$
As a consequence, since on each stratum there exists a constant $C_\kappa>1$ such that $\text{\rm diam}(M, R_h)  \leq C_\kappa/ \delta(h)$, 
we conclude that there exists a constant $C'_{\kappa,s}>0$ such that, for all $x\in M$, we have
\begin{equation}
\label{eq:away_0}
\begin{aligned}
\vert  u(x) \vert &\leq  \vert  e^{2\pi \imath \lambda d_h (x,x_{min})} u(x) - u(x_{min}) \vert   + \vert u(x_{min}) \vert \\ & \leq \frac{1}{2\pi \vert \lambda \vert }  \vert  d_{h, \lambda} u \vert_{\Omega^1L^2_h(M)} + 
\frac{C'_{\kappa,s}}{\delta(h)^2}  \vert d_{h, \lambda} u \vert_{\Omega^1 H^s(M)} \,.
\end{aligned}
\end{equation}
The first inequality in the statement is thus proved. The second inequality follows from the first  if there exists a constant $c_\kappa>0$ such that 
$\vert \lambda \vert \geq  c_\kappa \delta(h)^2$. 

We then consider the case when $\vert \lambda \vert \leq  c_\kappa \delta(h)^2$ for some constant $c_\kappa>0$ to be fixed below. For $\lambda =0$, by integration we have
$$
\vert u(y) - u(x)\vert  = \vert \int_{\gamma_{x,y}} d u \vert  =  \vert \int_{\gamma_{x,y}} d_{h, \lambda}  u \vert  +  2\pi \vert \lambda \vert L_h(\gamma_{x,y}) \max_{x\in M} \vert u(x)\vert\,,
$$
hence, under the hypothesis that $\int_M u\omega_h=0$,  we derive  the estimate
$$
 \max_{x\in M} \vert u(x)  \vert  \leq   \frac{C'_{\kappa,s}}{\delta^2(h)} \vert d_{h, \lambda}  u \vert_{\Omega^1 H^s(M)}  + \frac{2 \pi  C_\kappa}{\delta(h)}  \vert \lambda\vert  \max_{x\in M} \vert u(x)\vert   \,.
$$
Let $c_\kappa>0$ be a constant such that $ 4 \pi  C_\kappa c_\kappa \delta(h) <1$ for all $h\in \mathcal H(\kappa)$. Such a constant exists since the systole function is bounded
above on every stratum. For  $\vert \lambda \vert \leq  c_\kappa \delta(h)^2$, we then have $4 \pi  C_\kappa \vert \lambda\vert < \delta(h)$, hence by bootstrap
$$
\max_{x\in M} \vert u(x) \vert \leq  \frac{2C'_{\kappa,s}}{\delta^2(h)} \vert d_{h, \lambda}  u \vert_{\Omega^1 H^s(M)} \,,
$$
which concludes the argument in the case of zero average functions.  For a general function we have
$$
\max_{x\in M} \vert u(x) - \int_M u\omega_h  \vert \leq  \frac{2C'_{\kappa,s}}{\delta^2(h)} \left ( \vert d_{h, \lambda}  u \vert_{\Omega^1 H^s(M)} 
+  \vert 2 \pi \lambda \int_M u\omega_h \vert  \right)\,.
$$
hence the estimate follows from the identity in formula~\eqref{eq:average}.
\end{proof}

Let $Z^{-1}_{h, \lambda} (M)$ denote the space of $d_{h,\lambda}$-closed $1$-dimensional
currents which belong to the Sobolev space $\Omega^1 H^{-1}_h (M)$ (that is, currents which are continuous functionals on the space of $1$-forms with coefficients in the Sobolev space $H^1_h(M)$, with respect to
the product norm. See section~\ref{sec:TFlows}).
 
\begin{lemma} \label{lemma:dist_closed} There exists a constant $C_\kappa>0$ (depending only on the stratum
${\mathcal H}(\kappa)$)  such that for any Abelian differential $h\in {\mathcal H}(\kappa)$, for any $\lambda \in \R$ and $(x,{\mathcal T}) \in M\times \R$, the current 
$A^{\#}_{h,\lambda}(x,{\mathcal T}) \in \Omega^1 H^{-1}_h (M)$ (defined in formula \eqref{eq:A_diesis}) has uniformly bounded distance from the
closed subspace $Z^{-1}_{h, \lambda} (M)$ of $d_{h,\lambda}$-closed $1$-currents:  
$$
\inf_{Z\in Z^{-1}_{h, \lambda} (M)}  \vert   A^{\#}_{h,\lambda}(x,{\mathcal T}) -  Z\vert_{\Omega^1 H^{-1}_h(M)} \leq  \frac{C_\kappa}{\delta(h)^2} \,.
$$
\end{lemma}  
\begin{proof} We give two arguments. 

\smallskip
{\it First argument}. 
Let  ${E}^1_{h,\lambda} (M)$ be the closure in $\Omega^1 H^1_h(M)$ of the subspace  $d_{h, \lambda} [H^2_h(M)]$ of exact forms 
(in fact, it can be proved that the subspace $d_{h, \lambda} [H^2_h(M)]$ is closed in $\Omega^1 H^1_h(M)$ since the exterior derivative is 
an elliptic operator).  By Hilbert space theory there exists an orthogonal decomposition
$$
\Omega^1 H^1_h(M) = {E}^1_{h, \lambda} (M) \oplus  {E}^1_{h, \lambda}(M)^\perp\,.
$$
Let $C\in  \Omega^1 H^{-1}_h(M)$ be the current defined on ${E}^1_{h, \lambda}(M)$ as
$$
C(\alpha) := A^{\#}_{h,\lambda}(x,{\mathcal T}) (\alpha)\,, \quad    \text{ for all }  \alpha \in {E}^1_{h, \lambda} (M)\,,
$$
extended so that $C \vert  {E}^1_{h,\lambda}(M)^\perp =0$.  By definition we have that
$$
d_{h,\lambda} C = d_{h,\lambda}  A^{\#}_{h,\lambda}(x,{\mathcal T})\,,
$$
hence the current $Z:= C- A^{\#}_{h,\lambda}(x,{\mathcal T})$ is $d_{h,\lambda}$-closed.  We finally estimate the Sobolev norm
of the current $C\in \Omega^1 H^{-1}_h (M)$.   We claim that there exists a constant $C_\kappa>0$ such that
$$
\vert C(\alpha) \vert \leq \frac{C_\kappa}{\delta^2(h)} \vert  \alpha \vert_{\Omega H^1_h(M)}\,, \quad \text{for all }  \alpha \in \Omega H^1_h(M) \,.
$$
 For any  $u\in H^2_h(M)$, let $\bar u:= u - \int_M u \omega_h$. By a direct calculation we have
$$
\begin{aligned}
A^{\#}_{h,\lambda}(x,{\mathcal T}) ( d_{h,\lambda} u) &=  \int_0^{\mathcal T} e^{2\pi \imath \lambda t} (\imath_S d_{h,\lambda} \bar u) \circ \phi^S_t (x) dt
   \\ &= 
\int_0^{\mathcal T}\frac{d}{dt} ( e^{2\pi \imath \lambda t}  \bar u \circ \phi^S_t (x)  ) dt = e^{2\pi \imath \lambda {\mathcal T}} \bar u(\phi^S_{\mathcal T}(x))
-  \bar u(x)\,.
\end{aligned}
$$
hence from the Poincar\'e inequality of Lemma~\ref{lemma:Poincare_twisted}  and from the definition of the current $C$, it follows that
$$
\vert C(d_{h, \lambda} u) \vert = \vert  A^{\#}_{h,\lambda}(x,{\mathcal T}) (d_{h, \lambda} u)  \vert \leq \frac{C_\kappa}{\delta^2(h)} \vert d_{h, \lambda} u  \vert_{\Omega H^1_h(M)}\,,
 \quad \text{for all }  u \in H^2_h(M) \,.
$$
By continuity it follows that 
$$
\vert C(\alpha) \vert \leq  \frac{C_\kappa}{\delta^2(h)} \vert  \alpha \vert_{\Omega H^1_h(M)} \,, \quad \text{ for all } \alpha \in {E}^1_{h,\lambda} (M)\,,
$$
hence the claim follows, since $C$ is extended as zero on ${E}^1_{h,\lambda} (M)^\perp$. The first argument is thus completed.

\smallskip
{\it Second argument}.  For all $s, t  \geq 0$, let  $\Omega^1 H^{s,t}_h(M \times \T)$ denote the Sobolev space 
of  $1$-forms endowed with the following Hilbert norm: for any $1$-form 
$\alpha = \sum_{n\in \Z} e^{2\pi \imath n \theta} \alpha_n$ on $M\times \T$, let 
$$
\Vert  \alpha \Vert_{s,t} :=  \left ( \sum_{n\in \Z} (1+n^2)^{t/2}  \vert \alpha_n \vert^2_{\Omega^1 H^s_h(M)} \right)^{1/2} 
$$
and let $\Omega^1 H^{-s,-t}_h(M \times \T)$ denote the dual space. 

\smallskip
Let $C_{h, \lambda}(x,{\mathcal T})$ be  the current of integration, defined in formula~\eqref{eq:Ccurrent}, along an orbit of the
flow $\Phi^{S,\lambda}_t$ on $M\times \T$. 
It follows by the Sobolev trace theorem (see Lemma~\ref{lemma:Sob_trace})  that the current $C_{h, \lambda}(x,{\mathcal T}) \in \Omega^1 H^{-s, -t}_h(M \times \T)$ 
for $s, t >1/2$.  

By definition of the current $C_{h, \lambda}(x,{\mathcal T})$ there exists a geodesic arc $\gamma$ in $M\times \T$, of length bounded above by the diameter of 
$M\times \T$ with respect to the flat product metric, such that
$ C_{h, \lambda}(x,{\mathcal T}) + \gamma$ is a closed current on $M\times \T$. Let $\bar \gamma$ denote the projection, along the subspace $\mathcal Z^{-s, -t}_h (M\times \T)$
of closed currents, of the current of integration along the arc $\gamma$ on the closed subspace $K^{-s, -t}_{h,\lambda} (M\times \T)$, defined as 
$$
K^{-s, -t}_{h,\lambda} (M\times \T):= K_{h,\lambda} (M\times \T) \cap
\Omega^1 H^{-s, -t}_h(M \times \T)\,.
$$
The current $\bar \gamma$ is given by the following formula:
$$
\bar \gamma  =  \gamma  +   \gamma (\eta_T - \lambda^{-1} d\theta)  \eta_S\,.
$$
In fact, the form $\eta_S$ is closed and by definition 
$$
\bar \gamma (\lambda\eta_T -d\theta)  =0\,, \quad \text{ hence }  \,\, \bar \gamma \in  K_{h,\lambda} (M\times \T) \,.
$$
Since $C_{h, \lambda}(x,{\mathcal T}) \in K^{-s, -t}_{h,\lambda} (M\times \T)$ and $ C_{h, \lambda}(x,{\mathcal T}) + \gamma$ is closed, it follows that 
$$
C_{h, \lambda}(x,{\mathcal T}) + \bar \gamma \in ZK_{h,\lambda} (M\times \T)\,.
$$
It then follows from the definitions that, on $\Omega^1 H^1_h(M)  \subset  \Omega^1 H^1_h(M\times \T)$, 
$$
A^{\#}_{h,\lambda}(x,{\mathcal T})  = C^{\#}_{h, \lambda}(x,{\mathcal T}) :=  C_{h, \lambda}(x,{\mathcal T})  +  C_{h, \lambda}(x,{\mathcal T}) (\eta_T) \eta_S\,.
$$
We can now write
$$
C^{\#}_{h, \lambda}(x,{\mathcal T}) =  (C_{h, \lambda}(x,{\mathcal T}) + \bar \gamma)^{\#} - \bar \gamma^{\#} 
$$
and since the current $(C_{h, \lambda}(x,{\mathcal T}) + \bar \gamma)^{\#}$ is closed, it is enough to prove a bound on the current
$\bar \gamma^{\#}$. By the definition of the current $\bar\gamma$ we have
$$
\bar \gamma^{\#}= \bar \gamma + \bar \gamma (\eta_T) \eta_S =    \gamma +  \gamma (\eta_T) \eta_S = \gamma^{\#}
$$
and from Lemma~\ref{lemma:Sob_trace}  it follows that 
$$
\inf_{z\in Z^{-1}_{h, \lambda} (M)}  \vert   A^{\#}_{h,\lambda}(x,{\mathcal T}) -  z\vert_{\Omega^1 H^{-1}_h(M)} \leq  
\vert   \gamma^{\#}   \vert_{\Omega H^{-1}_h(M)} \leq   \vert \gamma \vert_{\Omega^1 H^{-1}_h(M)}\,.
$$
It follows from the Sobolev trace theorem (see Lemma \ref{lemma:Sob_trace})  and from the bound on the diameter of a translation surface in terms 
of the systolic length that
$$
\vert \gamma \vert_{\Omega H^{-1}_h(M)} \leq \frac{C'_\kappa}{\delta(h)}  \text{\rm diam} (M, h)   \leq \frac{C_\kappa}{\delta^2(h)} \,.
$$
The second argument is therefore completed.

\end{proof}

\section{The twisted cocycle}
\label{sec:Twist_Co}

For any smooth closed $1$-form $\eta$ on $M$, we introduce the twisted differential
\begin{equation}
\label{eq:twisted_diff}
d_\eta:= d + 2\pi \imath \eta \wedge \,,
\end{equation} 
which is a linear operator defined on the space $\Omega^\ast(M)$ of differential forms on $M$, and
maps the subspace $\Omega^k(M)$  of $k$-forms into the subspace  $\Omega^{k+1}(M)$ of $(k+1)$-forms,
for all $k\in \N$.  The twisted differential $d_{h, \lambda}$ introduced above corresponds to a special 
case: 
$$
d_{h, \lambda}  = d_{\eta} \,,   \quad \text{ for }  \eta= \lambda \re (h) \,.
$$
The twisted differential $d_\eta: d + 2\pi\imath \eta \wedge$, introduced in formula~\eqref{eq:twisted_diff},  defines a connection on the trivial bundle $M\times \C$ 
(\cite{We80}, Chap II, \S 1). It is flat since, for all  complex-valued form $\alpha\in \Omega^*(M)$
$$
d_\eta^2 \alpha= (d + 2\pi \imath \eta\wedge) (d\alpha + 2\pi \imath \eta \wedge \alpha) = 
d^2\alpha + 2\pi \imath d\eta \wedge \alpha =0 \,.$$
By the above flatness condition the operators 
$$
d_\eta: \Omega^k (M) \to   \Omega^{k+1}(M)
$$
define a {\it complex}, which is {\it elliptic} since the principal symbols of the twisted differentials
are the same as those of the standard exterior derivative elliptic complex (see \cite{We80}, Chap. IV, \S 2). 
For $k\in \{0, 1, 2\}$, we let $H^k_\eta (M, \C)$ be the corresponding cohomology, which is called {\it twisted 
cohomology}.  The  twisted cohomology $H^\ast_\eta (M, \C)$ in the  particular case when $\eta= \lambda \re (h)$ with be denote 
by $H^\ast_{h,\lambda} (M, \C)$.

\begin{lemma} 
\label{lemma:invariants} The cohomology space $H^0_\eta(M, \C)$ (which is isomorphic to $H^2_\eta(M, \C)$ 
by Poincar\'e duality) is non-trivial if and only if $[\eta] \in H^1(M, \Z) \subset H^1(M, \R)$ and in that case it has complex dimension equal to $1$.
\end{lemma} 
\begin{proof}  Let us assume that there exists a non-zero function $f \in C^\infty(M)$ such that 
$$
d_\eta f= df + 2\pi \imath \eta f =0 \,.
$$
If follows from the above equation that the function $f$ is constant along each leaf of the measured foliation
$\mathcal F_\eta= \{ \eta=0\}$, hence all the non-singular leaves of $\mathcal F_\eta$ are compact. We have
$$
d (f \bar f) =  (df) {\bar f}  + (\overline {df}) f = - 2\pi \imath \eta f \bar f +  2\pi \imath \eta  f \bar f =0\,,
$$
hence there exists $c_f\in \C\setminus \{0\}$ such that $f/c_f: M \to U(1)= \{z \in \C \vert \vert z\vert =1\}$ and there exists a real-valued function $\theta: M\to \R/\Z$ such that
$$
f(x) = c_f \exp \left(-2\pi \imath  \theta(x)\right) \,, \quad \text{ for all } x\in M\,.
$$
By definition we have  $df = -2\pi \imath f d\theta$,  and since by assumption $f\in Z^0_\eta(M, \C)$, the space of $d_\eta$-closed $0$-forms, that is, complex valued functions, 
and $f(x) \not=0$ for all $x\in M$, it follows that $d\theta = \eta$. Since $\theta: M\to \R/\Z$, we conclude that  $\eta \in H^1(M, \Z)$. 

\smallskip
Conversely, let us assume that $[\eta]\in H^1(M, \Z)$. Given any point $p\in M$, the function
$$
f_p (x) = \exp \left(-2\pi \imath \int_p^x \eta \right)  \,, \quad \text{ for all }\, x\in M, 
$$
is a well-defined, non-zero element of $Z^0_\eta(M, \C)$ since
$$
df_p = -2\pi \imath  f_p \eta \,.
$$
In addition, given any $g \in Z^0_\eta(M, \C)$ we have
$$
d (\bar {f_p} g) = (\overline{ df_p}) g + \bar{ f_p } (dg) = 2\pi \imath  {\bar {f_p}} g  \eta  - 2\pi \imath  \bar {f_p} g \eta =0\,,
$$
hence $\bar {f_p} g$ is a constant, which implies that $H^0_\eta(M, \C)$ has dimension equal to $1$.

\end{proof}

Since the complex is elliptic, after endowing the vector spaces $\Omega^k(M)$ of $k$-forms with the 
Hodge $L^2$ hermitian product associated to a holomorphic $1$-form $h$ on the Riemann surface $M$, 
by standard Hodge theory it is possible to represent every cohomology class by a twisted harmonic form.  
In fact, there exists a decomposition
$$
d_\eta =  d^{1,0} +2\pi \imath \eta^{1,0}  +   d^{0,1} + 2\pi \imath \eta^{0,1} 
$$
such that  $d^{1,0}_\eta:= d^{1,0} +2\pi \imath \eta^{1,0}$ and $d^{0,1}_\eta:= d^{0,1} +2\pi \imath \eta^{0,1}$
are maps 
$$
d^{1,0}_\eta :\Omega^{p,q}(M) \to \Omega^{p+1,q}(M) \quad \text{ and } \quad \quad 
d^{0,1}_\eta :\Omega^{p,q}(M) \to \Omega^{p,q+1}(M)\,,
$$
so that by the Hodge-Dolbeault theory (\cite{We80}, Chap. IV, \S 5)
$$
H^1_\eta (M, \C) =  H^{1,0}_\eta (M, \C) \oplus   H^{0,1}_\eta (M, \C) \,.
$$

\begin{lemma}  The twisted cohomology $H^1_\eta (M, \C)$ only depends, up to Hodge unitary equivalence, on the cohomology class $[\eta] \in H^1(M, \R)$ and  in fact only on the equivalence class $[[\eta]] \in H^1(M, \R) / H^1(M,\Z)$. The Hodge unitary equivalence is not unique as it depends on the choice 
of a base point. A change of base point induces a unitary automorphism of the twisted cohomology given by
the multiplication times a constant of unit modulus.
\end{lemma}

\begin{proof}  For any closed smooth $1$-form~$\eta$, let $Z^1_\eta(M, \C)$ and $B^1_\eta (M, \C)$ denote the
kernel of the twisted exterior differential $d_\eta : \Omega^1(M, \C) \to \Omega^2(M, \C)$ and the range
of the  twisted exterior differential $d_\eta : \Omega^0(M, \C) \to \Omega^1(M, \C)$. By definition of 
twisted cohomology we have 
$$
H^1_\eta(M, \C) :=  Z^1_\eta(M,\C) / B^1_\eta(M,\C)\,.
$$
Let $\eta$ and $\eta'$  be closed smooth $1$-forms in the same real cohomology class. By definition there exists a smooth function $f$ on $M$ such that $\eta'-\eta =df$.  

Let $U_f: \Omega^k(M, \C) \to \Omega^k(M, \C)$
denote the linear multiplication operator
$$
U_f (\alpha) =  e^{2\pi \imath f} \alpha\,.
$$
By a direct calculation we have
$$
(d_\eta \circ U_f )(\alpha) = e^{2\pi \imath f} ( d_\eta \alpha  + 2\pi \imath \, df \wedge \alpha) = (U_f \circ d_{ \eta'}) (\alpha)\,.
$$
It follows that the restrictions of $U_f$ to linear operators  $Z^1_{\eta'}(M,\C) \to Z^1_{\eta}(M,\C)$ and $B^1_{\eta'}(M,\C) \to B^1_{\eta}(M,\C)$ are isomorphisms.  In addition, since by definition $U_f$ is an operator of multiplication times a function of constant unit modulus, the projected operator $U_f: H^1_{\eta'}(M,\C)\to H^1_{\eta}(M,\C)$ is unitary with respect to the $L^2$ norm on forms, hence with respect to the Hodge norm.

Similarly, let us assume that $[\eta'-\eta] \in H^1(M, \Z)$. Given $p\in M$, the formula
$$
F_p^{\eta,\eta'} (x) :=   \int_p ^x  \eta'-\eta
$$
gives a well-defined function on $M$ with values in $\R/\Z$ such that $dF_p^{\eta,\eta'} = \eta' -\eta$. 
It follows that the function $\exp (2\pi \imath F_p^{\eta,\eta'})$ is well-defined on $M$. We define the
operator 
$$
U_p^{\eta',\eta}(\alpha) = \exp (2\pi \imath F_p^{\eta,\eta'})  \alpha
$$
and compute that
$$
(d_\eta \circ U_p^{\eta',\eta})(\alpha) = e^{ 2\pi \imath F_p^{\eta,\eta'} } ( d_\eta \alpha  + 2\pi \imath\, d F_p^{\eta,\eta'}  \wedge \alpha) = 
(U_p^{\eta',\eta} \circ d_{ \eta'}) (\alpha)\,.
$$
By the latter formula there is an induced isomorphism, unitary with respect to the Hodge norm,
$$
U_p^{\eta',\eta} :  H^1_{\eta'} (M,\C) \to H^1_{\eta} (M, \C)\,.
$$
Finally a change of the base point induces a unitary isomorphism given by multiplication times a 
constant of unit modulus. 
\end{proof}

\begin{lemma}  \label{lemma:cohom_dim} The dimension of the first twisted cohomology $H^1_\eta(M, \C)$ is given by the following formula:
$$
\text{ \rm dim}_\C \,H^1_\eta (M, \C) = \begin{cases} 2g\,,  \quad  &\text{ if }\, [\eta] \in H^1(M, \Z)\,; \\
2g-2\,,  \quad   &\text{ if }\, [\eta] \not\in H^1(M, \Z)\,.
\end{cases} 
$$
\end{lemma} 
\begin{proof}
The cohomology $H_\eta^1(M, \C)$ is isomorphic to the cohomology of the local system $\mathcal L_\eta$ defined by the representation $\rho_\eta: \pi_1 (M, \ast) \to U(1)$ defined as
$$
\rho_\eta (\gamma ) = \exp (2\pi \imath \int_\gamma \eta ) \,, \quad \text{ for all } \gamma \in  \pi_1 (M, \ast)\,.
$$
In fact, let $p: \hat M \to M$ denote the universal cover. The form $p^*(\eta)$ is closed, hence exact on
$\hat M$, so that there exists a function $F :\hat M \to \R$ such that $p^*(\eta)= dF$. We have that
$$
p^* (d_\eta \alpha ) =  \exp (-2\pi \imath F)  d ( p^*(\alpha) \exp (2\pi \imath F))  \,, \quad \text{ for all }\alpha \in \Omega^*(M)\,.
$$
Let then $\mathcal L_\eta$ denote the local system on $M$ defined as the sub-bundle of the space 
$\Omega^*(\hat M, \C)$ of complex-valued forms $\hat \alpha$ on $\hat M$ such that
$$
\gamma^* (\hat \alpha)  =  \exp (2\pi \imath \int_\gamma \eta ) \hat \alpha \,, \quad \text{ for all } \gamma \in \pi_1 (M, \ast)\,.
$$
The twisted cohomology $H^*_\eta(M, \C)$, defined as the cohomology of the complex of the twisted differential
$d_\eta$ on complex-valued forms $\Omega^*(M,\C)$, is therefore isomorphic to the cohomology $H^*_{\rho_\eta} (M, U(1)):= H^*(M, {\mathcal L}_\eta)$, defined as the cohomology of the complex of the exterior differential 
$d$ on $\mathcal L_\eta$-valued forms $\Omega^*(M, \mathcal L_\eta)$. 

The computation of the dimension of the cohomology $H^1_\rho(M, G)$ has been carried out by W.~Goldman
in \cite{G84}, section 1.5, for a general reductive group $G$. We~reproduce the argument in our case for the convenience of the reader.  
For any representation $\rho: \pi_1(M, \ast) \to U(1)$, the  cohomology $H^1_\rho(M, U(1)) \equiv H^1(M, \mathcal L_\rho)$, defined as the de Rham cohomology of the corresponding local system $\mathcal L_\rho$, can be identified with other cohomologies such as the singular, \v{C}ech, simplicial, cohomologies with local coefficients in the local system $\mathcal L_\rho$. By working  in simplicial cohomology, we note that  the (finite-dimensional) cochain complex is independent of the flat connection, so its Euler characteristic equals $2-2g$, since the  local system $\mathcal L_\rho$ has rank equal to $1$.  Now the Euler characteristic is invariant under taking the cohomology of the complex so the Euler characteristic of the graded cohomology space also equals $2-2g$.  

In the  case $H^0(M, \mathcal L_\rho) = 0$, since $M$ is a closed orientable surface, by Poincar\'e duality 
$H^2 (M, \mathcal L_\rho) = 0$. By definition of Euler characteristic of a complex, we have  
$$  0 - \text{ \rm dim}_\C\, H^1 (M,\mathcal L_\rho)  + 0 = 2- 2g\,,$$
so that  $\text{ \rm dim}_\C H^1 (M,\mathcal L_\rho) = 2g-2$ as stated.

In the case $H^0(M, \mathcal L_\rho) \equiv H^0_\eta (M, \C) \not =0$, by definition of the twisted cohomology there exists a non-zero function $f \in C^\infty(M)$ such that  $d_\eta f=0$.  The linear map $U_f$ defined as 
$$
U_f (\alpha) = {\bar f} \alpha \,, \quad  \text{ for all } \alpha \in \Omega^*(M, \C)\,,
$$
has the property that 
$$
d \circ U_f =  U_f \circ d_\eta\,, 
$$
hence it establishes a (unitary) isomorphism between $H^1_\eta(M, \C)$ and $H^1(M, \C)$. By Lemma~\ref{lemma:invariants}, we have that $H^0_\eta(M, \C) \not =0$ if and 
only if  $[\eta] \in H^1(M, \Z)$, and in that case $H^0_\eta(M, \C)$ has complex dimension $1$. It then follows by the formula for the Euler characteristic
that the dimension of $H^1_\eta(M, \C)$ is equal to $2g$. 

\end{proof}

The Teichm\"uller geodesic flow lifts by parallel transport to the Kontsevich--Zorich  cocycle on the bundle with fiber $H^1(M, \R)$  over the moduli space of Abelian differentials. The Kontsevich--Zorich cocycle projects onto
a flow on the bundle with fiber the real de Rham moduli space $H^1(M, \R) / H^1(M, \Z)$. We define a twisted cohomology bundle over the latter space.  The linear model for our construction is given by the bundle of
cohomologies of flat connections over the de Rham moduli space in the case of purely imaginary connections
(see \cite{GX08}, section 2.2).

\smallskip
The mapping class group $\Gamma_g$ acts on the stratum $\hat{ \mathcal H} (\kappa)$ in the Teichm\"uller space of Abelian differentials,
and it also acts by pull-back on the cohomology $H^1(M, \R)$ and on the quotient $H^1(M, \R)/ H^1(M, \Z)$.
We consider the space
$$
H^1_\kappa(M,\T) =  \left (\hat{ \mathcal H} (\kappa) \times H^1(M, \R)/ H^1(M, \Z) \right)/\Gamma_g\,
$$
and the bundle  with fiber $H^1_\eta (M, \C)$ at each point $[(h, \eta)] \in  H^1_\kappa(M,\T)$, that is,
$$
{\mathcal T}^1_\kappa(M, \C) := \{ (h, \eta, \alpha)\vert  [(h, \eta)] \in  H^1_\kappa(M,\T)  \text{ and } \alpha \in H^1_\eta (M, \C)\} / \Gamma_g\,.
$$
We remark that strictly speaking the elements  of this bundle are only defined up to equivalence relation
given by a unitary action of $H^1(M, \Z)$ on the twisted cohomology bundle and up to the multiplicative action 
of the group of complex numbers of modulus one.  In other terms, we can define
the real Hodge bundle
$$
H^1_\kappa(M,\R) =  \left (\hat{ \mathcal H}(\kappa) \times H^1(M, \R) \right)/\Gamma_g\,
$$
and the twisted cohomology bundle over the Hodge bundle 
$$
\hat {\mathcal T}^1_\kappa(M, \C) := \{ (h, \eta, \alpha)\vert  [(h, \eta)]\in  H^1_\kappa(M,\R)  \text{ and } \alpha \in H^1_\eta (M, \C)\} / \Gamma_g\,.
$$
In the above formula the symbol $[(h, \eta)]$ denotes the equivalence class of the pair $(h, \eta)$ with respect to
the action of the mapping class group $\Gamma_g$ by pull-back on the toral Hodge bundle over the lift $\hat{ \mathcal H} (\kappa)$ of the stratum ${ \mathcal H}(\kappa)$  to the Teichm\"uller space. 

\smallskip
The elements of the bundle $\hat {\mathcal T}^1_\kappa(M, \C) $ are defined up to the multiplicative action 
of the group complex numbers of modulus one. The subgroup $H^1(M, \Z)$ acts linearly on the bundle
$\hat {\mathcal T}^1_\kappa(M, \C)$ by unitary transformations and by definition we have
$$
{\mathcal T}^1_\kappa(M, \C) =  \hat {\mathcal T}^1_\kappa(M, \C)/ H^1(M, \Z) \,.
$$
The Teichm\"uller flow lifts to the bundle $H^1_\kappa(M,\T)$, then to the bundle ${\mathcal T}^1_\kappa(M, \C)$
by parallel transport. In other terms the action is given by the formulas 
$$
g_t [(h, \eta, \alpha)]=   [(g_t(h), \eta, \alpha)] \,, \quad \text{ for all }  \, [(h , \eta,\alpha)] \in {\mathcal T}^1_\kappa(M, \C)   \,.
$$
We remark that this action comes from an action of $SL(2, \R)$:  for all $g\in SL(2, \R)$ we define
$$
g [(h, \eta, \alpha)] =   [(g (h), \eta, \alpha)] \,, \quad \text{ for all }  \, [(h , \eta,\alpha)] \in {\mathcal T}^1_\kappa(M, \C)   \,.
$$
In the above formulas the symbol $[(h, \eta, \alpha)]$ denotes the equivalence class of the triple $(h, \eta,\alpha)$ with respect to the action of the mapping class group $\Gamma_g$ by pull-back on the twisted cohomology bundle over the lift $\hat{ \mathcal H} (\kappa) \times H^1(M, \T)$ of the toral Hodge bundle  $H^1_\kappa(M,\T)$  to the Teichm\"uller space. 

\section{First Variational Formulas}
\label{sec:Var_For}
We compute below variational formulas for the Hodge norm of real classes in
$$
H^1_{\eta} (M, \C) \oplus H^1_{-\eta} (M, \C) \,.
$$
Let $h\in \mathcal H(\kappa)$ be any Abelian differential. Since $h$ determines a complex structure 
on the surface, we can write  $\eta= \eta^{1,0} + \eta^{0,1}$, according to the Hodge decomposition, so that there exists a smooth function $f_\eta$  on $M$ such that 
$$
\eta^{1,0}  = f_\eta h   \quad \text{ and } \quad  \eta^{0,1} = \bar {f_\eta} \bar h \,.
$$
We can therefore introduce the Hodge decomposition 
$$
d_\eta = d^{1,0}_\eta + d^{0,1}_\eta =  d^{1,0} + 2\pi \imath \eta^{1,0} +  d^{0,1} + 2\pi \imath \eta^{0,1}
$$
and the twisted Cauchy-Riemann operators
$$
\partial^+_{h,\eta} = \partial^+_h +  2\pi \imath  \bar {f_\eta} \quad \text{ and } \quad  \partial^-_{h,\eta} = 
\partial^-_h +  2\pi \imath f_\eta\,.
$$
In fact, writing $\eta= a \re (h)+ b \im (h)$, we have
$$
\eta = a \frac{h +\bar h}{2} -\imath b \frac{h -\bar h}{2} =  \frac{ a-\imath b} {2} h +   \frac{ a+\imath b} {2} \bar h\,,
$$
hence in particular $f_\eta = \frac{ a-\imath b} {2}$ and we have
$$
\begin{aligned}
\partial^+_{h,\eta} = (S+\imath T) + \pi \imath (a +\imath b) = (S+ \pi \imath a) + \imath (T +\pi \imath b) \,, \\
\partial^-_{h,\eta} = (S-\imath T) + \pi \imath (a -\imath b) = (S+ \pi \imath a) - \imath (T +\pi \imath b) \,.
\end{aligned}
$$
Let us now consider the Teichm\"uller deformation  $g_t (h, \eta) = (h_t, \eta)$ with
$$
\re (h_t) = e^{-t} \re (h ) \quad \text{ and } \quad   \im (h_t) = e^t \im (h)\,.
$$
We have $\eta = a_t \re (h_t)+ b_t \im (h_t)$ with 
$$
a_t = e^t a \quad \text{ and } \quad b_t = e^{-t} b\,,
$$
hence
$$
\begin{aligned}
\partial^+_{h_t,\eta} = (e^tS+\imath  e^{-t} T) + \pi \imath (e^t a +\imath e^{-t} b) = 
e^t (S+ \pi \imath a) + \imath e^{-t} (T +\pi \imath b) \,, \\
\partial^-_{h_t,\eta} = (e^tS-\imath e^{-t}T) +  \pi \imath (e^t a -\imath e^{-t}b) = e^t (S+ \pi \imath a) - \imath e^{-t}(T +\pi \imath b)
\end{aligned}
$$
From these formula we derive the basic fact that
$$
\frac{d}{dt}  ( \partial^+_{h_t, \eta} ) = \partial^- _{h_t, \eta}  \quad \text{ and  } \quad  
\frac{d}{dt}  ( \partial^-_{h_t, \eta} ) = \partial^+ _{h_t, \eta}   \,.
$$

Let $\mathcal M^\pm_{h,\eta} \subset L^2_h(M)$ denote the kernels of the Cauchy-Riemann operators 
$\partial^\pm_{h,\eta}$ and, for simplicity of notation, let  $\mathcal M^\pm_{\eta, t} =
\mathcal M^\pm_{h_t,\eta}$ denote the kernels of the Cauchy-Riemann operators
$$\partial^\pm_{\eta, t} = \partial^\pm_{h_t, \eta}$$ along the orbit $g_t(h, \eta) = 
(h_t, \eta)$.  Any real class $c$ in the direct sum above can be represented as in the form
$$
c = \re ( [ m_{\eta, t} h_t ] +  [m_{ -\eta, t} h_t] )\,.
$$
with functions $m_{\eta, t} \in \mathcal M^+_{\eta, t}$ and $m_{-\eta, t} \in \mathcal M^+_{-\eta, t}$.  

\begin{definition} The Hodge norm of the real  twisted cohomology class  $c  \in  H^1_\eta (M, \C)$ represented
as $c= \re ( [ m_{\eta} h ] +  [m_{ -\eta} h] )$ with $m_{\pm \eta} \in {\mathcal M}^+_{h,\pm \eta}$ is defined as 
$$
\Vert (h, \eta, c) \Vert := \left(  \vert m_\eta  \vert_0^2 +  \vert m_{-\eta}  \vert_0^2 \right)^{1/2} \,.
$$
\end{definition} 

\begin{lemma} The variation of the Hodge norm is given by the formula
$$
\begin{aligned}
\frac{d}{dt} (\vert m_{\eta, t} \vert_0^2 + \vert m_{-\eta, t}\vert_0^2) &= 2 \re ( \<\overline{m}_{\eta, t}, m_{-\eta, t}  \> + \<\overline{m}_{-\eta, t}, m_{\eta, t}\>)  \\ &=  4\re \<\overline{m}_{\eta, t}, m_{-\eta,t}   \> \,.
\end{aligned}
$$
\end{lemma} 
\begin{proof}
Let $\pi^\pm_{\eta,t}:L^2_h(M) \to \mathcal M^\pm_{\eta,t}$
denote the orthogonal projections. By the condition that
$m_{\eta, t} \in \mathcal M^+_{\eta, t}$ and $m_{-\eta, t} \in \mathcal M^+_{-\eta, t}$, for all $t\in \R$, we claim that there exist
$(v_t), (w_t) \subset H^1(M)$, and 
$\phi_{\eta, t} \in \mathcal M^+_{\eta, t}$ and  $\psi_{-\eta, t} \in \mathcal M^+_{-\eta, t}$ such that
\begin{equation}
\label{eq:var}
\begin{cases}  m_{\eta, t} &=  \partial^+_{\eta,t} v_t  +  \pi^-_{\eta,t} (m_{\eta, t})  \\  
                        \frac{d }{dt} m_{\eta, t} &= -\partial^-_{\eta,t} v_t  + \phi_{\eta, t}   \end{cases}   \, \text{and } \,
                        \begin{cases}  m_{-\eta, t} &=  \partial^+_{-\eta,t} w_t  +  \pi^-_{-\eta,t} (m_{-\eta, t})  \\  
                        \frac{d }{dt} m_{-\eta, t}&= -\partial^-_{-\eta,t} w_t  + \psi _{-\eta, t}\,.\end{cases} 
\end{equation}
The proof of the above formulas follows the argument in the untwisted case given in \cite{F02}, Lemma 2.1.
For the reader's convenience we give the argument below.

Since $m_{\eta, t} \in  \mathcal M^+_{\eta, t}$ and $m_{-\eta, t} \in  \mathcal M^+_{-\eta, t}$, it follows from 
the definitions  that $\partial^+_{\eta, t} m_{\eta, t} = \partial^+_{-\eta, t} m_{-\eta, t}=0$, for all $t\in \R$, hence
 by a straightforward calculation we have
 $$
\begin{aligned}
 &\partial^-_{\eta, t} m_{\eta, t} +  \partial^+_{\eta, t} (\frac{d m_{\eta, t}}{dt} )  = \frac{d }{dt} ( \partial^+_{\eta, t} m_{\eta, t})  = 0 \,; \\
& \partial^-_{-\eta, t}  m_{-\eta, t}   +  \partial^+_{-\eta, t} (\frac{dm_{-\eta, t}}{dt} )  =  \frac{d }{dt} ( \partial^+_{-\eta, t} m_{-\eta, t}) = 0 \,.
\end{aligned}
$$
Moreover, by the definition of the cocycle, since the action of the Teichm\"uller flow on the twisted cohomology bundle is by parallel transport, the $d_\eta$-cohomology
class of the real  $d_{\eta}$-closed $1$-form $\re (m_{\eta, t} h_t + m_{-\eta, t} h_t )$ is constant with respect to $t\in \R$, hence  there exists a one-parameter family of smooth functions $(f_t)$ such that
$$
\frac{d}{dt} \re (m_{\eta, t} h_t + m_{-\eta, t} h_t )  =  d_{\eta} f_t  +   d_{-\eta} \bar{f_t}\,.
$$
Since $\frac{d h_t}{dt} = -\overline{h_t}$ we have
$$
\begin{aligned}
\frac{d m_{\eta, t}}{dt} + \frac{d m_{-\eta, t}}{dt}  - ( \overline{m}_{\eta, t} +\overline{m}_{-\eta, t}) &=  -\partial^-_{\eta,t} (v_t + \bar{w_t})  -\partial^-_{-\eta,t}  (\bar{v_t } +w_t)   
 \\ 
&+ \phi_{\eta, t}  + \psi_{-\eta, t} -\overline{ \pi^-_{\eta,t} (m_{\eta, t})} - \overline{ \pi^-_{-\eta,t} (m_{-\eta, t})}\,,
\end{aligned}
$$
which implies that $f_t = -(v_t + \bar{w_t})$,  and
\begin{equation}
\label{eq:phipsi}
\phi_{\eta, t} = \overline{ \pi^-_{-\eta,t} (m_{-\eta, t})}  \quad \text{ and } \quad  \psi_{-\eta, t}= \overline{ \pi^-_{\eta,t} (m_{\eta, t})}\,.
\end{equation}
The formulas claimed above are therefore proven.

The variation of the Hodge norm is then given by the formula
$$
\begin{aligned}
\frac{d}{dt} (\vert m_{\eta, t} \vert_0^2 + \vert m_{-\eta, t}\vert_0^2) &=   2 \re ( \<m_{\eta, t}, \frac{d m_{\eta, t}}{dt}\> + \<m_{-\eta, t}, \frac{dm_{-\eta, t}}{dt}\>) \\
&= 2 \re ( \<m_{\eta, t},  \overline{ \pi^-_{-\eta,t} (m_{-\eta, t})}   \> + \<m_{-\eta, t},\overline{ \pi^-_{\eta,t} (m_{\eta, t})}\>) \\
&= 2 \re ( \<\overline{m}_{\eta, t},  \pi^-_{-\eta,t} (m_{-\eta, t})   \> + \<\overline{m}_{-\eta, t},\pi^-_{\eta,t} (m_{\eta, t})\>) \\
&= 2 \re ( \<\overline{m}_{\eta, t}, m_{-\eta, t}  \> + \<\overline{m}_{-\eta, t}, m_{\eta, t}\>) = 
4\re \<\overline{m}_{\eta, t}, m_{-\eta, t}   \> \,.
\end{aligned}
$$
In the above chain of identities, we argue as follows. The first identity is given by the general formula for the derivative of the square of a Hilbert norm.  The second follows from
formulas~\eqref{eq:var} and \eqref{eq:phipsi}, since the conditions that $m_{ \eta, t} \in {\mathcal M}^+_{\eta,t}$ and  $m_{ -\eta, t} \in {\mathcal M}^+_{-\eta,t}$ 
imply respectively that
$$
\begin{cases}
\<m_{\eta, t},   \partial^-_{\eta,t} v_t     \> =  - \<m_{\eta, t},  ( \partial^+_{\eta,t})^\ast v_t     \> =0 \,, \\
\<m_{-\eta, t}, \partial^-_{-\eta,t} w_t    \> =   - \<m_{\eta, t},  ( \partial^+_{\eta,t})^\ast w_t     \> =  0\,.
\end{cases}
$$
The third identity holds since the real part of a complex number equals the real part of its conjugate. Finally, the fourth identity follows from formula \eqref{eq:var}
since, as a consequence of the fact that $\overline{m}_{\eta, t} \in {\mathcal M}^-_{-\eta,t}$ and $\overline{m}_{-\eta, t} \in {\mathcal M}^-_{\eta,t}$, we have respectively
$$
\begin{cases}
\<\overline{m}_{\eta, t},   \partial^+_{-\eta,t} w_t     \> =  - \<\overline{m}_{\eta, t},  ( \partial^-_{-\eta,t})^\ast w_t     \> =0 \,, \\
\<\overline{m}_{-\eta, t}, \partial^+_{\eta,t} v_t    \> =   - \<\overline{m}_{\eta, t},  ( \partial^-_{\eta,t})^\ast v_t     \> =  0\,,
\end{cases}
$$
The fifth and last identity again follows by taking complex conjugation inside the (second) real part.  The argument is thus complete.

\end{proof}
Let $\Lambda_\kappa : H^1_\kappa (M, \T) \to \R^+ \cup \{0\}$ be the function defined as
\begin{equation}
\label{eq:Lambda_kappa}
\Lambda_\kappa(h, [\eta]) :=  \sup \left\{  \frac{2 \vert \<\overline{m}_{\eta}, m_{-\eta}   \> \vert}{ \vert m_{\eta} \vert_0^2 + \vert m_{-\eta}\vert_0^2 } \, : \, ( m_\eta,  m_{-\eta}) \in \mathcal M^+_\eta\times \mathcal M^+_{-\eta}
\setminus \{(0,0)\} \right \} \,.
\end{equation}

As an immediate consequence of the first variational formulas, we derive an upper bound for the growth
of the Hodge norm of twisted cohomology classes under the twisted cocycle.

\begin{lemma}  \label{lemma:growth_1} Let $c \in H^1_\eta(M, \C)$.  We have
$$
\Vert  g_t ([h,\eta, c]) \Vert \leq  \Vert [h, \eta, c] \Vert  \exp \left(  \int_0^t  \Lambda_\kappa(g_s([h, \eta])  ) ds  \right)\,.
$$
\end{lemma}

We finally prove that the function $\Lambda_\kappa \leq 1$ everywhere and $\Lambda_\kappa<1$  outside
of any neighborhood of the zero section $H^1(M, \Z)$ of $H^1(M, \T)$.

\begin{lemma} \label{lemma:Lambda_gap}  The function $\Lambda_\kappa$ is continuous with values in $[0,1]$ and
$$
\Lambda_\kappa(h, [\eta]) <1\,, \quad   \text{ \rm for all } (h,[\eta])\,\, \text{\rm such that }\,\, [\eta] \not \in H^1(M, \Z) \,.
$$
\end{lemma} 
\begin{proof} The holomorphic and anti-holomorphic part, $d^{1,0}$ and $d^{0,1}$  of the 
exterior differential $d$ are elliptic, in the sense that for any $1$-form $\alpha$ on $M$ we have
\begin{equation}
\label{eq:d_elliptic}
\vert \alpha \vert_{\Omega^1 H^1_h(M)} \leq   \vert \alpha \vert_{\Omega^1 L^2_h(M)} + \min\{ \vert d^{1,0} \alpha \vert_{\Omega^2 L^2_h(M)},   \vert d^{0,1} \alpha \vert_{\Omega^2 L^2_h(M)} \}\,.
\end{equation}
We claim that by Rellich compact embedding theorem (see for instance~\cite{Ad}, Th. 6.2), for all $r>s>1$, the  unit ball in the space
$$
\{ (\eta, \alpha_{\eta} , \alpha_{-\eta}) \in \Omega^1 H^r (M) \times \text{ \rm Ker } (d^{0,1}_\eta) \times \text{ \rm Ker } (d^{0,1}_{-\eta} ) \}
$$
that is, the set of   $(\eta, \alpha_{\eta}, \alpha_{-\eta})$  such that 
\begin{equation}
\label{eq:unit_ball}
\vert \eta  \vert^2_{\Omega^1H^r (M) }+ \vert \alpha_{\eta} \vert_{\Omega^1 L^2_h(M)} ^2 +\vert \alpha_{-\eta} \vert_{\Omega^1 L^2_h(M)} ^2  \leq 1 
\end{equation}
is compact in the space $\Omega^1 H^s (M) \times [\Omega^1 L^2_h(M)]^2$.  

In fact, by Rellich embedding theorem the embedding $\Omega^1 H^r (M) \to \Omega^1 H^s (M)$ is compact for any $s>r$. In addition, by Sobolev
embedding theorem, for $r>1$ the space $\Omega^1 H^r (M)$ embeds continuously into the space of $1$-forms with continuous coefficients.  It then
follows from the bound in formula~\eqref{eq:d_elliptic} that whenever $\eta$, $\alpha_{\eta}$ and  $\alpha_{-\eta}$ belong to the set described in 
formula~\eqref{eq:unit_ball}, then  $\alpha_{\eta}$ and  $\alpha_{-\eta}$ belong to a bounded set in $\Omega^1 H^1_h(M)$, hence by Rellich embedding
theorem, to a compact subset of $\Omega^1 L^2_h(M)$.
It follows, in particular that for each $\eta \in  \Omega^1 H^r (M)$ the subspaces $\text{ \rm Ker } (d^{0,1}_{\pm\eta})$, hence also the subspaces 
$\text{ \rm Ker } (d^{1,0}_{\pm\eta})$,  are finite dimensional, and that
the supremum in the definition of the function $\Lambda_\kappa$ is achieved. We observe that by Hodge theory the complex dimension of 
$\mathcal M^\pm_{h, \pm \eta}$ equals half the complex dimension of the twisted cohomology, which we have computed in
Lemma~\ref{lemma:cohom_dim}. 

By the ellipticity of the operators $d^{1,0}$ and $d^{0,1}$, that is, from formula \eqref{eq:d_elliptic}, it also follows that the spaces 
$\text{ \rm Ker } (d^{1,0}_{\pm \eta})$ and $\text{ \rm Ker } (d^{0,1}_{\pm \eta})$ depend continuously, as subspaces of $\Omega^1L^2_h(M)$, on the closed 
$1$-form $\eta \in \Omega^1 H^r (M) $, hence the spaces $\mathcal M^{\pm}_{h,\pm \eta}$ depend continuously on the pair 
$(h, \eta)\in \mathcal H (\kappa) \times H^1(M, \T)$. Thus we conclude that the function $\Lambda_\kappa$ is continuous. 

\smallskip
By the Schwarz inequality we have
$$
\vert \<\overline{m}_{\eta}, m_{-\eta}  \> \vert \leq   \vert m_{\eta} \vert_0  \vert m_{-\eta} \vert_0 \leq  \frac{1}{2} ( \vert m_{\eta} \vert_0^2 + \vert m_{-\eta}\vert_0^2)\,,
$$  
with equality only if there exists a non-zero constant $c\in \C$ (of modulus one) such that $m_{\eta}= c 
\overline{m}_{-\eta}$. From this condition, it follows 
that $m_{\eta} \in \mathcal M^+_{\eta} \cap \mathcal M^-_{\eta}$, that is, $\partial^+_{\eta} m_{\eta} = \partial^-_{\eta} m_{\eta}=0$
hence in particular 
$$
(d + 2\pi \imath \eta) m_{\eta}=0\,.
$$
It follows that $H^0_\eta(M, \C)$ is non-trivial, and Lemma~\ref{lemma:invariants}  implies that $[\eta ]\in H^1(M, \Z)$.  

\medskip
\noindent A direct alternative argument goes as follows. Let $(X, Y)$ be a frame such that $\imath _X \eta =0$ and $\imath _Y \eta =-1$. We then have
$$
X m_\eta =\imath_X (d + 2\pi \imath \eta) m_\eta =0 \quad \text{ and } \quad (Y-2\pi \imath )m_\eta = \imath_Y (d + 2\pi \imath \eta) m_\eta =0\,.
$$
The first condition implies that $\eta$ defines a completely periodic foliation $\mathcal F_\eta$.  The second
condition that $M/ \mathcal F_\eta$ endowed with the transverse measure covers a circle of unit length, hence 
$[\eta]\in H^1(M, \Z)$ (as all periods are integers).  
\end{proof}

\medskip

We conclude that if the Teichm\"uller orbit of $(h, \eta)$ visits the complement of any given neighborhood
of the zero section $H^1(M, \Z)$ with positive frequency, then there exist constants $C>0$ and  $\Lambda<1$ such that, for all $c\in H^1_\eta(M, \C)$  we have
$$
\Vert  g_t ([h,\eta, c]) \Vert \leq   C \Vert [h, \eta, c]\Vert  e^{\Lambda t}\,,  \quad \text{ for all } \, t>0 \,.
$$
In the next section we investigate the dynamics of the lift of the Teichm\"uller flow to the toral bundle 
$H^1_\kappa(M,\T)$ over the stratum $\mathcal H (\kappa)$ of the moduli space of Abelian 
differential, with fiber $H^1_h(M,\T):= H^1(M,\R)/H^1(M, \Z)$ at any $h\in \mathcal H(\kappa)$,  with particular attention to the set of trajectories which asymptotically ``spend all their time'' in any neighborhood of the zero section $H^1(M,\Z)$ of the bundle.

\section{The toral Kontsevich--Zorich cocycle}
\label{sec:Tor_KZ}

The projection of the Kontsevich--Zorich cocycle to the quotient toral bundle $H^1_\kappa(M, \T):=
H^1_\kappa(M, \R) / H^1_\kappa(M, \Z)$  is the 
key dynamical system behind the proof of generic  weak mixing for translation flows, for interval
exchange transformations \cite{AvF07}, generic translation flows on non-arithmetic Veech surfaces 
\cite{AD16}, \cite{AL}.  We remark that the bundle $H^1_\kappa(M, \T)$ is isomorphic to the {\it character
variety bundle} introduced in \cite{FG}  for the compact group~$U(1)$. In fact, elements of the character variety
for a group $G$ are homomorphisms $\rho: \pi_1(M, \ast) \to G$. For any Abelian group, homomorphisms 
of $\pi_1(M, \ast)$ to $G$ factor through the integral homology $H_1(M, \Z)$.  Every homomorphism
of $H_1(M, \Z)$ to $U(1)\equiv  \R/ \Z$  lifts to a homomorphism from $H_1(M, \Z)$ to $\R$, which is an element
of $H^1(M, \R)$. It follows that the character variety for $G=U(1)$ is isomorphic to $H^1(M, \R)/ H^1(M, \Z)$. 
It was proved in~\cite{FG} that the lift of the Teichm\"uller flow to the bundle $H^1_\kappa(M, \T)$
is ergodic, in fact even mixing, with respect to the canonical lift of the any of Masur--Veech measures on 
strata of the moduli space of Abelian differentials. 

It was proved in \cite{AvF07} that the horizontal translation flow of a translation surface
$(M, h)$ is weakly mixing  if the line $\R [\re(h)] \in H^1_h(M, \T)$ does not intersect the {\it weak
stable lamination} of the zero section of the bundle $H^1_\kappa(M, \T)$. (Note that in fact  in \cite{AvF07} the vertical flow
was considered, hence the condition was applied to the line $\R [\im(h)]$ instead of $\R [\re(h)]$. The two points of view are equivalent by rotation of the Abelian differentials). 

The weak stable lamination is defined as the set of all $c \in H^1_h(M, \T)$ such that the orbit of $c$ under the projected Kontsevich--Zorich cocycle converges to the zero section 
along all unbounded sequences of return times to certain  compact subsets of the space of zippered rectangles. It was then proved in~\cite{AvF07} by a ``linear elimination'' argument (see \cite{AvF07}, Appendix A) that the set of translation surfaces $(M,h)$ such that $\R [\re(h)]$ intersects the weak stable lamination has Hausdorff codimension at least $g-1$ for the Masur--Veech measures (in general, the Hausdorff codimension is equal to the number of non-tautological positive exponents of the Kontsevich--Zorich cocycle). It was also proved by a ``non-linear elimination'' argument that a similar property holds for Lebesgue almost all interval exchange transformations. By the Veech criterion, the authors derived that almost all interval exchange transformations and almost all translation flows are weakly mixing.

\smallskip
We introduce a version of the weak stable space.   Let $K\subset \mathcal H(\kappa)$ be a non-empty compact subset and let $U$ be any open neighborhood of the zero section  of the bundle $H^1_\kappa (M, \T)$, that is, the projection of a neighborhood of $H^1_\kappa (M, \Z) \subset H^1(M, \R)$. For every $h\in \mathcal H(\kappa)$, and every
$\epsilon>0$,  let $W^{s}_{K,U}(h, \epsilon) \subset H^1_h(M, \T)$ denote the set
$$
W^{s}_{K, U}(h, \epsilon) = \{ c \in H^1_h(M, \T)  \,\vert \,   \limsup_{t \to +\infty} \frac{ \int_0^t \chi_K(g_\tau (h))\chi_U (g_\tau(h,c) ) d\tau } { \int_0^t \chi_K(g_\tau (h)) d\tau  } \geq 1-\epsilon  \} \,.
$$
Let $W^s_K(h)$ denote the intersection of all sets $W^{s}_{K,U}(h, \epsilon)$ as $U$ varies over the family $\mathcal U$ of all neighborhoods of the zero section of the bundle $H^1_\kappa (M, \T)$ and $\epsilon \in (0,1)$:
$$
W^{s}_K(h) :=  \bigcap_{U\in \mathcal U}\bigcup_{\epsilon\in (0, 1)}   W^{s}_{K,U}(h, \epsilon)\,.
$$
The following lemma provides a simple but effective way to bound  the Hausdorff dimension of a set
defined as an upper limit.
\begin{lemma} \label{lemma:Hdim} Let $\{W_n\}$ be a sequence of subsets of $\R^d$ and let $W \subset \R^d$  be
the set defined as 
$$
W = \limsup_{n \to \infty} W_n= \cap_{n\in \N} \cup_{m\geq n} W_m\,.
$$
Assume that, for each $n\in \N$, the set $W_n$ can be covered by $N_n$ balls of radius $R_n$. Then 
the Hausdorff dimension $\text{\rm H-dim}(W)$ satisfies the upper bound
$$
\text{\rm H-dim}(W) \leq \inf \{ \delta >0 \vert  \lim_{n\to +\infty}    \sum_{m\geq n}  N_m  R_m^\delta=0\}\,.
$$
\end{lemma}
\begin{proof}
Let $H^\delta$ denote the $\delta$-dimensional Hausdorff outer measure on $\R^d$.  Let $\{\Omega_m\}$
be a cover of the set $W_m$ by $N_m$ balls of radius $R_m$.  It follows that, for each $n\in \N$ we have
$$
W  \subset \Omega^{(n)} := \bigcup_{m\geq n} \Omega_m \,.
$$
By assumption we have
$$
\sum_{B \in \Omega^{(n)}}  \vert B \vert^\delta =  \sum_{m\geq n} \sum_{B \in \Omega_m}  \vert B \vert^\delta     = C_d^\delta \sum_{m\geq n}  N_m  R_m^\delta \,.
$$
By the definition of outer measure, it follows that
$$
H^{\delta} (W) \leq  C_d^\delta \lim_{n\to +\infty}    \sum_{m\geq n}  N_m  R_m^\delta \,.
$$
We conclude that $H^\delta (W)=0$ for any $\delta >0$ such that 
$$
 \lim_{n\to +\infty}    \sum_{m\geq n}  N_m  R_m^\delta =0\,, 
$$
 hence  $\text{\rm H-dim}(W) \leq \delta$ by the properties of Hausdorff dimension.
The argument is thus complete.
\end{proof}
We  generalize below to our setting the ``linear elimination'' argument of \cite{AvF07}.

\smallskip
Let $(t_n)$ a sequence of return times of the Teichm\"uller orbit $\{g_t (h) \vert t>0\}$  to the compact set $K\subset \mathcal H(\kappa)$. Let us define the sets
\begin{equation}
\label{eq:Wnsets}
W^{s}_{K,U,n}(h,\epsilon) = \{ c \in H^1_h(M, \T)  \,\vert \, \frac{\int_0^{t_n}
\chi_K(g_t(h))  \chi_U (g_t (h,c) ) dt }{ \int_0^{t_n}
\chi_K(g_t(h))  dt} \geq 1-\epsilon \}\,.
\end{equation}
Let $r_K>0$ be a radius such that, for all $h\in K$, the closed Hodge ball of radius $r_K$ in $H^1_h(M, \T)$ centered at the origin is isometric to the closed Hodge  ball of the same Hodge radius in $H^1_h(M,\R)$.  
Let $U(r)$ denote a neighborhood of radius equal to $r\in (0, r_K)$  (with respect to the Hodge metric) of the zero section of $H^1_\kappa(M,\T)$.

Let $h\in  \mathcal H(\kappa)$ be a Birkhoff generic point for the Teichm\"uller geodesic flow and Oseledets regular for the Kontsevich--Zorich cocycle on the Hodge bundle $H^1_\kappa(M, \R)$ with respect to the Masur--Veech measure. 
\begin{lemma}  
\label{lemma:covering}
There exist constants  $C_K>1$, $\nu>0$ and there exists a function $\epsilon_{K}:(0, r_K) \to (0,1)$ such that
$\lim_{r\to 0^+} \epsilon_K(r)=0$ such that the following holds. Let $V$ denote any affine subspace parallel to a subspace $V_0$ transverse to the central-stable space  $E^{cs}(h)$ and let $d_u := \text{ \rm dim }(V_0)$ the unstable dimension. The set  $W^{s}_{K,U(r),n}(h,\epsilon) \cap V$ is covered by at most $N_n (r,\epsilon)$ balls of Hodge radius at most  $R_n(r, \epsilon)$  so that the following bounds hold: 
\begin{equation}
\label{eq:est_claim}
\begin{aligned}
&\limsup_{n\to +\infty}  \frac{1}{t_n} \log R_n(r,\epsilon) \leq   - C^{-1}_K \mu_\kappa(K) \nu (1-\epsilon) \,; \\
&\limsup_{n\to +\infty} \frac{1}{t_n} \log N_n(r, \epsilon)  \leq  C_K d_u(\epsilon+\epsilon_K(r))\,.
\end{aligned} 
\end{equation}
\end{lemma}
\begin{proof}
The first estimate follows from the Birkhoff ergodic theorem and from the Oseledets theorem. 
For each $n\in \N$,  let $\tau_n\in [0,t_n]$ be defined as  
$$
\tau_n  :=  \inf_{c\in W^{s}_{K,U(r),n}(h, \epsilon)}  \sup \{  t\in [0, t_n]  \vert  g_{t} (h,c) \in U(r) \text{ and } g_t(h)\in K \}\,.
$$
Since $h$ is Birkhoff generic, by Birkhoff ergodic theorem and by the definition of the set $W^{s}_{K,U(r),n}(h, \epsilon)$, we have
 $$
\liminf_{n \to +\infty} \, \frac{\tau_n} {t_n} \geq  (1-\epsilon)\mu_\kappa (K) \,.
$$
By compactness and by the Oseledets theorem, there exists $\nu>0$ such that, for each $n\in \N$, every  connected component of the set
$W^{s}_{K, U(r),n}(h, \epsilon)\cap V$ is contained in a ball of radius $C_K r e^{-\nu \tau_n}$, hence the  estimate on the sequence $(R_n(r,\epsilon))$ holds, for all $r>0$. 

The second estimate, on the number $N_n(r,\epsilon)$ of connected components of the set $W^{s}_{K,U(r),n}(h,\epsilon) \cap V$, is proved by coding trajectories
over the time-interval $[0, t_n]$, as follows.   The total number of connected components will be estimated by the product of the number of words in the coding times
the number of connected components of the subset of the set $W^{s}_{K,U(r),n}(h,\epsilon) \cap V$ of trajectories with the same coding.

Let us describe the coding. Let $\pi: H^1_\kappa(M,\R) \to \mathcal H(\kappa)$ denote the bundle projection.
We code trajectories  according to whether they
are in~$U(r) \cap \pi^{-1}(K)$ (coded by the symbol $u$), in~$U(r)^c\cap \pi^{-1}(K)$ (coded by the symbol~$u'$) or the Teichm\"uller orbit is not 
in~$K$ (coded by the symbol~$K'$). In other terms, the coding is based on the following map $c$ from the toral cohomology bundle 
$H^1_\kappa (M, \T)$ to the alphabet $\{ u, u', K'\}$ defined as follows: 
$$
c (h, [\eta])  := \begin{cases}   u\,,   \quad  \text{ if }  (h, [\eta]) \in \pi^{-1}(K) \cap U(r)\,, \\
 u'     \quad  \text{ if }  (h, [\eta]) \in \pi^{-1}(K) \cap  \left( H^1_\kappa(M, \T) \setminus U(r)\right)\,, \\
  K'     \quad  \text{ if }  (h, [\eta]) \in  \pi^{-1} (\mathcal H(\kappa) \setminus K) \,,  \end{cases} 
$$
with the modifications described below.

Maximal trajectory arcs in $\pi^{-1}(K)$, but {\it outside} of the set  $U(r)\cap \pi^{-1}(K)$, have time length at least  $C^{-1}_K \vert \log r \vert$, since the maximal 
expansion rate  of the Kontsevich--Zorich cocycle at time $t>0$ with respect to the Hodge norm is  bounded above by $e^t$ and above the compact set 
$K \subset \mathcal H(\kappa)$ lattice points separation (with respect to the Hodge distance) is bounded below.  Hence it is enough to code trajectories in 
$K$ at time intervals equal to  $C^{-1}_K \vert \log r \vert/2$. In addition,  by compactness there exists a  
function $\delta_K: (0,1] \to \R^+$ (depending on the compact set  $K\subset \mathcal H(\kappa)$)  with $\lim_{r\to 0^+} \delta_K(r) =+\infty$, 
 such that, for all $h\in K$,  for all $r>0$ and for $\vert t \vert \leq \delta_K(r)$ the image  of the ball $B(0, r) \subset H^1_h(M, \R)$ (in the Hodge metric) 
 contains a single point (the origin) of the lattice $H^1_{g_t(h)} (M, \Z) \subset H^1_{g_t(h)} (M, \R)$.  Hence any trajectory arc which exits $\pi^{-1}(K)$, with 
 both endpoints in  $U(r) \cap \pi^{-1}(K)$, will still be coded by the letter $u$ unless it has time-length larger than $\delta_K(r)>0$.
 
 Let then $\delta'_K(r) = \min \{ C_K^{-1}\vert \log r \vert , \delta_K(r)\}$.  By the above remarks, it is enough to code trajectories as follows: 
 we divide each trajectory segment into segments of equal length $\delta'_K(r)$ 
(and a remainder which we neglect) and assign to each segment the label $u$  when it intersects  $ \pi^{-1}(K) \cap U(r)$, and it is entirely contained in 
$\pi^{-1}(K)$,  the label $u'$ when it is entirely contained in $\pi^{-1}(K) \cap  \left( H^1_\kappa(M, \T) \setminus U(r)\right)$ but not the complement of 
$\pi^{-1}(K)$, and finally $K'$ in the remaining case, when it intersects  the complement of $\pi^{-1}(K)$.

%corresponding to its midpoint.  

%Trajectory segments outside  the set $\pi^{-1}(K)$ are coded by the label $K'$ and any word $K'w K' $  with $w \subset \{u,u'\}$ is collapsed to the single letter 
%$K'$  whenever the time-length of the trajectory arc corresponding to the subword $w$ has time length less than $C^{-1}_K \vert \log r \vert/2$.

%By compactness, there also exists a  function $\delta_K: (0,1] \to \R^+$ (depending on the compact set $K\subset \mathcal H(\kappa)$)  with $\lim_{r\to 0^+} \delta_K(r) =+\infty$, 
%such that, for all $h\in K$,  for all $r>0$ and for $\vert t \vert \leq \delta_K(r)$ the image  of the ball $B(0, r) \subset H^1_h(M, \R)$ (in the Hodge metric) contains a single point (the origin) of the lattice $H^1_{g_t(h)} (M, \Z) \subset H^1_{g_t(h)} (M, \R)$.  Hence any trajectory arc which exits $\pi^{-1}(K)$, with both endpoints in 
%$U(r) \cap \pi^{-1}(K)$, will still be coded by the letter $u$ unless it has time-length larger than $\delta_K(r)>0$.

 % Let $\delta'_K(r) = \min \{ C_K^{-1}\vert \log r \vert , \delta_K(r)\}$. Since every orbit arc which corresponds to the same letter in a word has time length at least $\delta'_K(r)$,
% and since the symbol $u'$  occupies a total time at most $\epsilon t_n$ and $K'$ occupies a total time at most $\mu_K t_n$ (with $\mu_K\to 0^+$ as $\mu_\kappa (\mathcal H(\kappa) \setminus K) \to 0$),  it follows that the total number of different words is at most (by standard bounds on the binomial coefficients)
Let $t_n^K$ denote the total time that the trajectory spends in $\pi^{-1}(K)$ over the time interval $[0, t_n]$.  By the ergodic theorem, there exists $\mu_K>0$ (with $\mu_K\to 0^+$ 
as $\mu_\kappa (\mathcal H(\kappa) \setminus K) \to 0$)  such that $t_n -t_n^K \leq \mu_K t_n$.  
Since by assumption the total time that  the trajectory
spends in $\pi^{-1}(K) \cap U(r)$ is at least $(1-\epsilon) t_n^K$, the total number of different words is at most (by standard bounds on the binomial coefficients)
\begin{equation}
\label{eq:word_counting} 
 \binom { \frac{t^K_n}{\delta'_K(r)}  }  {\frac{ \epsilon t^K_n}{\delta'_K(r)}}   \binom { \frac{t_n}{\delta'_K(r)}  }  {\frac{ \mu_K t_n}{\delta'_K(r)}} 
  \leq  \left (\frac{e}{\epsilon} \right)^{\frac{\epsilon t^K_n}{\delta'_K(r)}} \left (\frac{e}{\mu_K} \right)^{\frac{\mu_K  t_n}{\delta'_K(r)}}  \,.
\end{equation}
For every word $w$, let $\Gamma'_w$ denote the set of arcs of trajectory of the cocycle in 
$U(r)^c \cap \pi^{-1}(K)$  with both endpoints in $U(r) \cap \pi^{-1}(K)$ , and let $\Gamma''_w$ denote the set of arcs of trajectory of the cocycle, with both endpoints in $U(r) \cap \pi^{-1}(K)$, which project to a Teichm\"uller arc not contained in $K$ and have time-length at least $\delta_K(r)>0$. In other terms, 
\begin{itemize}
\item  $\Gamma'_w$ is the set of orbit arcs  corresponding to strings  $w'= u' \dots u'$   such that $uw'u$  is a substring of $w$\,;
\item $\Gamma''_w$ is the set of orbit arcs  corresponding to strings  $w''= w_1 w_2 \dots w_\ell$, with $w_i \in \{u', K'\}$ for all $i\in\{1, \dots, \ell\}$,   such that $uw''u$  
is a substring of $w$.
\end{itemize} 
For every orbit  arc $\gamma \in \Gamma'_w \cup \Gamma''_w$, let $\tau (\gamma)$ denote its time length.

We claim that there exist  constants $C'_K>0$ and $r_K>0$ such that for any fixed word $w$ the number of different connected components with code $w$ is at most
\begin{equation}
\label{eq:comp_counting}
C'_K\prod_{\gamma'\in\Gamma'_w}  \max  (1, (r/r_K)  e^{\tau(\gamma')})^{d_u}\,  \prod_{\gamma''\in \Gamma''_w} 
  \max  (1, (r/r_K) e^{\tau(\gamma'')})^{d_u}    \,.
\end{equation}
This statement follows from the fact that the maximal expansion of the cocycle in a time $\tau>0$ with respect to the Hodge norm is equal to $e^\tau$, hence the bound follows by a volume estimate on the unstable space.
In fact, we argue as follows. Let $\mathcal C_{w,n}(h) \subset W^s_{K, U(r),n} \cap V$ denote the subset of all cohomology
classes which have a symbolic sequence equal to $w$ up to time $t_n>0$.  Every $c \in \mathcal C_{w,n}(h)$ 
can be labeled by the string $(m_1, \dots, m_k)$  of distinct lattice points in $H^1_h(M, \Z)$ such that $g_t (h,c)$ visits a ball $B_{g_{\tau_i}(h)}(m_i,r)$ in the Hodge metric on $H^1_{g_{\tau_i}h}(M, \Z)$ at a time $\tau_i$ for a sequence of times $0\leq \tau_1 < \tau_2 < \dots < \tau_k <\dots \leq t_n$.  Lattice points along the Teichm\"uller orbit $g_\R(h)$  can be identified by parallel transport.  On each subinterval $I=[a,b] \subset [0,t_n]$ such that $g_I (h) \subset 
U(r) \cap \pi^{-1}(K)$, by definition we have that   $g_a(h,c) \in B_{g_a(h)}(m,r)$ implies $g_b(h,c) \in B_{g_b(h)}(m,r)$, for any $c \in \mathcal C_{w,n}(h)$. Now on each maximal subinterval $I=[a,b] \subset [0,t_n]$ such that $g_I (h) \subset K^c$ or $g_I (h,c)\subset U(r)^c \cap \pi^{-1}( K)$ every Hodge ball $B_{g_a(h)}(m,r)$ is mapped by the cocycle into a subset of a Hodge ball of radius at most $r e^{\vert I \vert}$. Since $K$ is compact, there exists a constant $r_K>0$ such that, for any Abelian differential $h\in K$,  a Hodge ball of radius at most $re^{\vert I \vert}$ in $H^1_h(M, \R)$  contains at most $ (r/r_K)^{d_u} e^{d_u\vert I \vert}$ lattice points. It follows that for each such subinterval our upper bound on the number of connected components is multiplied times a factor $  (r/r_K)^{d_u}e^{d_u\vert I \vert}$.  The claim follows. 

Thus by formulas \eqref{eq:word_counting} and \eqref{eq:comp_counting},  for $r\leq r_K$ we have proved the estimate
$$
\begin{aligned}
\log N_n(r,\epsilon) &\leq  \log C'_K  + \{   \epsilon(1+\vert \log \epsilon \vert) + \mu_K (1+ \vert \log \mu_K \vert) \}  \frac{t_n}{ \delta'_K(r)}  \\ &+
 d_u \left( \sum_{\gamma'\in \Gamma'_w}   \max (0, \tau(\gamma') ) + \sum_{\gamma''\in \Gamma''_w}
 \max (0, \tau(\gamma'') )\right) \,.
\end{aligned}
$$
It remains to estimate the third and fourth term on the RHS of the above inequality.  For the third term,
since 
$c \in W^s_{K,U(r),n} (h, \epsilon)$ and $\Gamma'_w$ denote the set of arcs of 
trajectory of the cocycle in the complement of $U(r)$, which project to a Teichm\"uller arc in $K$, we have
$$
\sum_{\gamma'\in \Gamma'_w}  \max (0, \tau(\gamma') -\delta_K(r)) \, \leq \, \epsilon t_n\,.
$$
Finally, we estimate the fourth term.  We distinguish two cases: in case $(a)$ the total time-length of the part
of trajectory $\gamma''_w \in \Gamma''_w$ inside $U(r)^c \cap \pi^{-1} (K)$ is at least $\sigma_K\in (0,1)$ times the total time length of $\gamma''_w \in \Gamma''_w$;  in case $(b)$  the total time-length of the part
of trajectory inside $U(r)^c \cap \pi^{-1} (K)$ is at most $\sigma_K$ times the total time length of 
$\gamma''_w \in \Gamma''_w$, hence the time-length of the part of the Teichm\"uller trajectory outside $K$ in moduli space is at least $ 1-\sigma_K$ times the total time length of the arc $\gamma''_w$.

The total time-length of trajectories $\gamma''_w$ which are in case $(a)$ is bounded above by 
$\sigma_K^{-1}$ times the total  time that the trajectory spends in $U(r)^c \cap  \pi^{-1}(K)$.  

For case $(b)$, let 
$t_{n}(r)$ denote the total time-length of those Teichm\"uller trajectories, starting and ending in $K$, 
of length at least $\delta_K(r)$, which spend at least a fraction $1-\sigma_K$ of their time outside of $K$ 
up to time $t_n>0$.  Since  $\delta_K(r) \to +\infty$ as $r\to 0^+$, there exists $\sigma_K>0$ such that 
by the Birkhoff ergodic theorem, for any Birkhoff generic point for the Teichm\"uller flow, we have
$$
\lim_{r\to 0^+} \sup_{n\geq 0}   \frac{t_n(r)}{t_n}  =0 \,.
$$
We therefore define $\epsilon_{K}: (0, r_K) \to (0,1)$ as
$$
\epsilon_{K}(r):=  \sup_{n\in \N}    \frac{t_n(r)}{t_n} \,.
$$
Finally we have the estimate 
$$
\sum_{\gamma''\in \Gamma''_w}  \max (0, \tau(\gamma'') -\delta_K(r)) \leq  C_K (\epsilon +\epsilon_{K}(r) )t_n\,.
$$
\smallskip
The estimates claimed in formula \eqref{eq:est_claim}  are thus proved.

\end{proof}

\begin{theorem}   \label{thm:HD} Let $h \in \mathcal H (\kappa)$ be any Abelian differential which is forward Birkhoff generic for the Teichm\"uller flow and  Oseledets regular for the Kontsevich--Zorich cocycle.  For any affine subspace $V \subset H^1_h(M, \T)$, parallel to a linear subspace $V_{0}\subset H^1_h(M, \R)$ which is transverse to the central stable space $E^{cs}(h)$, the Hausdorff  dimension of the set $V \cap W^s_{K} (h)$ is equal to $0$. In fact, for any $\delta >0$ there exists an open neighborhood $U\subset H_\kappa^1(M, \T)$ of the zero section and $\epsilon>0$ such that the Hausdorff dimension of $V \cap W^s_{K,U} (h,\epsilon)$ is at most $\delta$. 
\end{theorem}

\begin{proof}  Recall that for any sequence $({t}_n)$  and for any $\epsilon>0$,  the sets
$W^s_{K,U, n} (h, \epsilon)$ have been defined in formula \eqref{eq:Wnsets} as 
$$
W^{s}_{K,U, n}(h,\epsilon) = \{ c \in H^1_h(M, \T)  \,\vert \, \frac{\int_0^{t_n}
\chi_K(g_t(h))  \chi_U (g_t (h,c) ) dt }{ \int_0^{t_n}
\chi_K(g_t(h))  dt} \geq 1-\epsilon \}\,.
$$
By definition, there exists a diverging sequence $({t}_n)$  of return times of the forward  Teichm\"uller 
orbit  $\{g_t(h) \vert t>0\}$ to the compact subset $K \subset 
\mathcal H (\kappa)$, such that for any $\epsilon\in (0, 1)$, we have the inclusion
$$
W^{s}_{K,U}(h)  \subset \limsup_{n \to \infty} W^{s}_{K,U, n}(h, \epsilon) = \cap_{n\in \N} \cup_{m\geq n} 
W^{s}_{K,U,m}(h,\epsilon)\,.
$$
By Lemma~\ref{lemma:covering} there exists a function $\epsilon_K (r)$ with  $\lim_{r\to 0^+}  \epsilon_K (r)=0$ such that the set $W^{s}_{K,U(r),n}(h,\epsilon)\cap V$ can be covered by at most $N_n(r, \epsilon)$ balls of radius at most $R_n(r,\epsilon)$ such that $N_n(r, \epsilon)$ and $R_n(r,\epsilon)$ satisfy the bounds in formula \eqref{eq:est_claim}.
By those estimates we have
$$
\begin{aligned}
 \sum_{m\geq n}  N_m  R_m^\delta &\leq \sum_{m\geq n}  e^{C_K d_u (\epsilon +\epsilon_K(r)) {t}_m }
 e^{- C^{-1}_K \mu_\kappa (K) \delta (1-\epsilon) \nu t_m}  \\ &=   \sum_{m\geq n} 
 e^{- (C^{-1}_K \mu_\kappa(K) \delta (1-\epsilon) \nu  -C_Kd_u(\epsilon +\epsilon_K(r)) t_m}  \,.
\end{aligned} 
$$
Let then $\delta > C_K^2 \mu_\kappa (K)^{-1} (\epsilon +\epsilon_K(r))  d_u [(1-\epsilon) \nu]^{-1}$. Since it is possible to assume $t_n \geq n$ (for large $n\in\N$), under this assumption we have
\begin{multline*}
\limsup_{n\to +\infty}  \sum_{m\geq n} 
 e^{-  [ C^{-1}_K \mu_\kappa (K)\delta (1-\epsilon) \nu  -C_Kd_u(\epsilon+\epsilon_K(r) ]  t_m} \\ \leq \lim_{n\to +\infty} \sum_{m\geq n}  
 e^{- [C^{-1}_K \mu_\kappa(K) \delta (1-\epsilon) \nu  -C_Kd_u (\epsilon +\epsilon_K(r)) ] m}  \\ \leq  \lim_{n\to +\infty}  \frac{ e^{- [C^{-1}_K \mu_\kappa(K) \delta (1-\epsilon) \nu  -C_Kd_u (\epsilon +\epsilon_K(r))]  n}} { 1-  e^{- (C^{-1}_K\delta (1-\epsilon) \nu  -C_Kd_u (\epsilon+ \epsilon_K(r)) } }   = 0 \,.
\end{multline*}
By Lemma~\ref{lemma:Hdim} we derive  the following Hausdorff dimension bound 
$$\text{\rm H-dim}\left (W^{s}_{U(r)}(h,\epsilon) \cap V  \right ) \leq C_K^2 \mu_\kappa (K)^{-1} (\epsilon +\epsilon_K(r))  d_u [(1-\epsilon) \nu]^{-1}\,.$$  
For any given $\delta >0$ there exist $r>0$ and $\epsilon >0$ such that
\begin{multline*}
\text{\rm H-dim}\left (W^{s}(h) \cap V  \right )  \leq \text{\rm H-dim}\left (W^{s}_{U(r)}(h,\epsilon) \cap V  \right )\\ \leq C_K^2 \mu_\kappa (K)^{-1} (\epsilon +\epsilon_K(r))  d_u [(1-\epsilon) \nu]^{-1} < \delta \,,
\end{multline*}
hence the  Hausdorff dimension of  $W^{s}(h)\cap V$  is equal to zero, as stated.
\end{proof}
We conclude the section with growth estimates for the twisted cocycle. 

\smallskip
Let $\Lambda_\kappa: H^1_\kappa (M, \T)  \to [0,1)$ be the function defined in formula \eqref{eq:Lambda_kappa}. We recall that, by Lemma~\ref{lemma:growth_1}, the  ergodic integrals of
$\Lambda_\kappa$ bound the logarithm of the norms of the twisted cocycle. We also recall that 
$\Lambda_\kappa <1$ everywhere except on the zero section of the bundle $H^1_\kappa (M, \T)$
and it is continuous by Lemma \ref{lemma:Lambda_gap}.

\begin{lemma} \label{lemma:growth_2} Let $h \in \mathcal H (\kappa)$ be any Abelian differential which is forward Birkhoff generic for the Teichm\"uller flow and 
forward Oseledets regular for the Kontsevich--Zorich cocycle with respect to the Masur--Veech measure. There exists a set $\R W^s(h) \subset H^1(M, \R)$ of Hausdorff dimension 
$g+1$ (which depends only on $[\im(h)]\in H^1(M,\Sigma;\R)$) such that  if  $[\re (h) ]\not \in \R W^s(h)$ then  there exists a  constant $\alpha_h>0$,  and for all $\lambda \in \R\setminus \{0\}$  there exists a constant $C(h,\lambda)>0$, such that, for all $c\in H^1_{h,\lambda}(M, \C)$ and for all $t>0$ we have
$$
\exp\left(  \int_0^{t} \Lambda_\kappa ( g_\tau [h,\lambda \re(h)] ) d\tau \right) \leq  C(h,\lambda)    
e^{(1-\alpha_h)t}\,.
$$
\end{lemma}
\begin{proof} Let $K\subset \mathcal H (\kappa)$ be a non-empty compact subset and let $\R W^{s}(h)$ denote the
set defined as
$$
\R W^{s}(h):= \{ \R \cdot c  \vert  c \in W^{s}_{K}(h) \}\,.
$$
For all $h\in \mathcal H(\kappa)$, Osedelets regular with respect to the Masur--Veech measure, the central stable space
(which is equal to the stable space) has codimension equal to the genus $g\geq 2$. By Theorem~\ref{thm:HD}   we have that, for any unstable affine subspace $V$, the Hausdorff dimension of the set $W^{s}_{K}(h) \cap V$ is equal to zero, hence (by properties of Hausdorff dimension) it follows that $W^{s}_{K}(h)$ has Hausdorff dimension equal to $g$, and $\R W^{s}(h)$ has Hausdorff dimension equal to $g+1 <2g$. 

Let $\mathcal U$ denote, as above, the family of all neighborhoods of the zero section of the bundle $H^1(M, \T)$. Since  by definition 
$$
W^{s}_{K}(h)= \bigcap_{U \in \mathcal U}  \bigcap_{\epsilon\in (0,1)}  W^{s}_{K,U}(h, \epsilon)\,,
$$
it follows that for any $[\re (h)]\not \in \R W^{s} (h)$ there exists $U\in \mathcal U$
and $\epsilon>0$ such that $\lambda [\re(h)] \not\in W^s_{K,U}(h, \epsilon)$ for all $\lambda \in \R$.
  It follows that there exists a time $t_{h}(\lambda)>0$ such 
that, for all $t \geq t_h(\lambda)$, we have
$$
\frac{1}{t} \int_0^{t} \chi_U (g_\tau (h, \lambda [\re (h)]) ) d\tau \leq 1-\epsilon  \,.
$$
The statement then follows from Lemma~\ref{lemma:growth_1} and Lemma~\ref{lemma:Lambda_gap}.
\end{proof}

\begin{lemma}  
\label{lemma:small_meas}
There exist $r_0  \in (0, r_K)$, $\epsilon_0 \in (0, 1)$  such that the following holds. 
For every $r\in (0, r_0)$, every $\epsilon\in (0, \epsilon_0)$, there exists $\chi >0$, and 
for every forward Birkhoff generic and Oseledets regular $h\in \mathcal H(\kappa)$, there exists a constant 
$C_\kappa(\im (h), r, \epsilon)>0$ such that, for all $n\in \N$, 
$$
\text{ \rm vol} \left( W_{K, U(r),n} (h, \epsilon) \right)  \leq   C_\kappa(\im (h), r, \epsilon) e^{-\chi t_n}\,.
$$
\end{lemma} 
\begin{proof}
By Lemma~\ref{lemma:covering} there exist constants  $C_K>1$, $\nu>0$ and there exists a function $\epsilon_{K}:(0, r_K) \to (0,1)$ 
with  $\lim_{r\to 0^+} \epsilon_K(r)=0$ such that the following holds.  For any affine unstable subspace $V$ and for $n$ large enough, 
the set $W_{K, U(r),n} (h, \epsilon) \cap V$ is covered by $N_n(r,\epsilon)$ balls of radius $R_n(r,\epsilon)$
with 
$$
R_n(r,\epsilon) \leq  e^{ -\frac{1}{2}  C^{-1}_K \mu_\kappa(K) (1-\epsilon)\nu   t_n}  \quad
\text{ and } \quad N_n(r,\epsilon) \leq e^{2C_K d_u(\epsilon +\epsilon_{K}(r))t_n   } \,.
$$
There exists $r_0\in (0, r_K)$ and  $\epsilon_0 \in (0,1)$  with 
$$
\chi :=   \frac{1}{2}  C^{-1}_K \mu_\kappa(K) (1-\epsilon_0)\nu   - 2C_K \left(\epsilon_0 + \sup_{0\leq r \leq r_0} \epsilon_{K}(r)\right) >0\,.
$$
The estimate claimed in the statement then holds for every $r\in (0, r_0)$ and every $\epsilon\in (0, \epsilon_0)$
with the above choice of $\chi>0$. The argument is therefore complete. 

\end{proof} 

\begin{lemma}   \label{lemma:growth_3} There exist constants $\alpha_\kappa$, $\alpha'_\kappa$
and $N_\kappa>0$ 
such that, for almost all Abelian differential $h \in \mathcal H (\kappa)$ with respect to the Masur--Veech
measure, there exists a constant $C_\kappa (h)>0$ such that, for all $n\in \N$ and for all 
$\lambda \in \R$ with $\vert \lambda \vert \geq e^{-\alpha'_\kappa t_n}$, we have
$$
\exp \left(  \int_0^{t_n} \Lambda_\kappa (g_t [h,\lambda \re(h)]  ) dt \right)\leq  C_\kappa (h)  
 (1+ \lambda^2)^{\frac{N_\kappa}{2} } e^{(1-\alpha_\kappa) t_n}\,.
$$
\end{lemma} 

\begin{proof} 
Let us fix $r>0$ and $\epsilon>0$ such that Lemma~\ref{lemma:small_meas} holds: there exists
$\chi >0$ and $C_\kappa(\im (h), r, \epsilon)$ such that, for all $n\in \N$,
$$
\text{ \rm vol} \left( W_{K, U(r),n} (h, \epsilon) \right)  \leq   C_\kappa(\im (h), r, \epsilon) e^{-\chi t_n}\,.
$$
Let $\chi' \in (0, \chi/2g)$ and let $\mathcal B_n$ denote the set of Abelian differentials 
$h  \in \mathcal H(\kappa)$, such that $h$ is forward Birkhoff regular and Oseledets generic, and 
in addition 
$$
\{ \lambda [\re(h)] \in H^1(M, \T) \vert  \lambda \in [e^{-\chi' t_n}, e^{\chi' t_n}] \} \, \cap  \, W_{K, U(r),n} (h, \epsilon) =\emptyset\,. 
$$
By definition, there exists $\alpha_\kappa:= \alpha(r, \epsilon)>0$  such that, whenever $h$ is forward Birkhoff regular and Oseledets generic, but $h\not\in \cup_{m\geq n} \mathcal B_m$, for all $m\geq n$ and for $\vert \lambda \vert  \in [e^{-\chi' t_m}, e^{\chi' t_m}]$, the class $\lambda [\re(h)]$ does not belong to $W_{K, U(r),m} (h, \epsilon)$, hence 
by Lemma~\ref{lemma:growth_1} and Lemma~\ref{lemma:Lambda_gap} (as in the proof of Lemma~\ref{lemma:growth_2}) we derive the bound
$$
\exp\left(  \int_0^{t_m} \Lambda_\kappa ( g_t [h,\lambda \re(h)] ) dt \right) \leq  C(h,\lambda)    e^{(1-\alpha_\kappa) {t_m}}\,.
$$
There exists therefore a constant $C_\kappa(\im (h))>0$ such that,
for all $m\geq n$ and for $\lambda \in \R$  with $\vert \lambda \vert \geq e^{-\chi' t_m}$ we have 
$$
\exp\left(  \int_0^{t_m} \Lambda_\kappa ( g_t [h,\lambda \re(h)] ) dt \right)  \leq 
C_\kappa(\im (h)) (1+ \lambda^2)^{ \frac{\alpha_\kappa}{2\chi'}} 
e^{(1-\alpha_\kappa) t_m}\,.
$$
In addition, for all $n\in \N$,  let $[\re (\mathcal B_n)] :=\{ [\re (h)] \in H^1(M, \T) \vert  h\in \mathcal B_n\}$. We have the following volume estimate 
$$
\text{ \rm vol} \left(\cup_{m\geq n} [\re (\mathcal B_m)] \right) \leq C_\kappa (\im (h),r,\epsilon) e^{ (2g \chi' -\chi) n}\,,
$$
hence the set $\mathcal B  = \cap_{n\in \N}  \cup_{m\geq n} \mathcal B_m$ has Masur--Veech measure zero.

\end{proof}

\section{Transfer cocycles}

In this section we prove a ``spectral gap'' result for the extension of the twisted cocycle to a bundle
of $1$-currents. The argument follows closely that given in  \cite{AtF08}, section 4.2, where a similar result was proved for the extension of the Kontsevich--Zorich cocycle to $1$-currents.

For any Abelian differential $h\in \mathcal H(\kappa)$ and for any real closed $1$-form $\eta\in Z^1(M,\R)$, let  $Z^{-1}_{h,\eta} (M)= Z^{-1}_{\eta} (M)$ denote the
subspace of $d_\eta$-closed  $1$-currents, that is, the space of $1$-currents $C\in \Omega^1 H^{-1}_h(M)$
such that $d_\eta C=0$. Let $E^{-1}_{h,\eta} (M)$ denote the subspace of $d_\eta$-exact currents, that
is, currents $C$  such that there exists $U \in L^2_h(M)$ with $C=d_\eta U$.  Let $ \Omega^1 H^{-1}_\kappa(M)$ denote the bundle with fiber at any $[h, \eta]\in H^1_\kappa(M, \T) $ the space $\Omega^1 H^{-1}_h(M)$ of $1$-currents. Let $\mathcal Z^{-1}_\kappa (M)$ and  $\mathcal E^{-1}_{\kappa} (M) \subset \mathcal Z^{-1}_\kappa (M)$ denote
the sub-bundles of twisted closed and twisted exact currents with fiber at $[h,\eta]$ the spaces $Z^{-1}_{h,\eta} (M)$ and $E^{-1}_{h,\eta} (M)$, respectively.

The Teichm\"uller flow extends to a cocycle on the bundle $\mathcal Z^{-1}_\kappa (M)$ over 
$H^1_\kappa(M, \T)$. The cocycle is defined by parallel transport with respect to the projection of the trivial connection on the product bundle 
$$
\hat {\mathcal H}(\kappa) \times  \{ (\eta, C) \vert  C \in Z^{-1}_\eta(M)  \}\,.
$$
By definition of the de Rham cohomology, the quotient bundle   
$$\mathcal Z^{-1}_\kappa (M)/ \mathcal E^{-1}_\kappa (M) \,,$$
is isomorphic to the twisted cohomology bundle $\mathcal T^1_\kappa(M, \C)$, hence
 the transfer cocycle over the Teichm\"uller flow on the bundle $\mathcal Z^{-1}_\kappa (M)$
projects to the twisted cocycle on the twisted cohomology bundle $\mathcal T^1_\kappa(M, \C)$. 
It follows that the set of Lyapunov exponents of the transfer cocycle on $\mathcal Z^{-1}_\kappa (M)$
is the union of the set of Lyapunov exponents of the twisted cocycle on $\mathcal T^1_\kappa (M, \C)$ with that
of the restriction of the transfer cocycle to the sub-bundle of twisted exact currents $\mathcal E^{-1}_\kappa (M)$.

\begin{lemma} 
\label{lemma:L_exact}
The  restriction of the transfer cocycle  to the subbundle 
$\mathcal E^{-1}_\kappa (M)$ of twisted exact currents has a continuous invariant norm $\mathcal L_\kappa$, hence  the unique Lyapunov exponent of the cocycle is equal to $0$ and has infinite multiplicity. In addition, for all 
$(h, \eta) \in \mathcal H^1_\kappa(M, \T)$ and for all  $C \in \mathcal E^{-1}_{h, \eta}  (M)$ we have
$$
\mathcal L_\kappa (C) \leq  (1+ \vert \eta \vert_{\Omega^1 L^2_h(M)})    \vert C \vert_{\Omega^1H^{-1}_h(M)} 
$$
\end{lemma} 
\begin{proof}  By definition, for any $h\in \mathcal H(\kappa)$, for any $\eta \in H^1_h(M,\T)$ and for any twisted exact $1$-current $C\in E^{-1}_{h, \eta} (M)$ there exists a unique function $U_C\in L^2_h(M)$ of zero average such that $C =d_\eta U_C$.  
The function ${\mathcal L}_\kappa: \mathcal E^{-1}_\kappa (M) \to \R^+$ defined as
$$
{\mathcal L}_\kappa ([h,\eta, C]) = \vert  U_C\vert_{L^2_h(M)} \,, \quad \text{ for all }  C \in 
\mathcal E^{-1}_\kappa (M)\,, $$
 is invariant under the twisted cocycle, hence it defines a continuous Lyapunov norm on $\mathcal E^{-1}_\kappa (M)$.
 \begin{footnote} {A Lyapunov norm is a norm which varies exponentially in time at a rate approximately equal to the Lyapunov exponent. 
 For a zero Lyapunov exponent, it is by definition an invariant norm.}
  \end{footnote}
In fact, the cocycle is defined by parallel transport with respect to the projection of the trivial connection on
the product bundle 
$$
\hat {\mathcal H}(\kappa) \times \{ (\eta, C) \vert (\eta, C) \in H^1_\kappa(M, \T) 
\times Z^{-1}_\eta (M) \}
$$
and the norm ${\mathcal L}_\kappa $ is the projection of a $\Gamma_g$-equivariant norm on the product bundle,
which is invariant under the $SL(2, \R)$-action on $\hat {\mathcal H}(\kappa)$.

Finally, for any $C = d_\eta U_C \in E^{-1}_{h,\eta} (M)$ and all $\alpha \in \Omega^1 H^\infty (M)$ we have
$$
\begin{aligned}
\vert \<C, \alpha\>\vert  &= \vert \<U_C,  d_\eta \alpha\>\vert  = \vert \<U_C, d\alpha + \eta \wedge \alpha\>\vert
\\ &\leq  \mathcal L_\kappa (C) (1+ \vert \eta \vert_{\Omega^1 L^2_h(M)})  \vert \alpha \vert_{\Omega^1 H^1_h(M)}\,,
\end{aligned}
$$
thereby completing the argument.
\end{proof}

Let $\zeta^{-1}_\kappa: \Omega^1 H^{-1}_\kappa(M) \to \R$ be the (continuous) distance functions to the Hilbert 
sub-bundle $\mathcal Z^{-1}_\kappa(M)$ of twisted closed currents defined as follows:  for each $h \in \mathcal H(\kappa)$ and all $\eta \in H^1_h (M, \T)$, the restriction $\zeta^{-1}_\kappa \vert  \Omega^1 H^{-1}_h(M)$ is equal to the distance function from the closed subspace $Z^{-1}_{h, \eta} (M) \subset  \Omega^1 H^{-1}_h(M)$ with respect to the Hilbert space metric on $ \Omega^1 H^{-1}_h(M)$.  

For any compact set 
$K\subset \mathcal H(\kappa)$  and any $\zeta >0$, we introduce the following closed, $g_\R$-invariant subsets 
$\mathcal A_K(\zeta)$ of the bundle $ \Omega^1 H^{-1}_\kappa (M)$.   For every Abelian differential $h\in {\mathcal H}(\kappa)$, the intersection of the set
$\mathcal A_K(\zeta)$ with the fiber $ \Omega^1 H^{-1}_h (M)$ of the bundle $ \Omega^1 H^{-1}_\kappa (M)$ of $1$-currents is defined as follows:
\begin{equation}
\label{eq:GammaCdelta}
\mathcal A_K(\zeta) \cap  \Omega^1 H^{-1}_h (M) =\{ A \in   \Omega^1 H^{-1}_h(M) \,\vert \,
g_t (h) \in K \Rightarrow  \zeta^{-1}_\kappa \left( g_t (A) \right) \leq \zeta \} \,.
\end{equation}

In other terms, the fibered subset $\mathcal A_K(\zeta)$ contains all currents which stay at bounded distance ($\leq \zeta$) from the sub-bundle of twisted closed currents for all returns of the Teichm\"uller orbit to a given compact set $K\subset \mathcal H(\kappa)$. The relevant examples of non-closed currents in 
$\mathcal A_K(\zeta)$ are given by currents of twisted integration along orbits of the horizontal translation flow 
in $(M,h)$. In fact, as we have proved in section~\ref{sec:Twist_Int}, for any compact set $K\subset \mathcal H(\kappa)$ there exists $\zeta_K >0$ such that any  current represented by a twisted integral along an orbit of the horizontal translation flow in $(M,h)$ belongs to $\mathcal A_K(\zeta)$ for 
$\zeta\geq \zeta_K$ (up to projection on the codimension one sub-bundle of currents perpendicular to the sub-bundle $\{\C\im(h) \vert h\in \mathcal H(\kappa)\}$)  .

The core technical result of this paper is the following `spectral gap'  lemma for the restriction of
the distributional cocycle $\{g_t \vert t\in \R\}$ to any invariant set $\mathcal A_K(\zeta)
\subset  \Omega^1 H^{-1}_\kappa(M)$.

 For any $h\in\mathcal H(\kappa)$, let $t_0=0$ and let $\{{t}_n \vert n\in\N\}$ denote a non-decreasing sequence of visiting times of the forward orbit $\{g_{t}(h) \vert t \geq 0\}$ to a given compact set $K\subset \mathcal H(\kappa) $. 
 We will regard any current  $A \in \Omega^1 H^{-1}_h(M) $ as  an element of the vector bundle  $\Omega^1 H^{-1}_\kappa(M)$ of currents over the moduli space of Abelian differentials. 
 
 In particular we have
$$
\vert A \vert _{-1} =  \vert A \vert _{ H^{-1}_h(M)} \,.
$$
\begin{lemma} 
\label{lemma:iterativest}
For any compact set $K\subset \mathcal H(\kappa)$, there exists a constant $C_K>1$ such that, for any $\zeta>0$,
for any $[h, \eta] \in H^1_\kappa(M, \T)$, for any $A \in  \mathcal A_K (\zeta)$ and for all $n\in \N$, the following estimate holds:
\begin{equation}
\begin{aligned}
\label{eq:iterativest}
\vert A \vert_{-1} \leq & C_K \, (1+\zeta) 
 (1+ \vert  g_{t_n} (A) \vert_{-1} ) \\  &\times \exp \left(\int_0^{{t}_n} \Lambda_\kappa (g_t [h, \eta] )dt \right)
   \left( \sum_{j=0}^{n-1}  e^{2({t}_{j+1}-{t}_j)} \right)^3\,.
   \end{aligned}
 \end{equation}
\end{lemma}

\begin{proof} The argument follows closely the proof of Lemma 4.5 in~\cite{AtF08}.  

For all $n\in \N$, let  $[h_n, \eta_n] = g_{{t}_n} [h, \eta]$ with
$h_n = g_{{t}_n} (h) \in K$. For each $j \in \N$, since $Z^{-1}_{h_j, \eta_j} (M)$ is closed in 
$\Omega^1 H^{-1}_{h_j} (M)$, there exists an orthogonal decomposition,
\begin{equation}
\label{eq:gammasplit}
 g_{{t}_j} (A) = Z_j \, + \, R_j  \,, \quad  \text{ \rm with } \, Z_j \in Z^{-1}_{h_j, \eta_j}(M) \,,\, \, R_j \perp Z^{-1}_{h_j, \eta_j}(M)\,,
\end{equation}
and, since $A  \in \mathcal A_K(\zeta)$ and $h_j\in K$, the following bound holds:
\begin{equation}
\label{eq:Gammadeltabound}
\vert R_j \vert_{-1} \leq  \zeta \,.
\end{equation}
For each $j\in \N$, let $\pi_j: \Omega^1 H^{-1}_{h_j} (M) \to Z^{-1}_{h_j,\eta_j} (M)$ denote the orthogonal projection  and let $\tau_j = {t}_{j+1} -{t}_{j}$. By~\eqref{eq:gammasplit} and by orthogonal projection on the $g_t$-invariant bundle $\mathcal Z^{-1}_\kappa(M)$ the following recursive identity holds:
\begin{equation}
\label{eq:recid1}
Z_{j} = g_{-\tau_j} (Z_{j+1}) \, + \,    \pi_{j}\circ  g_{-\tau_j}(R_{j+1})  \in 
Z^{-1}_{h_j,  \eta_j}(M)\,.
\end{equation}
By definition of the Sobolev norms and by the Teichm\"uller invariance of the $L^2$ norms, it is immediate to prove (see for instance \cite{AtF08}, formula $(3.24)$)  that
\begin{equation}
\label{eq:cocycleestimate}
\Vert   g_t \vert _{\Omega^1 H^{-1}_h} \Vert  \leq  e^{2\vert t \vert}  \,, \quad \text{\rm for all } \, (h,t) \in \mathcal H(\kappa) \times \R \,.
\end{equation}
Thus by the bound in formula~\eqref{eq:Gammadeltabound}, it follows that
\begin{equation}
\label{eq:remainderbound1}
\vert \pi_{j}\circ  g_{-\tau_j}(R_{j+1}) \vert_{-1}  \leq 
\vert g_{-\tau_j}(R_{j+1}) \vert_{-1}  \leq e^{2\tau_j} \, \zeta \,.
\end{equation}
By projection on the twisted cohomology bundle $\mathcal T^1_\kappa(M,\C)$ and by compactness, we derive from the identity \eqref{eq:recid1} and from the bound \eqref{eq:remainderbound1} that there exists $C^{(1)}_K >1$ such that, with respect to the Hodge norm,
\begin{equation}
\label{eq:recbound1}
\Vert [Z_j] - g_{-\tau_j} ([Z_{j+1}]) \Vert_{h_j,\eta_j}  \leq C^{(1)}_K \, \zeta\,e^{2\tau_j} \,.
\end{equation}
 By Lemma \ref{lemma:growth_1} and by formula~\eqref{eq:recbound1}  we have
\begin{equation}
\label{eq:recbound2}
 \Vert   [Z_j] \Vert_{h_j, \eta_j}  \leq   \Vert   [Z_{j+1}] \Vert_{h_{j+1}, \eta_{j+1} } 
 \exp \left(\int_{{t}_j}^{{t}_{j+1}}
 \Lambda_\kappa (g_t[h, \eta] ) dt  \right)
   \, +\,   C^{(1)}_K \, \zeta \,e^{2\tau_j} \,.
\end{equation}
For each $\ell \in \N$, it follows by reverse iteration on $1\leq j < \ell$ that
$$
\Vert   [Z_j] \Vert_{h_j, \eta_j} \leq \left ( \Vert   [Z_\ell] \Vert_{h_\ell, \eta_\ell} + C^{(1)}_K \, \zeta  \sum_{i=j}^{\ell-1}  \, e^{2\tau_i- \int_{{t}_i}^{{t}_{\ell}} \Lambda_\kappa (g_t[h, \eta] ) dt   } 
      \right)   \exp \left(\int_{{t}_j}^{{t}_{\ell}} \Lambda_\kappa (g_t[h, \eta] ) dt  \right) \,,
$$
which, since $\Lambda_\kappa \geq 0$ and $\tau_i\geq 0$ for all $i\in \N$, implies the estimate
\begin{equation}
\label{eq:Hbound}
 \Vert   [Z_j] \Vert_{h_j, \eta_j}  \leq  C^{(1)}_K  (1+\zeta) (1+ \Vert  [Z_\ell] \Vert_{h_\ell, \eta_\ell}  )  \exp \left(\int_{{t}_j}^{{t}_{\ell}}
 \Lambda_\kappa (g_t[h, \eta] ) dt  \right)  \sum_{i=j}^{\ell-1}  \,e^{2\tau_i} \,.
\end{equation}
By the definition of the Hodge norm, for each $j\in \N$, there exists a twisted harmonic form $\omega_j \in
Z^1_{h_j, \eta_j}(M)$ such that 
\begin{equation}
\label{eq:hrepr}
E_j = Z_j -\omega_j  \in E^{-1}_{h_j, \eta_j} (M)  \quad \text{ and }   \quad 
\vert \omega_j \vert_{-1} \leq \Vert   [Z_j] \Vert_{h_j, \eta_j} \,.
\end{equation}
 For each $j\in \N$, let us define
 \begin{equation}
 \label{eq:Eid1}
 F_j= E_j - g_{-\tau_j} (E_{j+1})  \,\, \in \,\, E^{-1}_{h_j, \eta_j}(M)\,.
 \end{equation}
By the recursive identity~\eqref{eq:Eid1}  the following bound holds with respect to the Lyapunov norm 
${\mathcal L}_\kappa$ on the bundle of twisted exact currents:
\begin{equation}
\label{eq:recbound3}
{\mathcal L}_{h_j} ( E_j ) \leq  {\mathcal L}_{h_{j+1}} (E_{j+1})  +  
{\mathcal L}_{h_j} (  F_j )\,.
\end{equation}
In fact, the restriction of the distributional cocycle $\{g_t\vert t\in\R\}$ to the bundle $\mathcal E^{-1}_\kappa(M)$
of twisted exact currents is isometric with respect to the norm ${\mathcal L}_\kappa$.  For each $\ell\in \N$, we derive from 
\eqref{eq:recbound3}  by (reverse) induction on $1\leq j < \ell$ that
\begin{equation}
\label{eq:Ebound}
{\mathcal L}_{h_1} ( E_1)  \leq  {\mathcal L}_{h_\ell}  (E_\ell) + \sum_{j=1}^{\ell-1}
{\mathcal L}_{h_j} ( F_j ) \,.  
\end{equation}
By the splitting \eqref{eq:gammasplit}  it follows that 
$$
g_{-\tau_j} (Z_{j+1} + R_{j+1}) = g_{-\tau_j} g_{t_{j+1}} (A) = g_{t_j}(A) = Z_j + R_j\,,
$$
hence by the identity \eqref{eq:hrepr} 
$$
g_{-\tau_j} (E_{j+1}+ \omega_{j+1} + R_{j+1}) = E_j + \omega_j + R_j\,.
$$
Thus, by the definition in formula~\eqref{eq:Eid1}, we conclude  that
 \begin{equation}
 \label{eq:Eid2}
 F_j=  g_{-\tau_j} (\omega_{j+1}+ R_{j+1}) - (\omega_j +R_j)\,,
\end{equation}
hence by compactness, be Lemma \ref{lemma:L_exact},  and by formulas~\eqref{eq:Gammadeltabound} and~\eqref{eq:hrepr}, there exists a constant  $C^{(2)}_K>1$ such that 
$$
\begin{aligned} 
{\mathcal L}_{h_j} ( F_j  ) & \leq   C^{(2)}_K \left( \vert g_{-\tau_j} (\omega_{j+1}+ R_{j+1}) \vert_{-1}  
+\vert \omega_j +R_j\vert_{-1} \right)  \\ &\leq   C^{(2)}_K \left( e^{2\tau_j} ( \Vert [Z_{j+1}] \Vert_{h_j, \eta_j} + \zeta) +   \Vert [Z_{j}] \Vert_{h_j, \eta_j} + \zeta\right) \,,
\end{aligned}
$$
hence, by \eqref{eq:Hbound}, there exists a constant $C^{(3)}_K>0$ such that, for all $\ell>1$, we have 
\begin{equation}
\label{eq:remainderbound2}
\begin{aligned}
\sum_{j=1}^{\ell-1} {\mathcal L}_{h_j} ( F_j  ) \leq C^{(3)}_K\, (1+\zeta)  &(1+\Vert  [Z_\ell] \Vert_{h_\ell, \eta_\ell}) \\ &\times \exp \left(\int_{{t}_1}^{{t}_\ell}  \Lambda_\kappa (g_t[h, \eta] ) dt  \right) \left( \sum_{j=1}^{\ell-1}  \,e^{2\tau_j} \right)^2\,.
\end{aligned}
\end{equation}
By the splitting \eqref{eq:gammasplit} and by formulas~\eqref{eq:Gammadeltabound}, \eqref{eq:Hbound}, \eqref{eq:hrepr}, \eqref{eq:Ebound} and \eqref{eq:remainderbound2}, there exists a constant $C^{(4)}_K >0$
such that for all $\ell>1$,
\begin{equation}
\label{eq:finalest}
\begin{aligned}
\vert g_{t_1}(A) \vert_{-1} \leq C^{(4)}_K\, (1+\zeta) \, 
&(1+\vert g_{t_\ell}(A) \vert_{-1}) \,  \\ &\times \exp\left( \int_{{t}_1}^{{t}_\ell} \Lambda_\kappa(g_t[h,\eta]) \,dt\right)  
 \left ( \sum_{j=1}^{\ell-1}  \,e^{2\tau_j} \right)^2\,.
\end{aligned}
 \end{equation}
Finally, by the bound~\eqref{eq:cocycleestimate}, since ${t}_0=0$, 
\begin{equation}
\vert A \vert_{-1} \leq e^{2 {t}_1} \, \vert g_{t_1}(A) \vert_{-1}  \,. 
\end{equation}
\end{proof}

\section{Proof of the Main Results} 
\label{sec:Proofs}

In this section we complete the proof of the main results stated in the Introduction.

For any $h\in\mathcal H(\kappa)$, let $t_0=0$ and let $\{{t}_n \vert n\in\N\}$ denote, as above, a non-decreasing sequence of visiting times of the orbit $\{g_{t}(h) \vert t \geq 0\}$ to a compact set $K\subset 
 \mathcal H(\kappa) $. 

\begin{lemma} 
\label{lemma:twist_integral_bound}
There exist constants $\alpha_\kappa$, $\alpha'_\kappa$ and $N_\kappa>0$ 
such that, for almost all Abelian differential $h \in \mathcal H (\kappa)$ with respect to the Masur--Veech
measure, there exists a constant $C_\kappa (h)>0$ such that the following holds. For any $x\in M$  with  forward regular horizontal trajectory,  
for all $\lambda \in \R\setminus \{0\}$, for  all $n\in \N$ and for all  functions $f \in H^1_h(M)$ of zero average, we have
$$
\left\vert   \int_0^{e^{t_n} }  e^{ 2\pi \imath \lambda t}  f \circ \phi^S_t (x)  \,dt  \right\vert
\leq C_\kappa (h)  \vert f \vert_{1} \frac{(1+ \lambda^2)^{\frac{N_\kappa}{2} }}{\vert \lambda \vert^{\frac{\alpha_\kappa}{\alpha'_\kappa}}} e^{(1-\alpha_\kappa) t_n}  \left( \sum_{j=0}^{n-1}  e^{2({t}_{j+1}-{t}_j)} \right)^3.
$$
\end{lemma} 
\begin{proof}
Let $A:= A_{h, \lambda} (x, e^{t_n})$ denote the current defined, for any $1$-form $\beta$ on $M$, as
$$
A(\beta) := \int_0^{e^{t_n} }  e^{ 2\pi \imath \lambda t} \imath_S \beta  \circ \phi^S_t (x)  \,dt  \,.
$$
Let then $A^{\#}$ denote the current
$$
A^{\#} := A+ A(\eta_T) \eta_S\,.
$$
By definition, for any $1$-form $\beta$ on $M$, we have
$$
A^{\#}(\beta) := \int_0^{e^{t_n} }  e^{ 2\pi \imath \lambda t} \imath_S \alpha  \circ \phi^S_t (x)  \,dt -\,\, \frac{e^{ 2\pi \imath \lambda e^{t_n}}  -1  } { 2\pi \imath \lambda }  \int_M  \imath_S \beta \,\omega_h\,.
$$
We prove below a bound for the Sobolev norm of the current $A^{\#}$ in $\Omega^1H_h^{-1}(M)$ and derive the result for large $\lambda\in \R$ from such an estimate applied 
to $1$-forms $\beta_f = f \eta_T$ for which, whenever $f \in H^1_h(M)$ has zero average, we have
$$
A^{\#} (\beta_f) = A(\beta_f) =  \int_0^{e^{t_n} }  e^{ 2\pi \imath \lambda t}f \circ \phi^S_t (x)  \,dt  \,.
$$ 
Let $h_n := g_{t_n} (h)$ and let $(S_n, T_n)$ denote its horizontal and vertical vector fields.
By definition, the current $g_{t_n} (A)$ is given by the formula
$$
g_{t_n} (A) (\beta)=   \int_0^{1 }  e^{ 2\pi \imath e^{t_n}\lambda t} \imath_{S_n} \beta \circ \phi^{S_n}_t (x)  \,dt  \,,
$$
while the current $g_{t_n} (A^{\#})$ is given by the formula
$$
g_{t_n} (A^{\#}) (\beta)=   g_{t_n} (A)^{\#} (\beta)  =  \int_0^{1 }  e^{ 2\pi \imath e^{t_n}\lambda t} \imath_{S_n} \beta \circ \phi^{S_n}_t (x)  \,dt 
-\,\, \frac{e^{ 2\pi \imath  e^{t_n} \lambda}  -1  } { 2\pi \imath e^{t_n} \lambda   }  \int_M  \imath_{S_n} \beta \,\omega_h\,.
$$
Since $h_n \in K$, by the Sobolev trace  theorem (see Lemma~\ref{lemma:Sob_trace})  there exists a constant $C_K>0$ such that,
for all $n\in \N$, we have
$$
\vert  g_{t_n} (A^{\#}) (\beta) \vert   \leq   \vert g_{t_n} (A) (\beta)\vert \leq  \int_0^{1 }  \vert  \imath_{S_n} \beta \circ \phi^{S_n}_t (x)\vert  \,dt
\leq  C_K \vert \beta \vert_{H^{-1}_{h_n}(M)} \,,
$$
hence 
$$
\vert g_{t_n} (A^{\#})\vert_{-1} \leq  \vert g_{t_n} (A)\vert_{-1} \leq  C_K \,.
$$
 By  definition and by Lemma~\ref{lemma:dist_closed}, there exists a constant
$\zeta_K>0$ such that, for any $t>0$ with $g_{t} (h) \in K$ there exists $Z \in 
Z^{-1}_{g_{t}(h, \eta)}(M)$ such that 
$$
\vert  g_{t} (A^{\#}) -Z \vert_{-1}  \leq \zeta_K\,,
$$
hence $A^{\#} \in \mathcal A_K (\zeta_K)$. By Lemma~\ref{lemma:iterativest} with  $\eta = \lambda \re ( h)$ there exists a constant $C'_K>0$ such that we have the estimate
$$
\vert A^{\#} \vert_{-1} \leq  C'_K \exp \left(\int_0^{{t}_n} \Lambda_\kappa (g_t [h, \lambda\re(h)] )dt \right)
   \left( \sum_{j=0}^{n-1}  e^{2({t}_{j+1}-{t}_j)} \right)^3
$$
and, by Lemma \ref{lemma:growth_3} for almost all $h\in \mathcal H(\kappa)$, there exist constants $\alpha_\kappa$, $\alpha'_\kappa$
and $N'_\kappa>0$ such that, for almost all Abelian differential $h \in \mathcal H (\kappa)$ with respect to the Masur--Veech measure there exists a constant $C_\kappa (h)>0$ such that, for all $n\in \N$, for all 
$\lambda \in \R$ with $\vert \lambda \vert \geq e^{-\alpha'_\kappa t_n}$, we have
$$
\exp \left(  \int_0^{t_n} \Lambda_\kappa (g_t [h,\lambda \re(h)]  ) dt \right)\leq  C_\kappa (h)  
 (1+ \lambda^2)^{\frac{N'_\kappa}{2} } e^{(1-\alpha_\kappa) t_n}\,.
$$
For $\vert \lambda \vert \leq e^{-\alpha'_\kappa t_n}$, by the Sobolev trace theorem (see Lemma~\ref{lemma:Sob_trace}) there exists a constant $C'_\kappa (h) >0$ 
such that we have
$$
\begin{aligned}
\left\vert   \int_0^{e^{t_n} }  e^{ 2\pi \imath \lambda t}  f \circ \phi^S_t (x)  \,dt  \right\vert
&\leq  C'_\kappa (h)  \vert f \vert_1 e^{t_n}  \\ & = C'_\kappa (h)  \vert f \vert_1 e^{\alpha_\kappa t_n}   e^{(1-\alpha_\kappa)t_n}
\leq   C'_\kappa (h)  \vert f \vert_1 \vert \lambda \vert^{- \frac{\alpha_\kappa}{\alpha'_\kappa} }  e^{(1-\alpha_\kappa)t_n}\,.
\end{aligned}
$$
The argument is therefore concluded.
\end{proof}

To conclude the proof of our main results we recall a decomposition lemma from \cite{AtF08} (Lemma 5.1). The point of this decomposition lemma is that it allows to reduce
the proof of upper bounds for (twisted) integrals along arbitrary orbit arcs to upper bounds on integrals along special ``best returns'' orbit arcs, generated by the action of the renormalization (the Teichm\"uller flow), and for which  the desired bounds were derived above in Lemma~\ref{lemma:twist_integral_bound}.

\begin{lemma}
\label{lemma:chop} 
Let $h\in {\mathcal H}(\kappa)$ and let $\{t_n\}_{n \in \N}$ be any non-decreasing divergent sequence of positive real numbers. For any $(x,\mathcal T) \in M \times \R^+$ such that $x\in M$ has forward regular horizontal trajectory, the horizontal orbit segment $\gamma_{h,x}(\mathcal T)$ has a decomposition into consecutive sub-segments,
\begin{equation}
\label{eq:chop} 
\gamma_{h, x}(\mathcal T) = \sum_{\ell=1}^n \sum_{m=1}^{m_\ell} \gamma_{h, x_{\ell,m}}(\mathcal T_\ell)\,\, + \,\, 
\gamma_{h,y}(\tau) \,, 
\end{equation}
such that $n:= \max \{ \ell \in \N \vert  \mathcal T_\ell \leq \mathcal T\}$ and, for all $1\leq \ell \leq n$, 
\begin{equation}
\label{eq:chopest}
m_\ell \leq e^{t_{\ell+1} - t_\ell}\,,  \quad \mathcal T_\ell = e^{t_\ell} \quad \text{\rm and} \quad \tau \leq e^{t_1}\,.
\end{equation}
\end{lemma}

We are finally ready to complete the proof of our main theorem, stated as Theorem~\ref{thm:Twist_Erg} in the Introduction. We state it again below for the convenience of the reader.

\begin{theorem} 
\label{thm:twist_integral_bound}
There exist constants $\alpha_\kappa$, $\beta_\kappa$ and $N_\kappa>0$ 
such that, for almost all Abelian differential $h \in \mathcal H (\kappa)$ with respect to the Masur--Veech
measure there exists a constant $C_\kappa (h)>0$ such that for all $\lambda \in \R\setminus \{0\}$,  for
all $(x, \mathcal T) \in M\times \R^+$, such that $x\in M$ has forward regular horizontal trajectory, and for all  $f \in H^1_h(M)$ of zero average, 
we have the estimate
$$
\left\vert   \int_0^{\mathcal T}  e^{ 2\pi \imath \lambda t}  f \circ \phi^S_t (x)  \,dt  \right\vert
\leq C_\kappa (h)  \vert f \vert_{1} \frac{(1+ \lambda^2)^{\frac{N_\kappa}{2} }}{\vert \lambda \vert^{\beta_\kappa}} \, \mathcal T^{1-\alpha_\kappa}.
$$
\end{theorem}
\begin{proof}  Let $(t_n)$ denote a sequence of return times of the orbit $\{g_t(h)\}$ to a compact set
$K \subset \mathcal H(\kappa)$ such that $\lim_{n\to +\infty} t_n/n = \mu \not =0$. It follows
that for any $\eta\in (0, \mu)$ there exists $n_\eta\in \N$ such that  we have
$$
(\mu -\eta) n \leq  t_n  \leq  (\mu +\eta) n   \,, \quad \text{ for all } n\geq n_\eta\,.
$$
It follows in particular that there exists a constant $C(\mu,\eta,h)>0$ such that
$$
\sum_{j=0}^{n-1} e^{2(t_{j+1}-t_j)}     \leq   C(\mu,\eta,h) \,   e^{4 \eta n}  \,.
$$
In fact, for $n>n_\eta$, by geometric summation we have
$$
\begin{aligned}
\sum_{j=0}^{n-1} e^{2(t_{j+1}-t_j)}  &\leq  \sum_{j=0}^{n_\eta-1} e^{2(t_{j+1}-t_j)}   + e^{2(\mu+\eta)}  \sum_{j=n_\eta}^{n-1} e^{4\eta n}  
\\ &\leq   \sum_{j=0}^{n_\eta-1} e^{2(t_{j+1}-t_j)}   + e^{2(\mu+\eta) }   \frac{  e^{4\eta n } - e^{4n_\eta}}{e^{4\eta}-1}\,,
\end{aligned}
$$
hence the above bound holds with
$$
C(\mu,\eta,h):=  \sum_{j=0}^{n_\eta-1} e^{2(t_{j+1}-t_j)}   +    \frac{  e^{2(\mu+\eta) }}{e^{4\eta}-1}\,.
$$
From Lemma~\ref{lemma:twist_integral_bound} we derive, for all $\ell\in \{1, \dots, n\}$, the bounds
$$
\left\vert   \int_0^{\mathcal T_\ell }  e^{ 2\pi \imath \lambda t}  f \circ \phi^S_t (x_{\ell,m})  \,dt  \right\vert
\leq C_\kappa (\mu, \eta, h)  \vert f \vert_{1} \frac{(1+ \lambda^2)^{\frac{N_\kappa}{2} }}{\vert \lambda \vert^{\frac{\alpha_\kappa}{\alpha'_\kappa}}}   e^{(1-\alpha_\kappa)t_\ell + 12\eta \ell}   \,\,.
$$
Given the sequence $(t_n)$ and any $(x, \mathcal T) \in M\times \R^+$ such that $x\in M$ has forward regular horizontal trajectory, we consider the induced decomposition of a horizontal orbit segment $\gamma_{h,x}(\mathcal T)$ given by Lemma~\ref{lemma:chop}.
For all $\ell\in \{1, \dots, n\}$ and $m\in \{1, \dots, m_\ell\}$, let $\tau_{\ell,m}$ denote the time of the
point $x_{\ell,m}$ along the orbit. By the definitions $\tau_{\ell,m}=\sum_{j=1}^{\ell-1} m_j\mathcal T_j + (m-1)\mathcal T_\ell$.  Since by Lemma~\ref{lemma:chop} we have a decomposition
$$
\begin{aligned}
 \int_0^{\mathcal T}  e^{ 2\pi \imath \lambda t}  f \circ \phi^S_t (x)  \,dt  &=
 \sum_{\ell=1}^n \sum_{m=1}^{m_\ell} 
 e^{ 2\pi \imath \lambda \tau_{\ell,m}  }   \int_0^{\mathcal T_\ell}  e^{ 2\pi \imath \lambda t}  f \circ \phi^S_t (x_{\ell,m})  \,dt \\ &+ \,\,  \int_{\mathcal T-\tau} ^{\mathcal T}  e^{ 2\pi \imath \lambda t}  f \circ \phi^S_t (x)  \,dt  \,,
 \end{aligned}
 $$
 we derive the bound
$$
\left\vert  \int_0^{\mathcal T}  e^{ 2\pi \imath \lambda t}  f \circ \phi^S_t (x)  \,dt \right\vert 
\leq C_\kappa (\mu, \eta, h)  \vert f \vert_{1} \frac{(1+ \lambda^2)^{\frac{N_\kappa}{2} }}{\vert \lambda \vert^{\frac{\alpha_\kappa}{\alpha'_\kappa}}} \left( \sum_{\ell=1}^n m_\ell  e^{(1-\alpha_\kappa)t_\ell + 12\eta \ell}  + \tau \right)\,.
$$
Finally we have, since by construction $\mathcal T_n = e^{t_n} \leq \mathcal T$,  hence $n \leq  (\mu-\eta)^{-1} \log {\mathcal T}$,
\begin{multline*}
\sum_{\ell=1}^n m_\ell  e^{(1-\alpha_\kappa)t_\ell + 12\eta \ell}   \leq 
C'_\eta (\mu, \eta, h) \sum_{\ell=1}^n e^{(1-\alpha_\kappa)(\mu-\eta)\ell  + 16\eta \ell}   \\  \leq C''_\eta(\mu, \eta,h) 
e^{ [(1-\alpha_\kappa)(\mu-\eta) +16\eta] n } \leq C^{(3)}_\eta(\mu, \eta, h) \mathcal T ^{1-\alpha_\kappa + (\mu-\eta)^{-1}16\eta }\,,
\end{multline*}
which, by taking $\eta >0$ such that $(\mu-\eta)^{-1} 32\eta < \alpha_\kappa$,  implies the estimate
$$
\left\vert  \int_0^{\mathcal T}  e^{ 2\pi \imath \lambda t}  f \circ \phi^S_t (x)  \,dt \right\vert 
\leq C_\kappa (h)  \vert f \vert_{1} \frac{(1+ \lambda^2)^{\frac{N_\kappa}{2} }}{\vert \lambda \vert^{\frac{\alpha_\kappa}{\alpha'_\kappa}}} \mathcal T ^{ 1-\frac{\alpha_\kappa}{2}}\,.
$$
The argument is therefore complete. (Note that the exponent $\alpha_\kappa$, $\beta_\kappa$, and $N_\kappa$ in the statement are respectively equal to $\alpha_\kappa/2$,
$\alpha_\kappa/\alpha'_\kappa$ and $N_\kappa$ in terms of the positive constants  $\alpha_\kappa$, $\alpha'_\kappa$ and $N_\kappa$ of Lemma~\ref{lemma:twist_integral_bound}).\end{proof}

The remaining main results stated in the Introduction are easily derived (see below) from the above Theorem
\ref{thm:twist_integral_bound} (Theorem~\ref{thm:Twist_Erg} in the Introduction) and from the general 
results of section~\ref{sec:spectral_dim} below. 

\begin{proof}[Proof of Theorem~\ref{thm:Effect_Erg}]
We have a Fourier decomposition 
$$
F(x, \theta) = \sum_{n\in \Z}  f_n (x) e^{2\pi \imath n \theta} \,, \quad \text{ for } (x, \theta) \in M\times \T\,.
$$
By the Fourier decomposition we have
$$
\int_0^{\mathcal T} F\circ \Phi^{S, \lambda}(x, \theta) dt =  \sum_{n\in \Z}  e^{2\pi \imath \lambda n \theta}\int_0^{\mathcal T} e^{2\pi \imath \lambda n t}  f_n\circ \phi^S_t (x) dt\,.
$$
By Theorem~\ref{thm:twist_integral_bound} (Theorem~\ref{thm:Twist_Erg}) we have, for $n\not=0$, 
$$
\left\vert \int_0^{\mathcal T} e^{2\pi \imath \lambda n t}  f_n\circ \phi^S_t (x) dt \right\vert
\leq C_\lambda(h) (1+ n^2)^{\frac{N_\kappa-\beta_\kappa}{2}} \vert f_n \vert_1 {\mathcal T}^{1-\alpha'_\kappa}\,.
$$ 
For $n=0$, by Theorem~\ref{thm:AtF08} (see \cite{AtF08}) or, in fact, for almost all $h\in \mathcal H(\kappa)$ already by the results of~\cite{F02}, we have 
$$
\left\vert \int_0^{\mathcal T}   f_0\circ \phi^S_t (x) dt  - \mathcal T \int_{M\times \T} F d\omega_h d\theta  \right\vert
\leq C_\lambda(h)  \vert f_0 \vert_1 {\mathcal T}^{1-\alpha_\kappa}\,.
$$
By definition, for any $s \geq 0$, the Sobolev space $H^s(\T, H^1(M))$ is the completion of the space $C^\infty (M\times \T ) $ with respect to the norm 
$$
\Vert  F \Vert_{H^s(\T, H^1(M))}:= \left( \sum_{n\in \Z}   (1+ n^2)^{s} \vert f_n \vert^2_{H^1(M)}  \right)^{1/2}\,.
$$
Since, by H\"older inequality, we  have
$$
\sum_{n\in \Z} (1+ n^2)^{\frac{N_\kappa-\beta_\kappa}{2}} \vert f_n \vert_1 
\leq \left( \sum_{n\in \Z} (1+ n^2)^{N_\kappa-\beta_\kappa -s}  \right)^{1/2}  \Vert  F \Vert_{H^s(\T, H^1(M))}\,,
$$
it follows that for $s > N_\kappa-\beta_\kappa +1$ there exists a constant $C_{\kappa,s}>0$  such that
$$
\left\vert \int_0^{\mathcal T} F\circ \Phi^{S, \lambda}(x, \theta) dt   - \mathcal T \int_{M\times \T} F d\omega_h d\theta\right\vert \leq C_{\kappa,s} C_\lambda(h)
 \Vert  F \Vert_{H^s(\T, H^1(M))} {\mathcal T}^{1-\alpha''_\kappa}\,,
$$
for any $\alpha''_\kappa \leq \min (\alpha_\kappa, \alpha'_\kappa)$, which completes the proof of the theorem.
\end{proof}

\begin{proof}[Proof of Corollary~\ref{cor:spectral}]  It is an immediate consequence of Theorem~\ref{thm:twist_integral_bound} (Theorem~\ref{thm:Twist_Erg} in the Introduction) and of Lemma~\ref{lemma:spectral_1} below, which derives a lower bound on spectral dimensions from an upper bound on twisted ergodic integrals. 
\end{proof}

Corollary~\ref{cor:Eff_WM} follows from Theorem~\ref{thm:twist_integral_bound} 
(Theorem~\ref{thm:Twist_Erg} in the Introduction), the quantitative equidistribution result for translation flows stated in Theorem~\ref{thm:AtF08} (see also \cite{F02}) and Lemma~\ref{lemma:Eff_WM}, which derives
a bound on the speed of weak mixing from bounds on twisted ergodic integrals. 
\begin{proof}[Proof of Corollary~\ref{cor:Eff_WM}]
By integration by parts we have
$$
\begin{aligned}
 \int_0^{\mathcal T}  e^{2\pi \imath \lambda t}  f \circ \phi^S_t  dt &=   \int_0^{\mathcal T}  
\frac{1}{2\pi\imath \lambda}  (\frac{d}{dt} e^{2\pi \imath \lambda t})  f \circ \phi^S_t  dt 
\\ &=  \frac{1}{2\pi\imath \lambda} \left( e^{2\pi \imath \lambda {\mathcal T}}  f \circ \phi^S_{\mathcal T} - f - 
\int_0^{\mathcal T}   e^{2\pi \imath \lambda t} Sf \circ \phi^S_t  dt                    \right) 
\end{aligned}
$$
hence for all $\lambda \not=0$, we have
$$
\begin{aligned}
\vert \int_0^{\mathcal T}   e^{2\pi \lambda t} f \circ \phi^S_t dt   \vert_{L^2_h(M)}  &\leq 
 \frac{1}{\pi  \lambda}  \vert f \vert_{L^2_h(M)} \\ &+  \frac{1}{2\pi  \lambda}  \vert \int_0^{\mathcal T}   e^{2\pi \imath \lambda t} Sf \circ \phi^S_t  dt \vert_{L^2_h(M)}
\end{aligned}
$$
By iterating the integration by parts (for $\vert \lambda \vert \geq 1$) we derive the bound
$$
\begin{aligned}
\vert \int_0^{\mathcal T}   e^{2\pi \lambda t} f \circ \phi^S_t dt   \vert_{L^2_h(M)}  &\leq 
\sum_{j=0}^{k-1}  \frac{1}{\vert \pi  \lambda\vert^{j+1}}  \vert S^j f \vert_{L^2_h(M)} \\ &+  \frac{1}{\vert 2\pi  \lambda\vert^k}  \vert \int_0^{\mathcal T}   e^{2\pi \imath \lambda t} S^k f \circ \phi^S_t  dt \vert_{L^2_h(M)}
\end{aligned}
$$
It follows that under the assumption that $S^j f \in L^2_h(M)$, for all $j\in \{0, \dots, N_\kappa\}$, and that
$f$ and $S^{N_\kappa} f \in H^1_h(M)$, the hypothesis of Lemma~\ref{lemma:Eff_WM}, for the part concerning
the bound on twisted integrals, are a consequence of Theorem~\ref{thm:twist_integral_bound}. The
hypothesis  of Lemma~\ref{lemma:Eff_WM}, for the part concerning the bounds on ergodic integrals
($\lambda =0$), follows from Theorem~\ref{thm:AtF08}  for functions of zero average. The corollary
is therefore proved.

\end{proof}

\section{Spectral dimension and effective weak mixing}
\label{sec:spectral_dim}

The content of this section is standard. We reproduce it here for the convenience of the reader.
We recall that for any measure $\sigma$ on $\R$ we can defined the lower and upper lower local
dimension, $\underline{d}_\sigma(\lambda)$ and $\overline{d}_\sigma(\lambda) $, at $\lambda\in \R$,
as follows:
$$
\begin{aligned}
&\underline{d}_\sigma(\lambda):= \underline{\lim}_{r\to 0^+} \frac{ \log \sigma ([\lambda -r, \lambda +r])}{\log r}\,, \\
& \overline{d}_\sigma(\lambda):= \overline{\lim}_{r\to 0^+} \frac{ \log \sigma ([\lambda -r, \lambda +r])}{\log r}\,.
\end{aligned}
$$
Let $\sigma_f$ denote the spectral measure of a function $f\in L^2(M, \mu)$ for
a flow $(\phi_\R)$ which preserves the probability measure $\mu$ on $M$. The measure $\sigma_f$ is a complex measure on $\R$ of finite total mass equal to
$\Vert f \Vert^2$.  Let $\underline{d}_f(\lambda)$ and $\overline{d}_f(\lambda)$ denote the
lower and upper local dimensions of the measure $\sigma_f$ at $\lambda \in \R$. 

\begin{lemma} \label{lemma:spectral_1} Let us assume that given $\lambda \in \R$ and a function $f\in L^2(M, \mu)$ there exist constants $C_f(\lambda)>0$ and  $0\leq \alpha_- \leq \alpha_+<1$ such that, for all $\mathcal T\geq {\mathcal T}_0>0$, 
$$
C_f(\lambda)^{-1} \mathcal T^{1-\alpha_+} \leq \Vert  \int_0^{\mathcal T} e^{-2\pi \imath \lambda t} f\circ \phi_t \,dt  \Vert_{L^2(M,\mu)}  \leq  C_f(\lambda) {\mathcal T}^{1-\alpha_-}  \,.
$$ 
Then there exist constants $C'_f(\lambda)$ and $r_0>0$ such that for all $r\in (0, r_0)$ we have
$$
C'_f(\lambda)^{-1}  r ^{\frac{2\alpha_+}{1-\alpha_+} }  \leq \sigma_f([\lambda -r, \lambda +r]  )  \leq     8 C_f(\lambda)  r ^{2\alpha_-}\,.
$$
In particular we derive
$$
 2 \alpha_- \leq \underline{d}_f(\lambda) \leq \overline{d}_f(\lambda)   \leq    \frac{2\alpha_+}{1-\alpha_+}\,.
$$
\end{lemma} 
\begin{proof}  By spectral theory we have
\begin{equation}
\label{eq:spectral}
\begin{aligned}
\Vert  &\int_0^{\mathcal T} e^{-2\pi \imath \lambda t} f\circ \phi_t(\cdot) \,dt  \Vert^2_{L^2(M,\mu)} =
\Vert  \int_0^{\mathcal T} e^{-2\pi \imath (\lambda -\xi)  t }dt  \Vert^2_{L^2(\R,d\sigma_f(\xi))} \\ & =
\int_\R \vert  \frac{ e^{-2\pi \imath (\lambda -\xi) {\mathcal T}} -1}{ 2\pi \imath (\lambda -\xi)} 
\vert^2 d\sigma_f(\xi)  = {\mathcal T}^2 \int_\R \vert  \frac{ e^{-2\pi \imath (\lambda -\xi) {\mathcal T}} -1}{ 2\pi \imath (\lambda -\xi){\mathcal T}} \vert^2 d\sigma_f(\xi)  \,.
\end{aligned}
\end{equation}
Let $\chi:\R \to \R^+$ denote the function
$$
\chi (x) :=  \left\vert \frac{ e^{-2\pi \imath x} -1}{ 2\pi \imath x} \right\vert^2 \,.
$$
Let $c>0$ be the strictly positive constant defined as
$$
c:= \min_{x\in [-1/2, 1/2] }  \chi (x) \,\geq\, \frac{1}{2} \,.
$$
It follows that 
$$
c {\mathcal T}^2 \sigma_f( [\lambda-\frac{1}{2{\mathcal T}}, \lambda + \frac{1}{2{\mathcal T}}]) \leq  \Vert  \int_0^{\mathcal T} e^{-2\pi \imath \lambda t} f\circ \phi_t \,dt  \Vert^2_{L^2(M,\mu)}\,,
$$
which is equivalent to the estimate
$$
 \sigma_f( [\lambda-r, \lambda + r]) \leq  4 c^{-1}  r^2 \Vert  \int_0^{\frac{1}{2r}} e^{-2\pi \imath \lambda t} f\circ \phi_t \,dt  \Vert^2_{L^2(M,\mu)}\,,
$$
Under the hypothesis we have
$$
 \sigma_f( [\lambda-r, \lambda + r]) \leq  4 c^{-1}  C_f(\lambda) r^2  r^{-2(1-\alpha_-)} =
 4 c^{-1}  C_f(\lambda)  r^{2 \alpha_-}\,.
$$
For the lower bound we write
$$
\begin{aligned}
 \int_\R &\vert  \frac{ e^{-2\pi \imath (\lambda -\xi) {\mathcal T}} -1}{ 2\pi \imath (\lambda -\xi){\mathcal T}} 
\vert^2 d\sigma_f(\xi)=  \int_{\vert \lambda -\xi \vert \leq r}  \vert \frac{ e^{-2\pi \imath (\lambda -\xi) {\mathcal T}} -1}{ 2\pi \imath (\lambda -\xi){\mathcal T}} \vert^2 d\sigma_f(\xi)  \\ &+   \int_{\vert \lambda -\xi \vert \geq r}   \vert \frac{ e^{-2\pi \imath (\lambda -\xi) {\mathcal T}} -1}{ 2\pi \imath (\lambda -\xi){\mathcal T}} \vert^2 d\sigma_f(\xi)\,.
\end{aligned}
$$
We have the following bounds: there exists $C>0$ such that
$$
\begin{aligned}
\int_{\vert \lambda -\xi \vert \leq r}  \vert \frac{ e^{-2\pi \imath (\lambda -\xi) {\mathcal T}} -1}{ 2\pi \imath (\lambda -\xi){\mathcal T}} \vert^2 d\sigma_f(\xi) &\leq   C \sigma_f([\lambda -r , \lambda + r ])\,, \\
\int_{\vert \lambda -\xi \vert \geq r}  \vert \frac{ e^{-2\pi \imath (\lambda -\xi) {\mathcal T}} -1}{ 2\pi \imath (\lambda -\xi){\mathcal T}} \vert^2 d\sigma_f(\xi) &\leq   \frac{C \Vert f \Vert^2}{r^2 {\mathcal T}^2} \,,
\end{aligned}
$$
hence we derive the lower bound
$$
\sigma_f([\lambda -r , \lambda + r ]) \geq \frac{C^{-1}}{{\mathcal T}^2}  \Vert  \int_0^{\mathcal T} e^{-2\pi \imath \lambda t} f\circ \phi_t \,dt  \Vert^2_{L^2(M,\mu)}  - \frac{\Vert f \Vert^2}{r^2 {\mathcal T}^2} \,.
$$
Finally, under the assumption that we have a lower bound
$$
 \Vert  \int_0^{\mathcal T} e^{-2\pi \imath \lambda t} f\circ \phi_t \,dt  \Vert_{L^2(M,\mu)}  \geq  C_f(\lambda)^{-1} {\mathcal T}^{1-\alpha_+}\,,
$$
we derive that there exists $C'_f(\lambda)>0$ such that 
$$
\sigma_f([\lambda -r , \lambda + r ]) \geq  C'_f (\lambda) {\mathcal T}^{-2\alpha_+}  - \frac{ \Vert f \Vert^2 }{r^2 {\mathcal T}^2} 
$$
then, by taking ${\mathcal T} = \left(\frac{2\Vert f \Vert}{[C'_f(\lambda)]^{1/2} r}\right)^{\frac{1}{1-\alpha_+}}$,  there exists a constant $C_f^{(\alpha)}>0$ such that
$$
\sigma_f([\lambda -r , \lambda + r ]) \geq  C_f^{(\alpha)} r^{\frac{2 \alpha_+}{1-\alpha_+}} \,.
$$
\end{proof} 

\begin{lemma} Let us assume that given $\lambda \in \R$ and a function $f\in L^2(M, \mu)$ there exist constants $C_f(\lambda)>0$, $r_0>0$ and  $0\leq \beta_- \leq \beta_+ \leq 1$ such that, for all $0<r\leq r_0$, 
$$
C_f(\lambda)^{-1}  r ^{2\beta_+ }  \leq \sigma_f([\lambda -r, \lambda +r]  )  \leq     C_f(\lambda)  r ^{2\beta_-}\,.
$$
Then there exist constants $C'_f(\lambda)$ and $ \mathcal T_0>0$ such that for all ${\mathcal T}\geq {\mathcal T}_0\geq e$ we have
$$
C'_f(\lambda)^{-1} {\mathcal T}^{1-\beta_+} \leq \Vert  \int_0^{\mathcal T} e^{-2\pi \imath \lambda t} f\circ \phi_t \,dt  \Vert_{L^2(M,\mu)}  \leq  C'_f(\lambda) \max\{{\mathcal T}^{1-\beta_-} , (\log {\mathcal T})^{1/2}\} \,.
$$ 
\end{lemma} 
\begin{proof}  For fixed $\lambda \in \R$ and ${\mathcal T}>0$ and for all $n\in \N$ we let $I_n (\lambda)\subset \R$ denote the set defined as follows:
$$
I_n :=  \{ \xi \in \R :   {\mathcal T}\vert \xi -\lambda \vert  \leq 2^{n-2} \}.
$$
By formula \eqref{eq:spectral}, we then write (for  $m>1$ to be chosen later) 
$$
\begin{aligned}
\Vert  &\int_0^{\mathcal T} e^{-2\pi \imath \lambda t} f\circ \phi_t \, dt  \Vert^2_{L^2(M,\mu)}  
\leq {\mathcal T}^2   \int_{I_0} \vert  \frac{ e^{-2\pi \imath (\lambda -\xi) {\mathcal T}} -1}{ 2\pi \imath (\lambda -\xi){\mathcal T}} \vert^2 d\sigma_f(\xi)  \\    &+ {\mathcal T}^2 \sum_{n= 1}^m \int_{I_n\setminus I_{n-1}} \vert  \frac{ e^{-2\pi \imath (\lambda -\xi) {\mathcal T}} -1}{ 2\pi \imath (\lambda -\xi){\mathcal T}} \vert^2 d\sigma_f(\xi)  + {\mathcal T}^2 2^{-2(m-2)} \sigma_f (\R\setminus I_m)\,.
\end{aligned}
$$
Let us denote below by $C$, $C_f(\lambda)$, $C'_f(\lambda)$ positive constants, independent of $\mathcal T \in \R$, but possibly dependent on the function $f\in L^2(M, \mu)$ or on the phase constant $\lambda \in \R$,  which may vary from line to line throughout the argument.

\smallskip
Since there exists a (universal) constant $C>0$ such that
$$
C^{-1} \leq  \vert  \frac{ e^{-2\pi \imath (\lambda -\xi) {\mathcal T}} -1}{ 2\pi \imath (\lambda -\xi){\mathcal T}} \vert \leq C\,,
\quad \text{ for all } \xi \in I_0 \,,
$$
it follows by the assumptions that there exists a constant $C_f(\lambda)>0$ such that
$$
C_f(\lambda)^{-1} {\mathcal T}^{-2\beta_+} \leq    \int_{I_0} \vert  \frac{ e^{-2\pi \imath (\lambda -\xi) {\mathcal T}} -1}{ 2\pi \imath (\lambda -\xi){\mathcal T}} \vert^2 d\sigma_f(\xi)  \leq C_f(\lambda) {\mathcal T}^{-2\beta_-} \,.
$$
Then from the hypothesis and the inequality
$$
\vert  \frac{ e^{-2\pi \imath (\lambda -\xi) {\mathcal T}} -1}{ 2\pi \imath (\lambda -\xi){\mathcal T}} \vert  \leq \frac{2^{-(n-3)} }{\pi } \,, \quad \text{ for all } \xi \not \in I_{n-1}\,,
$$
it follows that there exist  constant $C_f(\lambda), C'_f(\lambda)>0$ such that
$$
\begin{aligned}
  \int_{I_n\setminus I_{n-1}} \vert  \frac{ e^{-2\pi \imath (\lambda -\xi) {\mathcal T}} -1}{ 2\pi \imath (\lambda -\xi){\mathcal T}} \vert^2 d\sigma_f(\xi)  &\leq    C_f(\lambda) \frac{  2^{-2(n-3)} }{\pi^2  }  
   \left (\frac{2^{n-2} }{ {\mathcal T}} \right)^{2\beta_-}  \\ &=   \frac{C'_f(\lambda) }{16^{\beta_-} }  {\mathcal T}^{-2\beta_-}  \, 2^{-(2-2\beta_-)n}   \,.
\end{aligned}
$$
It then follows that whenever $\beta_- <1$ there exists a constant $C^{(\beta_-)}_f(\lambda)>0$ such that, for all ${\mathcal T} >{\mathcal T}_0$,  we have
$$
 \sum_{n= 1}^{+\infty}  \int_{I_n\setminus I_{n-1}} \vert  \frac{ e^{-2\pi \imath (\lambda -\xi) {\mathcal T}} -1}{ 2\pi \imath (\lambda -\xi){\mathcal T}} \vert^2 d\sigma_f(\xi) \leq  C^{(\beta_-)}_f(\lambda)  {\mathcal T}^{-2\beta_-} \,,
 $$
hence the argument is completed in this case.  For $\beta_-=1$ we have 
$$
 \sum_{n= 1}^m \int_{I_n\setminus I_{n-1}} \vert  \frac{ e^{-2\pi \imath (\lambda -\xi) {\mathcal T}} -1}{ 2\pi \imath (\lambda -\xi){\mathcal T}} \vert^2 d\sigma_f(\xi) \leq  \frac{ C_f(\lambda) }{4}  {\mathcal T}^{-2}  m\,,
$$
hence, by taking $m= [ \frac{\log {\mathcal T}}{\log 2}]$ we derive that
$$
 \sum_{n= 1}^m \int_{I_n\setminus I_{n-1}} \vert  \frac{ e^{-2\pi \imath (\lambda -\xi) {\mathcal T}} -1}{ 2\pi \imath (\lambda -\xi){\mathcal T}} \vert^2 d\sigma_f(\xi)  + 2^{-2(m-2)} \Vert f \Vert^2 \leq 
 C'_f(\lambda) {\mathcal T}^{-2 } \log {\mathcal T} \,,
$$
thereby completing the argument in all cases.
\end{proof}

We conclude the section with a general lemma on effective weak mixing.

\begin{lemma} 
\label{lemma:Eff_WM}
Let $\phi_\R$ be a flow on a probability space $(M,\mu)$ and let  $f\in L^2(M, \mu)$. Let us assume that there exists $\alpha$, $\beta>0$ such that  there exists a constant $I(f)>1$ such that, for all  $\lambda \in \R$ and for all ${\mathcal T}>1$,  we have
$$
\begin{aligned}
\Vert &\int_0^{\mathcal T}   e^{2\pi \imath \lambda t}  f \circ \phi_t \,dt   \Vert_{L^2(M,\mu)}  \leq I(f)
\vert \lambda \vert^{-\beta}  {\mathcal T}^{1-\alpha} \,, \\ 
&\text{ and } \quad \Vert \int_0^{\mathcal T}    f \circ \phi_t \,dt   \Vert_{L^2(M,\mu)}  \leq I(f)   {\mathcal T}^{1-\alpha}\,.
\end{aligned}
$$
Then there exist constants $\alpha' := \alpha' (\alpha, \beta)>0$ and  $C>0$ such that
the following effective weak mixing  bound holds. For all $g\in L^2(M, \mu)$ and for ${\mathcal T}>1$ we have
$$
\frac{1}{{\mathcal T}} \int_0^{\mathcal T}   \vert \langle f\circ \phi_t, g\rangle_{L^2(M, \mu)}\vert^2 dt  \leq   C I(f) \Vert f \Vert_{L^2(M, \mu)}  \Vert g \Vert^2_{L^2(M, \mu)} {\mathcal T}^{-\alpha'} \,.
$$
\end{lemma} 
\begin{proof}
Let $\sigma_{f,g}$ denote the spectral measure of the pair $f, g\in L^2(M,\mu)$. By definition, the measure
$\sigma_{f,g}$ is the Fourier transform of the absolutely continuous measure $\langle f \circ \phi_t , g\rangle dt$.
By properties of the Fourier transform, we can write
$$
\begin{aligned}
\int_0^{\mathcal T} \vert \langle f \circ \phi_t , g\rangle \vert^2  dt &= \int_\R \chi_{[0,{\mathcal T}]} \langle f \circ \phi_t , g\rangle 
\overline{ \langle f \circ \phi_t , g\rangle}  dt  \\ &= \int_\R \left ( \int_0^{\mathcal T} 
e^{2\pi \imath \lambda t} \langle f \circ \phi_t , g\rangle  dt \right)  d \bar \sigma_{f,g}(\lambda)
\end{aligned}
$$
Let $\eta>0$ such that $\beta \eta <\alpha$. Since by H\"older inequality
$$
\vert \int_0^{\mathcal T} 
e^{2\pi \imath \lambda t} \langle f \circ \phi_t , g\rangle  dt \vert  \leq  \Vert  g  \Vert_{L^2(M,\mu)} \times \Vert \int_0^{\mathcal T}   e^{2\pi \imath \lambda t}  f \circ \phi_t \,dt   \Vert_{L^2(M,\mu)}    \,,
$$
it follows that, for $\vert \lambda \vert \geq \mathcal {\mathcal T}^{-\eta}$, we have
$$
\begin{aligned}
 \int_{\vert \lambda\vert \geq {\mathcal T}^{-\eta}} & \left ( \int_0^{\mathcal T} 
e^{2\pi \imath \lambda t} \langle f \circ \phi_t , g\rangle  dt \right)  d \bar \sigma_{f,g}(\lambda)
\\ &\leq  I(f)  \mathcal T^{1-\alpha + \beta \eta}  \Vert  f  \Vert_{L^2(M,\mu)} \Vert  g  \Vert^2_{L^2(M,\mu)} \,.
\end{aligned}
$$
Finally, we claim that  we have
\begin{equation}
\label{eq:near_zero}
 \sigma_{f,g}(-{\mathcal T}^{-\eta}, {\mathcal T}^{-\eta}) \leq  [8 I(f)]^{1/2} \Vert g \Vert_{L^2(M,\mu)}  {\mathcal T}^{-\alpha \eta}\,.
\end{equation}
In fact, by the hypothesis that 
$$
\Vert \int_0^{\mathcal T}    f \circ \phi_t \,dt   \Vert_{L^2(M,\mu)}  \leq I(f)   {\mathcal T}^{1-\alpha}\,,
$$
and by Lemma~\ref{lemma:spectral_1} (with $\lambda=0$ and $r={\mathcal T}^{-\eta}$), we derive the bound
$$
\sigma_f (-{\mathcal T}^{-\eta}, {\mathcal T}^{-\eta}) \leq 8 I(f) {\mathcal T}^{-2 \alpha \eta}\,,
$$
then the claim follows from the general inequalities for spectral measures:
$$
\sigma_{f,g} \leq  \sigma_f^{1/2} 
\sigma_g^{1/2} \leq   \Vert g \Vert_{L^2(M,\mu)}  \sigma_f^{1/2} \,.
$$
By the bound in formula~\eqref{eq:near_zero}  we then conclude that
$$
\begin{aligned}
 \int_{\vert \lambda\vert \leq {\mathcal T}^{-\eta}}  \left ( \int_0^{\mathcal T} 
e^{2\pi \imath \lambda t} \langle f \circ \phi_t , g\rangle  dt \right)&  d \bar \sigma_{f,g}(\lambda)
 \\ &\leq [8 I(f )]^{1/2} \Vert f \Vert_{L^2(M,\mu)} \Vert g \Vert^2_{L^2(M,\mu)}  {\mathcal T}^{1-\alpha \eta} \,.
\end{aligned}
$$
The statement of the lemma follows.

\end{proof}


\begin{thebibliography}{20}

\bibitem{Ad}[Ad] R. A. Adams, \emph{Sobolev Spaces}, Pure and Applied Mathematics, v. 65, Academic Press, New York-London, 1975.
\bibitem[AtF08]{AtF08} J.~Athreya and G.~Forni, Deviation of ergodic averages for rational polygonal billiards, {\it Duke Math. J.} 
{\bf 144} (2) (2008), 285--319.
\bibitem[AD16]{AD16} A.~Avila and V.~Delecroix, Weak mixing directions in non-arithmetic Veech surfaces, {\it J. Amer. Math. Soc.}  \textbf{29} (2016), 1167--1208.
\bibitem[AvF07]{AvF07} A.~Avila and G.~Forni,  Weak mixing for interval exchange transformations and translation 
flows, {\it Ann. of Math.}   \textbf{165}(2) (2007), 637--664 . 
\bibitem[AL]{AL} A.~Avila and M.~Leguil,  Weak mixing properties of interval exchange transformations and translation flows, {\it Bull. Soc. Math. France} \textbf{146} (2) (2018), 391--426.
\bibitem[AV07]{AV07} A.~Avila and M.~Viana, Simplicity of Lyapunov spectra: proof of the Zorich--Kontsevich conjecture, {\it Acta Math.}  {\bf 198} (1) (2007), 1--56.
\bibitem[BS14]{BS14} A.~I.~Bufetov and B.~Solomyak, On the modulus of continuity for spectral measures in substitution dynamics,  {\it Adv. Math.} {\bf 260} (2014), 84--129.
\bibitem[BS18a]{BS18a} \bysame, The H\"older property for the spectrum of translation flows in genus two, {\it Israel J. Math.}  {\bf 223} (1) (2018), 205--259.
\bibitem[BS18b]{BS18b} \bysame, On ergodic averages for parabolic product flows, {\it Bull. Soc. Math. France}
\textbf{146} (4) (2018), 675--690.
\bibitem[BS18c]{BS18c} \bysame, A spectral cocycle for substitution systems and translation flows, preprint, arXiv:1802.04783.
\bibitem[BS19]{BS19} \bysame, H\"older regularity for the spectrum of translation flows, preprint, arXiv:1908.09347.
\bibitem[C69]{C69} R.~V.~Chacon, Weakly mixing transformations which are not strongly mixing, {\it Proc. Amer. Math. Soc.} \textbf{22} (1969), 559--562.
\bibitem[Fi17]{Fi17} S.~Filip,  Zero Lyapunov exponents and monodromy of the Kontsevich--Zorich cocycle, 
{\it Duke Math. J.} \textbf{166} (2017) (4), 657--706. 
\bibitem[FlaFo06]{FlaFo06} L.~Flaminio and G. Forni, Equidistribution of nilflows and applications to theta sums, 
{\it Ergodic Theory Dynam. Systems}~\textbf{26} (2) (2006), 409--433.
\bibitem[FlaFo14]{FlaFo14} \bysame, On effective equidistribution for higher step nilflows, preprint,  arXiv:1407.3640v1.
\bibitem[FFT16]{FFT16}  L.~Flaminio, G.~Forni and J.~Tanis, Effective equidistribution of twisted horocycle flows and horocycle maps, {\it Geometric and Functional Analysis} \textbf{26 (5)},1359--1448.
\bibitem[F97]{F97} G.~Forni, Solutions of the Cohomological Equation for Area-Preserving Flows on Compact Surfaces of Higher Genus,
 {\it Ann. of Math.} \textbf{146}(2) (1997), 295--344. 
 \bibitem[F02]{F02}
\bysame, Deviation of ergodic averages for area-preserving flows on
  surfaces of higher genus, {\it Ann. of Math.} (2) \textbf{155} (2002), no.~1,
  1--103.
 \bibitem[F07]{F07} \bysame, Sobolev regularity of solutions of the cohomological equation,  {\it Erg. Th. Dynam. Sys.},  online at https://doi.org/10.1017/etds.2019.108
 (arXiv:0707.0940v2).
 \bibitem[F11]{F11}  \bysame, A geometric criterion for the nonuniform hyperbolicity of the Kontsevich--Zorich cocycle, {\it J. Mod. Dynam.}   \textbf{5} (2) (2011), 355--395.
 \bibitem[FG]{FG}  G.~Forni and W.~Goldman,   Mixing Flows on Moduli Spaces of Flat Bundles over Surfaces,
 in {\it Geometry and Physics: Volume II},  A Festschrift in honour of Nigel Hitchin. Edited by A.~Dancer, J.~ E.~Andersen, and O.~Garc\'ia-Prada.  Oxford University Press, October 2018.
  \bibitem[G84]{G84}  W.~Goldman,  The symplectic nature of fundamental groups of surfaces, {\it Advances in Mathematics} {\bf 54} (2) (1984), 200--225.
  \bibitem[GX08]{GX08} W.~Goldman and E.~Xia, Rank One Higgs Bundles and Representations of Fundamental Groups of Riemann Surfaces,  Memoirs of the American Mathematical Society, 
2008.
 \bibitem[GT12]{GT12} B.~Green and T.~Tao, The quantitative behaviour of polynomial orbits on nilmanifolds, {\it Ann. of Math.} \textbf{175} (2012), 465--540.
\bibitem[GK88]{GK88} E.~Gutkin and A.~B.~Katok, Weakly mixing billiards, in {\it Holomorphic dynamics} (Mexico, 1986), Lecture Notes in Math. 1345, {\bf 163--176}, Springer-Verlag, New York, 1988.
\bibitem[Ka80]{Ka80} A.~B.~Katok, Interval exchange transformations and some special flows are not
mixing, {\it  Israel J. Math.} {\bf  35} (1980), 301--310.
\bibitem[KS67]{KS67} A.~B.~Katok and A.~M.~Stepin, Approximations in ergodic theory, {\it Uspehi Mat. Nauk} {\bf 22} (1967), 81--106.
\bibitem[KMS86]{KMS86} S.~Kerckhoff, H.~Masur and J.~Smillie,
Ergodicity of Billiard Flows and Abelian Differentials, {\it Ann. of Math.} {\bf 124}(1986), 293--311.
\bibitem[Kn98]{Kn98}  O.~Knill, Singular Continuous Spectrum and Quantitative Rate of Mixing, {\it Disc. Cont. Dynam. Sys.} {\bf 4} (1)  (1998),   33--42.
\bibitem[Ko97]{Ko97} M.~Kontsevich, Lyapunov exponents and Hodge theory, in {\it The Mathematical Beauty of Physics} (Saclay, 1996), Adv. Ser. Math. Phys. {\bf 24}, 318--332, World Sci. Publ.,
River Edge, NJ, 1997.
\bibitem[Lu98]{Lu98} I.~Lucien, M\'elange faible topologique des flots sur les surfaces, 
{\it Ergodic Theory Dynam. Systems} {\bf 18} (1998), 963--984.
\bibitem[Ma82]{Ma82} H.~Masur, Interval exchange transformations
and measured foliations, {\it Ann. of Math.} {\bf 115}(1982), 168-200.
\bibitem[MS91]{MS91} H.~Masur and J.~Smillie, Hausdorff dimension of 
sets of non-ergodic measured foliations, {\it Ann. of Math.} {\bf 134},
(1991), 455-543.
\bibitem[Ne59]{Ne59} E.~Nelson, Analytic Vectors, {\it Ann. of Math.} {\bf 70} (1959),
572--615.
\bibitem[NR97]{NR97} A.~Nogueira and D.~Rudolph, Topological weak-mixing of interval exchange maps, {\it Ergodic Theory Dynam. Systems} {\bf 17} (1997), 1183--1209.
\bibitem[Pe83]{Pe83} K.~Petersen, {\it Ergodic Theory}, Cambridge Univ. Press, Cambridge, 1983.
\bibitem[TV15]{TV15} J.~Tanis \& P.~Vishe, Uniform bounds for period integrals and
    sparse equidistribution, {\it International Mathematics Research
    Notices} \textbf{2015} (24) (2015), 13728--13756 (https://doi.org/10.1093/imrn/rnv115).
\bibitem[Ve82]{Ve82} W.~A.~Veech, Gauss measures for transformations on the
space of interval exchange maps, {\it Ann. of Math.} \textbf {115}(1982),
201--242.
\bibitem[Ve84]{Ve84} \bysame, The metric theory of interval exchange transformations. I. Generic spectral properties, {\it Amer. J. Math.} {\bf 106} (1984), 1331--1359.
\bibitem[V10]{V10} A.~Venkatesh.  Sparse equidistribution problems, period bounds and subconvexity,  
{\it Ann. of Math.} \textbf{172} (2010), 989--1094.
\bibitem[We80]{We80} R.~O.~Wells, {\it Differential Analysis on Complex Manifolds},  
Springer-Verlag, New York, 1980.
\bibitem[Wri15]{Wri15} A. Wright, Cylinder deformations in orbit closures of translation surfaces, {\it Geom. Topol.}  \textbf{19} (1) (2015), 413--438.
\bibitem[Zo97]{Zo97} A.~Zorich, Deviation for interval exchange transformations, {\it Ergodic Theory Dynam. Systems} {\bf 17} (1997), 1477--1499.
\end{thebibliography}
\end{document}